\documentclass[smallcondensed]{svjour3}  
\smartqed      
\usepackage{graphicx}

\usepackage{amsmath}
\usepackage{amsfonts}
\usepackage{amssymb}
\usepackage{graphicx}
\usepackage{enumitem}
\usepackage{lscape}
\usepackage{longtable}
\usepackage{rotating}
\usepackage{multirow}
\usepackage[ruled,vlined,noline,linesnumbered]{algorithm2e}
\usepackage{color}
\usepackage{url}
\usepackage{subfigure}
\usepackage{rotating}
\usepackage{fullpage}
\usepackage[colorlinks=true]{hyperref}
\usepackage[normalem]{ulem}
\usepackage{setspace}
\AtBeginDocument{%
  \addtolength\abovedisplayskip{-0.5\baselineskip}%
  \addtolength\belowdisplayskip{-0.5\baselineskip}%
}

\newcommand{\R}{{\mathbb R}}
\newcommand{\N}{{\mathbb N}}
\newcommand{\EE}{{\mathbb E}}
\DeclareMathOperator{\argmin}{argmin}

\newcommand{\cC}{{\mathcal C}}
\newcommand{\cE}{{\mathcal E}}
\newcommand{\cH}{{\mathcal H}}
\newcommand{\cO}{{\mathcal O}}

\newcommand{\cD}{{\mathcal D}}
\newcommand{\cF}{{\mathcal F}}
\newcommand{\cG}{{\mathcal G}}
\newcommand{\cW}{{\mathcal W}}

\newcommand{\demi}{\frac{1}{2}}
\newcommand{\ie}{{\it i.e.}}

\newcommand{\Rb}{\R\cup\{+\infty\}}
\newcommand{\eps}{\varepsilon}

\newcommand{\eqdef}{:=}

\newcommand{\dotp}[2]{\left\langle #1,\,#2 \right\rangle}
\newcommand{\norm}[1]{\left\|{#1}\right\|}
\newcommand{\pa}[1]{\left({#1}\right)}
\newcommand{\brac}[1]{\left[{#1}\right]}

\newcommand{\myparagraph}[1]{\paragraph{{\bf #1}}}

\DeclareMathOperator{\dom}{dom}

\newcommand{\abs}[1]{\left |{#1}\right |}
\newcommand{\bpa}[1]{\Big({#1}\Big)}
\newcommand{\bra}[1]{\left\lbrace{#1}\right\rbrace}
\newcommand{\setcond}[2]{\bra{#1:~#2}}
\newcommand{\DS}[1]{\displaystyle{#1}}

\newcommand{\qandq}{ \enskip \text{and} \enskip }

\if
{
\usepackage[pageref]{backref}
\renewcommand*{\backrefalt}[4]{%
\ifcase #1 %
(Not cited)%
\or
(Cited on p.~#2)%
\else
(Cited on pp.~#2)%
\fi
}

}
\fi

\begin{document}

\title{On the effect of perturbations in first-order optimization methods with inertia and Hessian driven damping}
\titlerunning{Perturbations of inertial systems with Hessian driven damping}

\author{Hedy Attouch \and Jalal Fadili \and  Vyacheslav Kungurtsev}
\authorrunning{H. Attouch, J. Fadili, V. Kungurtsev} % if too long for running head

\institute{
H. Attouch \at IMAG, Univ. Montpellier, CNRS, Montpellier, France\\
\email{hedy.attouch@umontpellier.fr} 
\and J. Fadili
\at{Normandie Univ, ENSICAEN, CNRS, GREYC, Caen, France}\\
\email{Jalal.Fadili@greyc.ensicaen.fr} 
\and V. Kungurtsev
\at  Department of Computer Science and Engineering, Czech Technical University, Prague.\\\email{vyacheslav.kungurtsev@fel.cvut.cz}
}

%\title{{On the effect of perturbation errors in first-order optimization methods with inertia and Hessian driven damping}}
%
%
%\author{Hedy Attouch}\address{IMAG, Univ. Montpellier, CNRS, Montpellier, France. \texttt{hedy.attouch@umontpellier.fr}}
%\author{Jalal Fadili}\address{Normandie Univ, ENSICAEN, CNRS, GREYC, Caen, France. \texttt{Jalal.Fadili@greyc.ensicaen.fr}}
%\author{Vyacheslav Kungurtsev}\address{Department of Computer Science and Engineering, Czech Technical University, Prague. \texttt{vyacheslav.kungurtsev@fel.cvut.cz}}

\date{}

\maketitle
 
\begin{abstract}
Second-order continuous-time dissipative dynamical systems with viscous and Hessian driven damping have inspired effective first-order algorithms for solving convex optimization problems. While preserving the fast convergence properties of the Nesterov-type acceleration, the  Hessian driven damping makes it possible to significantly attenuate the oscillations. To study the stability of these algorithms with respect to perturbations, we analyze the behaviour of the corresponding continuous systems when the gradient computation is subject to exogenous additive errors. We provide a quantitative analysis of the asymptotic behaviour of two types of systems, those with implicit and explicit Hessian driven damping. We consider convex, strongly convex, and non-smooth objective functions defined on a real Hilbert space and show that, depending on the formulation, different integrability conditions on the perturbations are sufficient to maintain the  convergence rates of the systems. We  highlight the differences between the implicit and explicit Hessian damping, and in particular point out  that the assumptions on the objective and perturbations needed in the implicit case are more stringent than in the explicit case.
\end{abstract}

%\begin{resume}

%\end{resume}

%\subjclass{37N40, 46N10, 49M30, 65B99, 65K05, 65K10, 90B50, 90C25.}



\section{Introduction}
The continuous-time dynamic perspective of optimization algorithms, which can be viewed as temporal discretization schemes thereof, offers an insightful and powerful framework for the study of the behaviour of these algorithms. In this paper, we study inertial systems involving both viscous and Hessian-driven damping, where the first-order gradient information is only accessible up to some {\emph{exogenous additive error}}. 

\subsection{Problem statement}
Throughout the paper, we make the following standing assumptions:
\[
\boxed{
\text{
$f$ is a convex function on a real Hilbert space $\cH$, and $S \eqdef \argmin_\cH f\neq \emptyset$.\hspace{1cm}}
} 
\]

\noindent We will study perturbed versions of two second-order ordinary differential equations (ODE). {They differ from each other in that the Hessian driven damping appears  explicitly in one and implicitly in the other}.

\subsubsection{Explicit Hessian}
The first system we look at, which was proposed in \cite{attouch2019first} (see also \cite{APR1}), takes the form
\begin{equation}\tag{ISEHD}\label{eq:origode}
\ddot{x}(t)+\gamma(t)\dot{x}(t)+\beta(t)\frac{d}{dt}(\nabla f(x(t)))+b(t)\nabla f(x(t)) = 0,
\end{equation}
where $f \in \cC^1(\cH)$,  $\gamma,\beta,b:[t_0,+\infty[ \to \R_+$ are continuous functions, and $t_0>0$ is the initial time. The coefficients $(\gamma,\beta,b)$ have a physical interpretation corresponding to natural phenomena:
\begin{enumerate}[label=$\bullet$]
\item $\gamma(t)$ is the viscous damping coefficient,
\item $\beta(t)$ is the Hessian-driven damping coefficient (which will be made clear),
\item $b(t)$ is the  time scaling coefficient (see \cite{ACR-SIOPT}).
\end{enumerate}

\noindent We term the above ODE an Inertial System with {\emph{Explicit}} Hessian Damping (ISEHD for short), since 
\[
\frac{d}{dt}(\nabla f(x(t))) = \nabla^2 f(x(t)) \dot x(t) ,
\] 
when $f$ is of class $\cC^2(\cH)$. Throughout the paper, we consider \eqref{eq:origode} with the particular choice of parameters 

\begin{center}
$\gamma(t) = \displaystyle{\frac{\alpha}{t}}$, \; $\alpha \geq 0$,\; $\beta(t) \equiv \beta > 0$ and $b(t) \equiv 1$.
\end{center}

This choice of the viscous damping parameter $\gamma(t) = \frac{\alpha}{t}$ is justified by its direct link with the accelerated gradient method of Nesterov \cite{Nest1,Nest2}, as shown in   \cite{AC10},  \cite{attouch2018fast},  \cite{ADR}, \cite{CD}, \cite{SBC}.
Related systems have been considered  in \cite{LJ} from the closed loop control perspective and in \cite{SDJS} by means of 
high-resolution of differential equations. 

\subsubsection{Implicit Hessian}

The second system we consider, inspired by~\cite{alecsa2019extension} (see also \cite{MJ} for a related autonomous system  in the case of a strongly convex function $f$), is
\begin{equation}\tag{ISIHD}\label{eq:odetwo}
\ddot{x}(t)+\frac{\alpha}{t} \dot{x}(t)+\nabla f\Big(x(t)+\beta(t)\dot{x}(t)\Big)=0 ,
\end{equation}
where $\alpha \ge 3$ and $\beta(t)=\gamma+\frac{\beta}{t}$, $\gamma, \; \beta \geq 0$. We coin this ODE an Inertial System with {\emph{Implicit}} Hessian Damping (ISIHD for short). The rationale justifying our use of the term ``implicit" comes from the observation that by a Taylor expansion (as $t \to +\infty$ we have $\dot{x}(t) \to 0$ which justifies using Taylor expansion), one has
\[
\nabla f\pa{x(t)+\beta(t)\dot{x}(t)}\approx \nabla f (x(t)) + \beta(t)\nabla^2 f(x(t))\dot{x}(t) ,
\]
hence making the Hessian damping appear indirectly in (\ref{eq:odetwo}). This ODE was found to have a smoothing effect on the energy error and oscillations. 

\subsubsection{Exogenous additive error}
We are interested in the situation where $\nabla f(x(t))$ is always evaluated with an exogenous additive error $e(t)$. With the choice of parameters made above, the perturbed dynamics of~\eqref{eq:origode} and \eqref{eq:odetwo} are written
\begin{equation}
\boxed{
\ddot{x}(t)+\frac{\alpha}{t}\dot{x}(t)+\beta\frac{d}{dt}\Big(\nabla f(x(t))+e(t)\Big)+ \nabla f(x(t))+e(t)=0 ,\tag{ISEHD-\textsc{Pert}} \hspace{1cm}\label{eq:origode_a}
}
\end{equation}
\begin{equation}
\boxed{
\ddot{x}(t)+\frac{\alpha}{t} \dot{x}(t)+\nabla f\Big(x(t)+\beta(t)\dot{x}(t)\Big)+e(t)=0. \hspace{2.9cm}\tag{{ISIHD-\textsc{Pert}}} \label{eq:odetwo_a}
}
\end{equation}
For system \eqref{eq:origode_a},  the overall perturbation error affecting the system is 
$\beta\dot{e}(t) + e(t).$
Because the Hessian appears explicitly, both the error on the gradient and its derivative appear. It can then be anticipated that assumptions regarding both $e(t)$ and $\dot{e}(t)$, in particular their integrability, will be instrumental in deriving any convergence guarantees. On the other hand, in the system \eqref{eq:odetwo_a} with implicit Hessian damping, the error perturbation $e(t)$ appears without its time derivative. Naturally, we will see in this case that convergence results will be derived without any assumptions on the time derivative of the error. While this may be seen as an advantage at first glance, this comes at a price. Indeed, as we will also see, to maintain fast convergence guarantees, the integrability requirements on the error $e(t)$ will be more stringent for \eqref{eq:odetwo_a} than for \eqref{eq:origode_a}, \ie, higher-order moments of $e(t)$ will be required to be finite. We anticipate that when it comes to discrete algorithms, the assumptions on the objective and perturbations needed in the implicit case are more stringent than in the explicit case. We plan to study these questions in a future work. Note that similar questions arise when the perturbation is attached to a Tikhonov regularization term with an asymptotically vanishing coefficient \cite{BCL}. 

One of our motivations for the above additive perturbation model originates from optimization, where the gradient may be accessed only inaccurately, either because of physical or computational reasons. The prototype example we think of is
\[
f(x) = \EE_{\xi}[F(x,\xi] ,
\] 
where $\EE_{\xi}[\cdot]$ is expectation with respect to the random variable $\xi$, and $F(\cdot,\xi) \in \cC^1(\cH)$ for any $\xi$. This is a popular setting in numerous applications (imaging, statistical learning, etc.), where computing $\nabla f(x)$ (or even $f(x)$) is either impossible or computationally very expensive. Rather, one draws $m$ independent samples of $\xi$, say $(\xi_i)_{1 \leq i \leq m}$, and compute the average estimate 
\[
\widehat{\nabla f}(x) = \frac{1}{m} \sum_{i=1}^m \nabla F(x,\xi_i) .
\] 
In our notation, the error is then $e(t) = \nabla f(x(t)) - \widehat{\nabla f}(x(t))$. Under independence and mild assumptions, one has by the law of iterated logarithm that $\|e(t)\|=\cO\pa{\sqrt{\frac{\log(\log(m))}{m}}}$ almost surely. Thus, to make this error vanish or even integrable, one has to take $m(.)$ an increasing function of $t$ at least at the rate $\cO(t^{2(1+\delta)})$ for $\delta > 0$. At this stage, some readers may have expected a smallness condition on the perturbation rather than integrability conditions, i.e., of $\|e(t)\|=\mathcal{O}(r(t))$ for some $r(t)$ such that $\lim\limits_{t\to\infty} r(t)=0$. Of course, integrability is stronger as it obviously implies that the error function, whenever it converges, will vanish asymptotically. However, one has to keep in mind that our goal is not only to establish convergence (and rate of convergence) of the objective values to a "perturbation dominated region" around the optimal value, but to show additional convergence guarantees, in particular the important matter of weak convergence of the trajectory. Deriving convergence of trajectories is typically much more challenging than the objective and cannot be proved under a mere smallness condition on the error.

%%%%%%%%%%%%%%%%%%%%%%%%%%%%%%%%%%%%%%%%%%%%%%%%%%%%%%%%%%%%%%%%%%%%%%%%%%%%%%%%%%%%%%%%%%%%%%%%%%%%%%%%%%%%%%%
%%%%%%%%%%%%%%%%%%%%%%%%%%%%%%%%%%%%%%%%%%%%%%%%%%%%%%%%%%%%%%%%%%%%%%%%%%%%%%%%%%%%%%%%%%%%%%%%%%%%%%%%%%%%%%%
\subsection{Contributions}
In \cite{attouch2019first} (resp. \cite{alecsa2019extension}), which studied the unperturbed system \eqref{eq:origode} (resp. \eqref{eq:odetwo}),  fast convergence rates were obtained for  the objective, velocities and gradients.
Our main contribution in this paper is to analyze the robustness and stability of these systems, by quantifying their convergence properties in the presence of errors. We do this both in the general convex case and in the strongly convex case. We also study the case where the function $f$ is non-smooth convex by proposing a first order  formulation in time and space, with existence results and Lyapunov analysis. The main motivation for our work is to pave the way for the design and study of provably accelerated optimization algorithms that appropriately discretize the above dynamics while handling inexact evaluations of the gradient with deterministic and/or stochastic errors. The extension to the discrete setting of the results here will be the focus of a forthcoming paper.

%%%%%%%%%%%%%%%%%%%%%%%%%%%%%%%%%%%%%%%%%%%%%%%%%%%%%%%%%%%%%%%%%%%%%%%%%%%%%%%%%%%%%%%%%%%%%%%%%%%%%%%%%%%%%%%
\subsection{Related Works}

Due to  the importance of the subject in optimization and control, several articles have been devoted to  the study of perturbations  in dissipative inertial systems and in the corresponding 
 accelerated first order algorithms.  The subject was first considered in the case of a fixed  viscous damping, \cite{acz,HJ1}. Then it was studied within the framework of  the accelerated gradient method of  Nesterov, and of the corresponding inertial dynamics with vanishing viscous damping, see \cite{SDJS,AC2R-JOTA,attouch2018fast,AD15,SLB,VSBV}. In the presence of the additional Hessian driven damping, first results have been obtained in \cite{APR1,attouch2019first,attouchiterates2021} in the case of a smooth function. To the best of our knowledge, our work is the first to consider these questions in full generality and in presence of perturbations.

%%%%%%%%%%%%%%%%%%%%%%%%%%%%%%%%%%%%%%%%%%%%%%%%%%%%%%%%%%%%%%%%%%%%%%%%%%%%%%%%%%%%%%%%%%%%%%%%%%%%%%%%%%%%%%%
\subsection{Contents}
In Section~\ref{s:wellposed}, we prove that the two systems are well-posed both in the smooth and non-smooth cases. In Section~\ref{s:convex}, we study the convex case, and establish convergence rates for both systems under appropriate integrability assumptions on the error. In  Section~\ref{s:sconvex}, we consider the strongly convex case. Section~\ref{s:nonsmooth}\; is devoted to studying non-smooth $f$. In Section~\ref{s:num}, we present some numerical illustrations of the results. In Section~\ref{s:conc}, we draw key conclusions and present some perspectives. 

%%%%%%%%%%%%%%%%%%%%%%%%%%%%%%%%%%%%%%%%%%%%%%%%%%%%%%%%%%%%%%%%%%%%%%%%%%%%%%%%%%%%%%%%%%%%%%%%%%%%%%%%%%%%%%%
\subsection{Main notations}
$\cH$ is a real Hilbert space,  $\dotp{\cdot}{\cdot}$ is the scalar product on $\cH$ and $\norm{\cdot}$ is the corresponding norm. $\Gamma_0(\cH)$ is the class of proper, lower semicontinuous (lsc) and  convex functions from $\cH$ to $\Rb$. A function  $g: \cH \to \Rb$ is  $\mu$-strongly convex ($\mu >0$) if $g - \frac{\mu}{2}\| \cdot\|^2$ is convex. For $g \in \Gamma_0(\cH)$, its domain is $\dom(g) \eqdef  \setcond{x\in\cH}{g(x) < +\infty}$. $\partial g$ denotes the (convex) subdifferential operator of $g$. When $g$ is differentiable at $x\in\cH$, then $\partial g(x)=\bra{\nabla g(x)}$. We also denote $\dom(\partial g) \eqdef \setcond{x \in \cH}{\partial g(x) \neq \emptyset}$.

$\cC^s(\cD)$ is the class of $s$-continuously differentiable functions on $\cD$, $\cD$ will be specified in the context. For $T > t_0 $ and $p \geq 1$, $L^p(t_0,T; \cH)$ is the Lebesgue space of measurable functions $x: t \in [t_0,T] \mapsto x(t) \in \cH$ such that $\int_{t_0}^T \norm{x(t)}^p dt < +\infty$. $\cW^{1,1} (t_0,T; \cH)$ is the Sobolev space of functions $x(.) \in L^1(t_0,T; \cH)$ with distributional derivative $\dot{x}(\cdot)\in L^{1} (t_0,T;\cH)$. We will also invoke the notion of a strong solution to a differential inclusion, see \cite[Definition 3.1]{Bre1}, that will be given precisely in Section \ref{section-gen-existence}. The reader interested mostly in quantitative convergence estimates can skip the corresponding section.

We take $t_0 >0$ as the origin of time. This is justified by the singularity of the viscous damping coefficient $\frac{\alpha}{t}$ at the origin. This is not restrictive since we are interested in asymptotic analysis. We denote $S$ to be the set of minimizers of $f$, \ie, $S \eqdef \argmin\limits_{x\in\cH} f(x)$ which is assumed nonempty, and $x^\star$ to be an arbitrary element of $S$ (unique in the case of strongly convex $f$). We denote $\bar{f} \eqdef \inf\limits_{x\in\cH} f(x)$.

%%%%%%%%%%%%%%%%%%%%%%%%%%%%%%%%%%%%%%%%%%%%%%%%%%%%%%%%%%%%%%%%%%%%%%%%%%%%%%%%%%%%%%%%%%%%%%%%%%%%%%%%%%%%%%%
\section{Well-posedness}\label{s:wellposed}
When $\beta > 0$, the presence of Hessian driven damping in the inertial dynamics makes it possible to reformulate the equations as  first-order systems both in time and in space, without explicit evaluation of the Hessian. This will allow us to extend the existence of trajectories and the convergence results to the case  $f \in \Gamma_0(\cH)$, by simply replacing the gradient of $f$ with the subdifferential $\partial f$. This approach was initiated in \cite{aabr} and used in \cite{APR1} for the unperturbed case.
\subsection{Explicit Hessian Damping}
\subsubsection{Formulation as a first-order system}
 Let us start by establishing this equivalence in the case of a smooth function $f$.
\begin{theorem}\label{Thm-first-order-system}
Let $f:\cH\to\R$ be a $\cC^2(\cH)$ function and $e: [t_0,+\infty[ \to \cH$ be $\cC^1(\cH)$. Suppose that $\alpha\geq0$, $\beta>0$. Let $(x_0,\dot x_0)\in \cH \times \cH$. The 
following statements are equivalent:
\begin{enumerate}[label=\arabic*.]
\item \label{item:Thm-first-order-system_1}
$x:[t_0,+\infty [ \to \cH$ is a solution trajectory of \eqref{eq:origode_a}
% \begin{equation*}
% \ddot{x}(t)+\frac{\alpha}{t}\dot{x}(t)+\beta \nabla^2 f(x(t))\dot{x}(t)+\beta \dot{e}(t)+\nabla f(x(t))+e(t)=0
% \end{equation*}
with the initial conditions $x(t_0)=x_0$, $\dot x(t_0)=\dot x_0$. 
\item \label{item:Thm-first-order-system_2}
$(x,y):[t_0,+\infty [ \to \cH \times \cH$ is a solution trajectory of the first-order system 
\begin{equation}\label{eq:ISEHDequiv1stode}
\begin{cases}
\dot x(t) +  \beta (\nabla f (x(t))+ e(t)) - \pa{\frac{1}{\beta} - \frac{\alpha}{t}} x(t) + \frac{1}{\beta}y(t) &= 0 \\
\dot{y}(t)-\pa{\frac{1}{\beta} - \frac{\alpha}{t} + \frac{\alpha \beta}{t^2}} x(t) + \frac{1}{\beta} y(t) &= 0,
\end{cases}
\end{equation}
\end{enumerate}
with initial conditions $x(t_0)=x_0$, 
$y(t_0)=-\beta(\dot x_0+\beta\nabla f(x_0))+(1-\beta\alpha/t_0)x_0 -\beta^2 e(t_0)$.
\end{theorem}
\begin{proof}
\textit{{\ref{item:Thm-first-order-system_2} }}$\Rightarrow$ \textit{{\ref{item:Thm-first-order-system_1} }} 
Differentiating the first equation of \eqref{eq:ISEHDequiv1stode} gives
\begin{equation}\label{eq:fos01}
\ddot x(t) +  \beta \pa{\nabla^2 f (x(t)) \dot{x}(t) + \dot{e}(t)} -
\frac{\alpha}{t^2}x(t)- \pa{\frac{1}{\beta} - \frac{\alpha}{t}} \dot x(t)+\frac{1}{\beta}\dot y(t)  =  0 .
\end{equation}
Replacing $\dot y(t)$ by its expression as given  by the second equation of \eqref{eq:ISEHDequiv1stode} gives
\begin{equation}
\ddot x(t) +  \beta \pa{\nabla^2 f (x(t)) \dot{x}(t) + \dot{e}(t)} -
\frac{\alpha}{t^2}x(t)- \pa{ \frac{1}{\beta} -  \frac{\alpha}{t} } \dot x(t) 
+\frac{1}{\beta}\pa{\pa{ \frac{1}{\beta} -  \frac{\alpha}{t} + \frac{\alpha \beta}{t^2}} x(t)
 -\frac{1}{\beta} y(t)}  =  0 .\label{eq:fos02}
\end{equation}
Then replace $y(t)$ by its expression as given  by the first equation of \eqref{eq:ISEHDequiv1stode}
\begin{multline*}
\ddot x(t) +  \beta \pa{\nabla^2 f (x(t)) \dot{x}(t) + \dot{e}(t)} -
\frac{\alpha}{t^2}x(t)- \pa{ \frac{1}{\beta} -  \frac{\alpha}{t} } \dot x(t) \\
+\frac{1}{\beta}\pa{\pa{ \frac{1}{\beta} -  \frac{\alpha}{t} +  \frac{\alpha \beta}{t^2}  } x(t)
+ \dot x(t) +  \beta (\nabla f (x(t))+ e(t)) - \pa{ \frac{1}{\beta} -  \frac{\alpha}{t}}  x(t)}  =  0 .\label{eq:fos03}
\end{multline*}
After simplification of the above expression, we obtain \eqref{eq:origode_a}.

\medskip

\textit{{\ref{item:Thm-first-order-system_1}}} $\Rightarrow$ \textit{{\ref{item:Thm-first-order-system_2}}} Define $y(t)$ by  the first equation of \eqref{eq:ISEHDequiv1stode}. Differentiating $y(t)$ and using equation \eqref{eq:origode_a} allows one to eliminate $\ddot{x}(t)$, which finally gives the second equation of \eqref{eq:ISEHDequiv1stode}. \qed
\end{proof}

%%%%%%%%%%%%%%%%%%%%%%%%%%%%%%%%%%%%%%%%%%%%%%%%%%%%%%%%%%%%%%%%%%%%%%%%%%%%%%%%%%%%%%%%%%%%%%%%%%%%%%%%%%%%%%%
\subsubsection{Existence and uniqueness of a solution}\label{section-gen-existence}

Capitalizing on the result of Theorem~\ref{Thm-first-order-system}, the following first order formulation assists in providing a meaning to our system when $f \in \Gamma_0(\cH)$.
It is obtained by substituting the subdifferential $\partial f$ for the gradient $\nabla f$ in the first-order formulation \eqref{eq:ISEHDequiv1stode}.
\begin{definition}
Let $\alpha\geq0$, $\beta>0$ and $f \in \Gamma_0(\cH)$. {Given
$(x_0, y_0) \in  \dom(f) \times \cH $, the Cauchy problem for the perturbed inertial system with explicit generalized Hessian driven damping is defined by}
\begin{equation}\label{eq:fos2}
\begin{cases}
\dot x(t)+\beta(\partial  f(x(t)) + e(t)) -\pa{\frac{1}{\beta}-\frac{\alpha}{t}}x(t)+\frac{1}{\beta}y(t) & \ni 0 \\
\dot{y}(t) -\pa{\frac{1}{\beta}-\frac{\alpha}{t}+\frac{\alpha\beta}{t^2}}x(t) +\frac{1}{\beta} y(t) & = 0 \\
x(t_0)=x_0, y(t_0)=y_0 .
\end{cases}
\end{equation}
\end{definition}
Let us formulate \eqref{eq:fos2} in a condensed form as an evolution equation in the product space $\cH \times \cH$.
Setting $Z(t) = (x(t), y(t)) \in  \cH \times \cH$, \eqref{eq:fos2}
can be equivalently written 
\begin{equation}\label{eq:fos3}
\dot{Z}(t) + \partial \cG(Z(t)) + \cD(t, Z(t)) \ni 0, \quad {Z(t_0)=(x_0,y_0)},
\end{equation}
where $\cG \in \Gamma_0(\cH \times \cH)$ is the function defined by $\cG (Z) = \beta f (x)$, and the time-dependent operator $\cD:\ [t_0,+\infty[ \times \cH \times \cH \to \cH \times \cH$ is given by
\begin{equation}\label{eq:fos5}
\cD(t,Z)=
  \pa{\beta e(t)-\pa{\frac{1}{\beta}-\frac{\alpha}{t}}x+\frac{1}{\beta}y,
    -\pa{\frac{1}{\beta}-\frac{\alpha}{t}+\frac{\alpha\beta}{t^2}}x
 +\frac{1}{\beta} y } .
\end{equation}

The differential inclusion \eqref{eq:fos3} is governed by the sum of the maximal monotone operator $\partial \cG$ (a convex subdifferential) and the time-dependent affine continuous operator $\cD (t,\cdot)$. The existence and uniqueness of a global solution for the corresponding Cauchy problem is a consequence of the general theory of evolution equations governed by maximally monotone operators. In this setting, we need to invoke the notion of strong solution that we make precise now.

\begin{definition}\label{defsolforte}
Given $g \in \Gamma_0(\cH)$, and an operator $D:[t_0,+\infty[\times \cH\to\cH$, we say that $z:[t_0,T]\to \cH$ 
is a strong solution trajectory on $[t_0,T]$ of the differential inclusion
\begin{equation}\label{eq:fosdef}
\dot z(t)+\partial  g(z(t))+D(t,z(t)) \ni 0,
\end{equation}
if the following properties are satisfied:
\begin{enumerate}[label=\arabic*.]
\item \label{defsolforte-1}
$z$ is continuous on $[t_0,T]$ and absolutely continuous on any compact subset of $]t_0,T]$;
\item \label{defsolforte-3} 
$z(t)\in\dom(\partial  g)$ for almost every $t\in]t_0,T]$, and 
 \eqref{eq:fosdef} is verified for almost every $t\in]t_0,T]$.
\end{enumerate}

\noindent $z:[t_0,+\infty[\to \cH$ is  a global strong solution of \eqref{eq:fosdef}, if it is a strong solution  on $[t_0,T]$ for all $T>t_0$. 
\end{definition}

The existence and uniqueness of a global strong solution of the Cauchy problem \eqref{eq:fos2} is established in the following theorem.

\begin{theorem} \label{Thm-existence}
Let $f \in \Gamma_0(\cH)$, $\alpha \geq 0$ and $\beta >0$.
Suppose that $e \in L^2 (t_0, T; \cH)$ for every $T > t_0$. Then, for any Cauchy data $(x_0, y_0) \in  \dom(f) \times \cH $, there exists a unique global strong solution $(x,y):[t_0, +\infty[ \to \cH \times \cH $ of \eqref{eq:fos2} satisfying the initial condition $x(t_0)=x_0$, $y(t_0)= y_0$. Moreover, this solution exhibits the following properties:
\begin{enumerate}[label={\rm (\roman*)}]
\item \label{existence-1}
$y \in \cC^1([t_0,+\infty[)$, and 
$ 
\dot{y}(t)
  -\pa{\frac{1}{\beta}-\frac{\alpha}{t}+\frac{\alpha\beta}{t^2}}x(t)
   +\frac{1}{\beta} y(t) =0,$ 
for  $t\geq t_0$;
\item \label{existence-2}
$x$ is absolutely continuous on $[t_0,T]$ and $\dot x\in L^2(t_0,T; \cH)$ for 
all $T>t_0$;
\item \label{existence-3}
$x(t)\in\dom(\partial  f)$ for all $t>t_0$;
\item \label{existence-4}
$x$ is Lipschitz continuous on any compact subinterval of $]t_0,+\infty[$;
\item 
the function $ t\mapsto f(x(t))$ is absolutely continuous on 
$[t_0,T]$ for all $T>t_0$;
\item \label{existence-6}
there exists a function $\xi:[t_0,+\infty[\to \cH$ such that 
\begin{enumerate}[label={\rm (\alph*)}]
\item \label{existence-6:a}
  $\xi(t)\in\partial  f(x(t))$ for all $t>t_0$;
\item \label{existence-6:b}
  $\dot x(t)+\beta\xi(t) + \beta e(t)
  -\pa{\frac{1}{\beta}-\frac{\alpha}{t}}x(t)+\frac{1}{\beta}y(t)=0$  
  for almost every $t>t_0$;
\item \label{existence-6:c}
  $\xi\in L^2(t_0,T; \cH)$ for all $T>t_0$;
\item \label{existence-6:d}
  $\displaystyle{\frac{d}{dt}f(x(t))}=\langle\xi(t),\dot x(t)\rangle$ for almost every 
  $t>t_0$.
\end{enumerate}
\end{enumerate}
\end{theorem}

\begin{proof} 
It is sufficient to prove that $(x,y)$ is a strong solution of \eqref{eq:fos2} on $[t_0,T]$ and that the properties hold on $[t_0,T]$ for all $T>t_0$. So let us fix $T>t_0$. As we have already noticed, \eqref{eq:fos2} can be  written in the form \eqref{eq:fos3} which is a Lipschitz perturbation of the differential inclusion governed by the  subdifferential of a proper lsc convex function. A direct  application of \cite[Proposition~3.12]{Bre1} gives the existence and uniqueness of a strong  global solution $Z=(x,y):[t_0,T]\to\cH \times \cH$ to \eqref{eq:fos3}, or equivalently to \eqref{eq:fos2}, with initial condition $Z(t_0)=(x(t_0),y(t_0))=(x_0,y_0)$.  Verification of items \ref{existence-3} to \ref{existence-6} follows the same lines as the proof of \cite[Theorem~4.4]{APR1}. Of particular importance is the generalized derivation chain rule given in \ref{existence-6}\ref{existence-6:d}, which follows from \cite[Lemma~3.3]{Bre1} after checking that the corresponding assumptions are met thanks to \ref{existence-2}, \ref{existence-6}\ref{existence-6:a} and \ref{existence-6}\ref{existence-6:c}.\qed
\end{proof}
{
Under sufficient differentiability properties of the data,
we recover a classical solution, \ie \; $x(\cdot)$ is a $\cC^2([t_0, +\infty[)$ function, all the derivatives involved in   the equation \eqref{eq:origode_a} are taken in the sense of  classical differential calculus, and the equation \eqref{eq:origode_a} is satisfied for all $t \geq t_0$. This is made precise in the following statement.
\begin{corollary}\label{corr-existence-origode_a}
Assume that $f$ is a convex $\cC^2(\cH)$ function and $e$ belongs to $\cC^1([t_0, +\infty[)$. Then, for any $t_0 > 0$, and any Cauchy data $(x_0, \dot{x}_0)$, the system \eqref{eq:origode_a} with $\alpha, \beta \geq 0$ admits a unique classical global solution $x: [t_0,+\infty[ \to \cH$ satisfying $(x(t_0),\dot{x}(t_0)) = (x_0, \dot{x}_0)$.
\end{corollary}
\begin{proof}
Under the above regularity assumptions, the first equation of the first order system \eqref{eq:fos2}
$$
\dot x(t)+\beta(\nabla  f(x(t)) + e(t)) -\pa{\frac{1}{\beta}-\frac{\alpha}{t}}x(t)+\frac{1}{\beta}y(t) = 0
$$
implies that $\dot x$ is a $\cC^1([t_0, +\infty[)$ function, and hence $x\in\cC^2([t_0, +\infty[)$.
Then, combining Theorem~\ref{Thm-first-order-system} with Theorem~\ref{Thm-existence} with $y(t_0) = -\beta(\dot x_0+\beta\nabla f(x_0))+(1-\beta\alpha/t_0)x_0 -\beta^2 e(t_0)$, we obtain the existence and uniqueness of a classical solution  to the Cauchy problem associated with \eqref{eq:origode_a}.
\end{proof}
}

%%%%%%%%%%%%%%%%%%%%%%%%%%%%%%%%%%%%%%%%%%%%%%%%%%%%%%%%%%%%%%%%%%%%%%%%%%%%%%%%%%%%%%%%%%%%%%%%%%%%%%%%%%%%%%%
\subsection{Implicit Hessian Damping}

%%%%%%%%%%%%%%%%%%%%%%%%%%%%%%%%%%%%%%%%%%%%%%%%%%%%%%%%%%%%%%%%%%%%%%%%%%%%%%%%%%%%%%%%%%%%%%%%%%%%%%%%%%%%%%%
\subsubsection{Formulation as a first-order system}
Let us now turn to \eqref{eq:odetwo_a}. We use the shorthand notation $\alpha(t)=\alpha/t$. Here and in the rest of the paper, we assume that $\beta(\cdot)$ is  $\cC^1([t_0, +\infty[, \R^+)$ and 
$\inf\limits_{t\in[t_0,+\infty[} \beta(t) > 0$.\\
Let us introduce the new function
\begin{equation}\label{eq:odetwo-2}
y(t) \eqdef x(t)+\beta(t)\dot{x}(t),
\end{equation}
whose time derivation gives 
\begin{equation}\label{eq:odetwo-3}
\dot{y}(t) = \dot{x}(t)+\beta(t)\ddot{x}(t)+ \dot{\beta}(t) \dot{x}(t).
\end{equation}
From \eqref{eq:odetwo_a} we know that
\begin{equation}\label{eq:odetwo-4}
\ddot{x}(t)= -\alpha (t) \dot{x}(t)-\nabla f (y(t))-e(t).
\end{equation}
By combining \eqref{eq:odetwo-3} and \eqref{eq:odetwo-4} we obtain
\begin{align}
\dot{y}(t) &= \dot{x}(t)+\beta(t)\pa{-\alpha(t)\dot{x}(t)-\nabla f(y(t))-e(t)} + \dot{\beta}(t) \dot{x}(t) \nonumber \\
&= \pa{1 - \alpha(t)\beta(t) + \dot{\beta}(t)}\dot{x}(t)-\beta(t)\pa{\nabla f(y(t) + e(t)}.\label{eq:odetwo-5}
\end{align}
From \eqref{eq:odetwo-2} and the fact that $\inf\limits_{t\in[t_0,+\infty[} \beta(t) > 0$ we get  
$
\dot{x}(t) = \frac{1}{\beta (t)}(y(t) -x(t)).
$
Replacing $\dot{x}(t)$ in \eqref{eq:odetwo-5} with this expression gives
\begin{eqnarray*}
\dot{y}(t) 
&=& \pa{1 - \alpha(t)\beta(t) + \dot{\beta}(t)}\frac{1}{\beta (t)}(y(t) -x(t))-\beta(t)\pa{\nabla f (y(t)) +e(t)} \\
&=& -\frac{1}{\beta (t)}\pa{1 - \alpha(t)\beta(t) + \dot{\beta}(t)}x(t) + \frac{1}{\beta (t)}\pa{1 - \alpha(t)\beta(t) + \dot{\beta}(t)}y(t) - \beta(t)\pa{\nabla f (y(t)) + e(t)} .
\end{eqnarray*}
The reverse implication is obtained in a similar way. Let us summarize the results.
\begin{theorem}\label{Thm-first-order-system_implicit}
Let $f \in \cC^1(\cH)$. Suppose that $\alpha \geq 0$ and $\inf\limits_{t\in[t_0,+\infty[} \beta(t) > 0$. The following statements are equivalent: 
\begin{enumerate}[label=\arabic*.]
\item $x:[t_0,+\infty [ \to \cH$ is a solution trajectory of \eqref{eq:odetwo_a} with initial conditions $x(t_0)=x_0$, $\dot x(t_0)=\dot x_0$. 
\item $(x,y):[t_0,+\infty [ \to \cH \times \cH$ is a solution trajectory of the first-order system 
\begin{align}\label{eq:ISIHDequiv1stode}
\begin{cases}
\dot{x}(t) + \frac{1}{\beta(t)}x(t) - \frac{1}{\beta(t)}y(t) = 0.   \vspace{1mm} \\
\dot{y}(t) + \beta(t)\pa{\nabla f (y(t)) + e(t)} + \frac{1}{\beta (t)}\pa{1 - \alpha(t)\beta(t) + \dot{\beta}(t)}(x(t) - y(t))  = 0  
\end{cases}
\end{align}
with initial conditions $x(t_0)=x_0$, 
$y(t_0)=x_0 + \beta(t_0)\dot{x}_0$.
\end{enumerate}
\end{theorem}
%

%%%%%%%%%%%%%%%%%%%%%%%%%%%%%%%%%%%%%%%%%%%%%%%%%%%%%%%%%%%%%%%%%%%%%%%%%%%%%%%%%%%%%%%%%%%%%%%%%%%%%%%%%%%%%%%
\subsubsection{Existence and uniqueness of a solution}
Existence and uniqueness of a global strong solution for the {Cauchy problem associated with the} unperturbed problem \eqref{eq:odetwo} was shown in \cite{alecsa2019extension} when $\nabla f$ is Lipschitz continuous using the Cauchy-Lipschitz theorem. This result can be easily extended to \eqref{eq:odetwo_a}. Rather, we take a different path here and proceed as in Section~\ref{section-gen-existence}, so that we can extend the above formulation  to the case where $f \in \Gamma_0(\cH)$, by replacing the gradient $\nabla f$ with the subdifferential $\partial f$.
\begin{definition}
Let $\alpha(t) \geq 0$, $\beta (t) > 0$, $f \in \Gamma_0(\cH)$. {Given
$(x_0, y_0) \in \cH  \times \dom(f) $, the Cauchy problem associated with the} perturbed inertial system with implicit generalized Hessian driven damping is defined by
\begin{align}
 \begin{cases}
\dot{x}(t) + \frac{1}{\beta(t)}x(t) - \frac{1}{\beta(t)}y(t) = 0 	 \\
\dot{y}(t) + \beta(t)\pa{\partial f (y(t)) + e(t)} + \frac{1}{\beta (t)}\pa{1 - \alpha(t)\beta(t) + \dot{\beta}(t)}(x(t) - y(t))  \ni 0  \label{eq:fos2_implicit}\\
x(t_0) = x_0, y(t_0)=y_0 .
\end{cases}
\end{align}
\end{definition}
We reformulate \eqref{eq:fos2_implicit} in the product space $\cH \times \cH$ by setting $Z(t) = (x(t), y(t)) \in  \cH \times \cH$, and thus \eqref{eq:fos2_implicit} can be equivalently written as
\begin{equation}\label{eq:fos3_implicit}
\dot{Z}(t) + \beta(t) \partial \cG (Z(t)) + \cD(t, Z(t)) \ni 0,
\end{equation}
where $\cG \in \Gamma_0(\cH \times \cH)$ is the function defined as
$ \cG (Z) =  f (y)$, and the time dependent operator $\cD:\ [t_0,+\infty[ \times \cH \times \cH \to \cH \times \cH$ is given by
\begin{equation}\label{eq:fos5_implicit}
\cD(t,Z)=
  \pa{\frac{1}{\beta (t)}(x - y) ,
    \beta(t)e(t) + \frac{1}{\beta (t)}\pa{1 - \alpha(t)\beta(t) + \dot{\beta}(t)}(x - y)} .
\end{equation}

\myparagraph{Constant $\beta$} When $\beta$ is independent of $t$, the differential inclusion \eqref{eq:fos3_implicit} is governed by the sum of the convex subdifferential operator $\beta \partial \cG$ and the time-dependent affine continuous operator $\cD (t,\cdot)$. The existence and uniqueness of a global strong solution for the Cauchy problem associated to \eqref{eq:fos2_implicit} follows exactly from the same arguments as those for Theorem~\ref{Thm-existence}. {In turn, if $f\in \cC^1(\cH)$, $e\in \cC([t_0, +\infty[)$, and $\beta \in \cC^1([t_0, +\infty[)$, then \eqref{eq:ISIHDequiv1stode} admits a unique $\cC^1([t_0, +\infty[)$ global solution $(\dot{x},\dot{y})$. It then follows from the first equation in \eqref{eq:ISIHDequiv1stode} that $\dot x$ is a $\cC^1([t_0, +\infty[)$ function, and hence $x\in\cC^2([t_0, +\infty[)$. Existence and uniqueness of a classical global solution to the Cauchy problem associated to \eqref{eq:odetwo_a} is then obtained thanks to the equivalence in Theorem~\ref{Thm-first-order-system_implicit}.}

\smallskip

\myparagraph{Time-dependent $\beta$} When $\beta$ depends on time, one cannot invoke directly the results of \cite{Bre1}. Instead, one can appeal to the theory of evolution equations governed by general time-dependent subdifferentials as proposed in \cite{AD} for example. {In fact, for a system in the simpler form \eqref{eq:fos3_implicit}, one can argue more easily, by  making the change of time variable $t = \tau(s)$ with  $\beta(\tau(s))\dot{\tau}(s)=1$.  Lemma~\ref{lem:timerescale} then shows that \eqref{eq:fos3_implicit} is equivalent to 
\begin{equation}\label{eq:fos3_implicit_rescale}
\dot{W}(s) +  \partial \cG (W(s)) + \cF(s, W(s)) \ni 0,
\end{equation}
where $W(s)= Z(\tau(s))$, and $\cF(s,W(s))= \frac{1}{\beta(\tau(s))}\cD(\tau(s),W(s))$ is affine continuous in its second argument.  Provided that $\beta \not\in L^1(t_0,+\infty;\R)$, this defines a proper change of variable in time}. With the formulation \eqref{eq:fos3_implicit_rescale}, we are brought back to the appropriate form to argue as before and invoke the results of \cite{Bre1}. We leave the details to the reader for the sake of brevity.

%%%%%%%%%%%%%%%%%%%%%%%%%%%%%%%%%%%%%%%%%%%%%%%%%%%%%%%%%%%%%%%%%%%%%%%%%%%%%%%%%%%%%%%%%%%%%%%%%%%%%%%%%%%%%%%
\section{Smooth Convex Case}\label{s:convex}
%%%%%%%%%%%%%%%%%%%%%%%%%%%%%%%%%%%%%%%%%%%%%%%%%%%%%%%%%%%%%%%%%%%%%%%%%%%%%%%%%%%%%%%%%%%%%%%%%%%%%%%%%%%%%%%

%%%%%%%%%%%%%%%%%%%%%%%%%%%%%%%%%%%%%%%%%%%%%%%%%%%%%%%%%%%%%%%%%%%%%%%%%%%%%%%%%%%%%%%%%%%%%%%%%%%%%%%%%%%%%%%
\subsection{Explicit Hessian Damping}\label{s:secondorderconv}

Consider first the explicit Hessian system~\eqref{eq:origode_a}, where we assume that $f \in \cC^2(\cH)$, and recall the specific choices of $\gamma(t)=\frac{\alpha}{t}$, $\alpha > 0$, $\beta(t)\equiv\beta$ and $b(t)\equiv 1$. 
We will develop a Lyapunov analysis to study the dynamics of~\eqref{eq:origode_a}. Some of our arguments are inspired by the works of~\cite{attouch2018fast} and \cite{attouch2019first}. Throughout this section we use the shorthand notation 
\begin{equation}\label{def_g}
g(t):= e(t) + \beta \dot{e}(t)
\end{equation}
{for the overall contribution of the errors terms. We will first establish the minimization property which is valid by simply assuming the integrability of the error term and its derivative. Then, by reinforcing these hypotheses, we will obtain  rapid convergence results, and the convergence of trajectories.}

\subsubsection{Minimizing properties}
{Define  $u: [t_0,+\infty[ \to \cH$ by
\[
u(t) \eqdef x(t) + \beta \int_{t_0}^t \nabla f(x(s)) ds,
\]
which will be instrumental in the proof of the following theorem. Note that, in the following statement, it is simply assumed that $f$ is bounded from below, the set $S \eqdef \argmin_{\cH} f$ may be empty.
\begin{theorem}\label{lem:smoothexplicitestimates}
Let $f: \cH \to \R$ be a $\cC^2(\cH)$ function which is  bounded from below. Assume that $e: [t_0,+\infty[ \to \cH$ is a $\cC^1(\cH)$ function which satisfies the integrability properties  $\DS{\int_{t_0}^{+\infty} \norm{e(t)}dt<+\infty}$ and $\DS{\int_{t_0}^{+\infty} \norm{\dot{e}(t)}dt<+\infty}$. Suppose that $\alpha,\beta>0$. Then, for any solution trajectory $x: [t_0,+\infty[ \to \cH$ of \eqref{eq:origode_a}, we have
\begin{enumerate}[label={\rm (\roman*)}]
\item \label{lem:smoothexplicitestimates1}
$\sup\limits_{t \geq t_0} \norm{\dot{u}(t)}<+\infty$; 
\item \label{lem:smoothexplicitestimates2}
$\DS{\int_{t_0}^{+\infty} \frac{1}{t} \|\dot{x}(t)\|^2 dt < +\infty}$, $\DS{\int_{t_0}^{+\infty} \norm{\nabla f(x(t))}^2 dt < +\infty}$, $\DS{\int_{t_0}^{+\infty} \frac{1}{t} \|\dot{u}(t)\|^2 dt < +\infty}$;
\item \label{lem:smoothexplicitestimates3}
 $\lim\limits_{t\to+ \infty} \norm{\dot{u}(t)}=0$; \,  $\lim\limits_{t\to+ \infty} \norm{\dot{x}(t)}=0$; $\lim\limits_{t\to+ \infty} \norm{\nabla f(x(t))}=0$;
\item \label{lem:smoothexplicitestimates4}
$\lim\limits_{t\to +\infty} f(x(t))=\inf_{\cH} f$.
\end{enumerate} 
\end{theorem}
}
\begin{proof}
Recall $\bar{f} \eqdef \inf_{\cH} f$. Since our analysis is asymptotic, {there is no restriction in assuming} that $t \geq t_1 \eqdef \max(t_0,2\alpha\beta)$. We will then prove the statements in terms of $t_1$ and passing to $t_0$ is immediate thanks to the properties of the solution $x(t)$ in Theorem~\ref{Thm-existence}. 
\paragraph{Claim~\ref{lem:smoothexplicitestimates1}}
For $T \geq t \geq t_1$, define the function
\[
W_T(t) \eqdef \frac{1}{2}\norm{\dot{u}(t)}^2 + \pa{f(x(t))-\bar{f}} - \int_t^T \dotp{\dot{u}(\tau)}{g(\tau)} d\tau .
\]
Observe that $W_T$ is well-defined under our assumptions. Thus, taking the derivative in time and using \eqref{eq:origode_a}, we get,
\begin{eqnarray}
\dot{W}_T(t) 
&=& \dotp{\dot{u}(t)}{\ddot{u}(t)+g(t)} + \dotp{\dot{x}(t)}{\nabla f(x(t))} \nonumber\\
&=& \dotp{\dot{x}(t)+\beta \nabla f(x(t))}{\ddot{x}(t)+\beta\nabla^2 f(x(t))\dot{x}(t)+g(t)} + \dotp{\dot{x}(t)}{\nabla f(x(t))} \nonumber\\
&=& \dotp{\dot{x}(t)+\beta \nabla f(x(t))}{-\frac{\alpha}{t}\dot{x}(t) -\nabla f(x(t))} + \dotp{\dot{x}(t)}{\nabla f(x(t))}  \nonumber\\
&=&-\frac{\alpha}{t} \norm{\dot{x}(t)}^2 - \beta \norm{\nabla f(x(t))}^2 - \frac{\alpha\beta}{t}\dotp{\dot{x}(t)}{\nabla f(x(t))} \nonumber\\
&\leq& -\frac{\alpha}{2t} \norm{\dot{x}(t)}^2 - \beta\pa{1-\frac{\alpha\beta}{2t}} \norm{\nabla f(x(t))}^2  \nonumber\\
&\leq& -\frac{\alpha}{2t} \norm{\dot{x}(t)}^2 - \frac{\beta}{2} \norm{\nabla f(x(t))}^2 , \label{eq:WTbound}
\end{eqnarray}
where we used Young inequality and the fact that $t \geq t_1 > \alpha\beta$. This implies that $W_T$ is non-increasing and in turn that $W_T(t)\le W_T(t_1)$ for $t \in [t_1,T]$, \ie
\begin{equation*}
\frac{1}{2}\norm{\dot{u}(t)}^2+\pa{f(x(t))-\bar{f}}-\int_t^T \dotp{\dot{u}(\tau)}{g(\tau)} d\tau 
\le \frac{1}{2} \|\dot{u}(t_1)\|^2+\pa{f(x(t_1))-\bar{f}}-\int_{t_1}^T \dotp{\dot{u}(\tau)}{g(\tau)} d\tau .
\end{equation*}
Therefore,
\[
\frac{1}{2}\norm{\dot{u}(t)}^2 \le \frac{1}{2}\|\dot{u}(t_1)\|^2+\pa{ f(x(t_1))-\bar{f}}+\int_{t_1}^t \|\dot{u}(\tau)\|\|g(\tau)\|d\tau .
\]
{Applying the Gronwall} Lemma~\ref{lem:BrezisA5}, we get
\[
\sup\limits_{t \geq t_1}  \norm{\dot{u}(t)}\le \pa{\|\dot{u}(t_1)\|^2+2(f(x(t_1))-\bar{f})}^{1/2}+\int_{t_1}^{+\infty} \|e(\tau)\| d\tau + \beta\int_{t_1}^{+\infty} \|\dot{e}(\tau)\| d\tau < +\infty ,
\]
hence proving the first claim. 

\paragraph{Claim~\ref{lem:smoothexplicitestimates2}}
Now define {for all $t\geq t_1$ }
\begin{align*}
W(t) \eqdef \frac{1}{2} \norm{\dot{u}(t)}^2+(f(x(t))-\bar{f}) - \int_t^{+\infty} \dotp{\dot{u}(\tau)}{g(\tau)} d\tau .
\end{align*}
This is again a well-defined function thanks to the first claim, 
{and the integrability of $g$}. Moreover, $W(\cdot )$ is bounded from below,
\begin{equation}\label{W_minorized}
\inf_{t \geq t_1} W(t) \ge -\brac{\sup\limits_{t \geq t_1} \norm{\dot{u}(t)}} \brac{\int_{t_1}^{+\infty} \|e(\tau)\|d\tau+\beta\int_{t_1}^{+\infty} \|\dot{e}(\tau)\|d\tau} > -\infty .
\end{equation}
Observe that $\dot{W}(t)=\dot{W}_T(t)$. This together with \eqref{eq:WTbound} yields
\begin{equation}\label{W_decrease}
\dot{W}(t) + \frac{\alpha}{2t} \norm{\dot{x}(t)}^2 + \frac{\beta}{2} \norm{\nabla f(x(t))}^2 \leq 0 .
\end{equation}
{Integrating and using that $W$ is bounded from below, we obtain the first two claims. From $\dot{u}(t)= \dot{x}(t) + \beta \nabla f(x(t))$, we deduce that  $\|\dot{u}(t)\|^2 \leq 2( \|\dot{x}(t)\|^2  + \beta^2 \| \nabla f(x(t))\|^2 )$. After integration,
we get the last claim
\begin{align*}
\int_{t_1}^{+\infty} \frac{1}{t} \|\dot{u}(t)\|^2 dt 
&\leq 2 \pa{\int_{t_1}^{+\infty} \frac{1}{t} \|\dot{x}(t)\|^2 dt + \int_{t_1}^{+\infty} \frac{\beta^2}{t} \|\nabla f(x(t))\|^2 dt} \\
&\leq 2\pa{\int_{t_1}^{+\infty} \frac{1}{t} \|\dot{x}(t)\|^2 dt + \frac{\beta^2}{t_1} \int_{t_1}^{+\infty} \|\nabla f(x(t))\|^2 dt} < +\infty .
\end{align*}}
%\noindent {\sout{Claim~{\ref{lem:smoothexplicitestimates3}} is a consequence of claim~{\ref{lem:smoothexplicitestimates2}} since $1/t$ is not integrable on $[t_1,+\infty[$.}}
\paragraph{Claim~{\ref{lem:smoothexplicitestimates3}} and~\ref{lem:smoothexplicitestimates4}} 
Define $h: t \in [t_1,+\infty[ \mapsto \frac{1}{2}\|u(t)-z\|^2$ for arbitrary $z\in\cH$. We then have
\begin{align*}
\ddot{h}(t)+\frac{\alpha}{t}\dot{h}(t) 
&= \norm{\dot{u}(t)}^2 + \dotp{u(t)-z}{\ddot{u}(t)+\frac{\alpha}{t}\dot{u}(t)} \\
&= \norm{\dot{u}(t)}^2 + \dotp{u(t)-z}{\ddot{x}(t)+\beta\nabla^2 f(x(t))\dot{x}(t) + \frac{\alpha}{t}\dot{x}(t)+\frac{\alpha\beta}{t}\nabla f(x(t))} \\
&= \norm{\dot{u}(t)}^2 - \dotp{u(t)-z}{g(t)+\pa{1-\frac{\alpha\beta}{t}}\nabla f(x(t))} \\
&= \norm{\dot{u}(t)}^2 - \pa{1-\frac{\alpha\beta}{t}}\dotp{x(t)-z}{\nabla f(x(t))} - \dotp{u(t)-z}{g(t)} \\
&-  \beta\pa{1-\frac{\alpha\beta}{t}}\dotp{\int_{t_1}^t \nabla f(x(s)) ds}{\nabla f(x(t))}   \\
&= \norm{\dot{u}(t)}^2 - \pa{1-\frac{\alpha\beta}{t}}\dotp{x(t)-z}{\nabla f(x(t))} - \dotp{u(t)-z}{g(t)} - \beta\pa{1-\frac{\alpha\beta}{t}} \dot{I}(t) ,
\end{align*}
where $\DS{I(t) \eqdef \frac{1}{2}\norm{\int_{t_1}^t \nabla f(x(s)) ds}^2}$. {From the convexity of $f$ and Cauchy-Schwarz inequality we get}
\begin{equation*}
\ddot{h}(t)+\frac{\alpha}{t}\dot{h}(t) + \pa{1-\frac{\alpha\beta}{t}}\pa{f(x(t))-f(z)} + \beta\pa{1-\frac{\alpha\beta}{t}}\dot{I}(t)
\leq \norm{\dot{u}(t)}^2  + \norm{u(t)-z}\norm{g(t)} .
\end{equation*}
Inserting $W(t)$ into this expression, we get,
\begin{eqnarray}
&&\ddot{h}(t)+\frac{\alpha}{t}\dot{h}(t) + \pa{1-\frac{\alpha\beta}{t}}\pa{W(t)+ \bar{f} -f(z)} + \beta\pa{1-\frac{\alpha\beta}{t}}\dot{I}(t) \nonumber \\
&& \leq \pa{\frac{3}{2}-\frac{\alpha\beta}{2t}}\norm{\dot{u}(t)}^2  
+ \norm{u(t)-z}\norm{g(t)} - \pa{1-\frac{\alpha\beta}{t}} \int_{t}^{+\infty} \dotp{\dot{u}(\tau)}{g(\tau)} d\tau .\label{basic-est-1}
\end{eqnarray}
According to \eqref{W_decrease} and \eqref{W_minorized}   $W(\cdot)$ is nonincreasing and bounded from below. Therefore, it converges to some $W_\infty \in \R$ as $t \to +\infty$. 
{Since $\dot{u}$ is bounded, and $g$ is integrable, we have that $\tau \to \dotp{\dot{u}(\tau)}{g(\tau)} $ is integrable on $[t_0, +\infty[$. Therefore
$$
\lim_{t\to +\infty} \int_t^{+\infty} \dotp{\dot{u}(\tau)}{g(\tau)} d\tau =0.
$$
By definition of $W(t)$, this implies that, as $t \to +\infty$
$$
\frac{1}{2} \norm{\dot{u}(t)}^2+(f(x(t))-\bar{f}) \to W_\infty.
$$
If $W_\infty =0$,  since the
two terms that enter the above expression (potential energy and kinetic energy) are nonnegative, we obtain that each of them tends to zero as $t \to +\infty$.
This gives the claims ~{\ref{lem:smoothexplicitestimates3}} and~\ref{lem:smoothexplicitestimates4}.
To prove that $W_\infty =0$, we  argue by contradiction, and show that assuming $W_\infty >0$ leads to a contradiction. Since $W(\cdot)$ is nonincreasing, we then have 
$W(t) \geq W_\infty >0$. Take $z\in \cH$ such that $f(z) < \bar{f} + \demi W_\infty$. Then
$$
W(t)+ \bar{f} -f(z) > W_\infty  - \demi W_\infty = \demi W_\infty.
$$
Returning to (\ref{basic-est-1}) we deduce that, for $t\geq t_1$
\begin{eqnarray}
&&\ddot{h}(t)+\frac{\alpha}{t}\dot{h}(t) + \demi\pa{1-\frac{\alpha\beta}{t}} W_\infty + \beta\pa{1-\frac{\alpha\beta}{t}}\dot{I}(t) \nonumber \\
&& \leq \pa{\frac{3}{2}-\frac{\alpha\beta}{2t}}\norm{\dot{u}(t)}^2  
+ \norm{u(t)-z}\norm{g(t)} - \pa{1-\frac{\alpha\beta}{t}} \int_{t}^{+\infty} \dotp{\dot{u}(\tau)}{g(\tau)} d\tau .\label{basic-est-2}
\end{eqnarray}
Since $t>2\alpha\beta$, we have $1-\frac{\alpha\beta}{t} >\demi$. Therefore, after rearranging the terms in \eqref{basic-est-2}, we obtain
$$
\frac{1}{4} W_{\infty}
\le \frac{3}{2}\norm{\dot{u}(t)}^2+\norm{u(t)-z}\norm{g(t)} 
+\pa{\sup\limits_{t\ge t_1} \norm{\dot{u}(t)}}\int_t^{+\infty} \|g(s)\|ds -\frac{1}{t^\alpha}\frac{d}{dt}(t^\alpha \dot{h}(t))
-\beta\pa{1-\frac{\alpha\beta}{t}}\dot{I}(t).
$$
}
Multiplying both sides by $\frac{1}{t}$, and integrating between $t_1$ and $\tau > t_1$,
\begin{eqnarray}
\frac{1}{4} W_{\infty} \log\pa{\frac{\tau}{t_1}} 
&\le& \frac 32\int_{t_1}^\tau \frac{1}{t}\norm{\dot{u}(t)}^2 dt +\int_{t_1}^\tau\frac{\|g(t)\|\|u(t)-z\|}{t}dt 
+ \pa{\sup\limits_{t\ge t_1} \norm{\dot{u}(t)}}\int_{t_1}^\tau \pa{\frac1t\int_t^{+\infty}\|g(s)\|ds} dt \nonumber \\ 
&&- \int_{t_1}^\tau \frac{1}{t^{\alpha+1}}\frac{d}{dt}(t^\alpha \dot{h}(t)) dt 
- \beta\int_{t_1}^\tau \pa{\frac{1}{t}-\frac{\alpha\beta}{t^2}}\dot{I}(t) dt . \label{eq:binf}
\end{eqnarray}
Throughout the rest of the proof, we will use the inequality 
\[
\norm{u(t)-z} \leq \norm{u(t_1)-z} + \int_{t_1}^t \norm{\dot{u}(s)} ds \leq \norm{u(t_1)-z} + t \sup_{s\ge t_1} \norm{\dot{u}(s)} .
\]
{Let us examine successively the different terms which enter the second member of  \eqref{eq:binf}. The first term  is bounded according to claim~{\ref{lem:smoothexplicitestimates2}}}. The second term is also bounded since 
\[
\int_{t_1}^\tau\frac{\|g(t)\|\|u(t)-z\|}{t}dt \le \pa{\frac{\|u(t_1)-z\|}{t_1}+\sup\limits_{t \ge t_1} \norm{\dot{u}(t)}}\int_{t_1}^{+\infty}
\pa{\|e(t)\|+\beta\|\dot{e}(t)\|}dt < +\infty .
\]
The third term can be handled by integration by parts,
\begin{equation*}
\int_{t_1}^\tau \pa{\frac 1t\int_t^{+\infty}\|g(s)\|ds} dt 
=\log\tau\int_\tau^\infty \|g(s)\| ds\
-\log t_1 \int_{t_1}^{+\infty} \|g(s)\| ds
+\int_{t_1}^\tau \|g(t)\|\log t ~ dt .
\end{equation*}
For the fourth term, set $\DS{K(\tau)=-\int_{t_1}^\tau \frac{1}{t^{\alpha+1}}\frac{d}{dt}(t^\alpha \dot{h}(t))dt}$ and integrate by parts twice to get,
\begin{align*}
K(\tau)&=-\brac{\frac{1}{t}\dot{h}(t)}_{t_1}^\tau-(\alpha+1)\int_{t_1}^\tau \frac{1}{t^2}\dot{h}(t)dt
= -\brac{\frac{1}{t}\dot{h}(t)}_{t_1}^\tau-\frac{(1+\alpha)}{\tau^2} h(\tau) + \frac{(1+\alpha)}{t_1^2} h(t_1)
-2(1+\alpha)\int_{t_1}^\tau \frac{h(t)}{t^3} dt \\
&\leq -\brac{\frac{1}{t}\dot{h}(t)}_{t_1}^\tau + \frac{(1+\alpha)}{t_1^2} h(t_1)
\leq \frac{1}{t_1}\abs{\dotp{\dot{u}(t_1)}{u(t_1)-z}}+\frac{1}{\tau}\abs{\dotp{\dot{u}(\tau)}{u(\tau)-z}} + \frac{(1+\alpha)}{t_1^2} h(t_1) \\
&\leq C+ \sup\limits_{t\ge t_1}\norm{\dot{u}(t)}\pa{\frac{\|u(t_1)-z\|}{t_1}+\sup\limits_{t \ge t_1} \norm{\dot{u}(t)}} < +\infty.
\end{align*}
For the last term, we infer from Lemma~\ref{L:int_bounded} that
\[
\sup_{\tau > t_1} -\beta\int_{t_1}^\tau \pa{\frac{1}{t}-\frac{\alpha\beta}{t^2}}\dot{I}(t) dt < +\infty .
\]
Overall, we have shown that there exists a constant $C > 0$ such that \eqref{eq:binf} reads
\[
\frac{1}{4} W_{\infty}  \log\pa{\frac{\tau}{t_1}} \le C+\pa{\sup\limits_{t\ge t_1}\norm{\dot{u}(t)}}\pa{\log \tau\int_\tau^\infty\|g(s)\| ds+\int_{t_1}^\tau \|g(t)\| \log t dt} .
\]
Observe that $\lim_{\tau \to +\infty} \int_\tau^\infty\|g(s)\| ds = 0$ since $g$ is integrable. Then, divide the last inequality by $\log\pa{\frac{\tau}{t_1}}$ and let $\tau\to\infty$. According to  Lemma~\ref{basic-int},  we get that $W_\infty \le 0$. Thus $W_\infty=0$, hence the contradiction. 
{The last statement,  $\lim\limits_{t\to+ \infty} \norm{\dot{x}(t)}=0$, is obtained by following an argument similar to the one above, which  now  uses the perturbed version of the classical energy function, namely
\[
W_0 (t) \eqdef \frac{1}{2} \norm{\dot{x}(t)}^2+(f(x(t))-\bar{f}) - \int_t^{+\infty} \dotp{\dot{x}(\tau)}{g(\tau)} d\tau .
\]
We do not detail this proof for the sake of brevity.
Then, according to $\nabla f (x(t)) = \frac{1}{\beta}\left(\dot{u}(t)-\dot{x}(t)\right)$, we obtain  the  convergence of $\nabla f (x(t))$ to zero.
}
\qed
\end{proof}

\begin{remark}
The result in \cite[Lemma~4.1]{attouch2018fast} is a particular case of Theorem~\ref{lem:smoothexplicitestimates} when $\beta=0$. Theorem~\ref{lem:smoothexplicitestimates} is also a generalization of \cite[Theorem~1.3 and Proposition~1.5]{APR1} to the perturbed case. Performing this generalization necessitates a different Lyapunov function together with particular novel estimates in the course of derivation.  
\end{remark}

%%%%%%%%%%%%%%%%%%%%%%%%%%%%%%%%%%%%%%%%%%%%%%%%%%%%%%%%%%%%%%%%%%%%%%%%%%%%%%%%%%%%%%%%%%%%%%%%%%%%%%%%%%%%%%%
\subsubsection{Fast convergence rates}
We now move on to showing fast convergence of the objective. For this, we will need to strengthen the integrability assumption on the errors. We will denote in this section the two functions
\begin{equation}\label{def:w_delta}
w(t) \eqdef 1-\frac{\beta}{t} \qandq \delta(t) \eqdef t^2 w(t) .
\end{equation}
{
\begin{theorem}\label{fast_conv_smooth}
Let  $f \in \cC^2(\cH)$, $e\in\cC^1([t_0, +\infty[, \cH)$. Suppose that the damping parameters satisfy $\alpha > 3$, $\beta > 0$.   Suppose that   $\DS{\int_{t_0}^{+\infty} t \|e(t)\| dt<+\infty}$ and $\DS{\int_{t_0}^{+\infty} t \|\dot e(t)\| dt<+\infty}$.   Then, for any  solution trajectory  $x$ of~\eqref{eq:origode_a} the following holds:
\begin{enumerate}[label={\rm (\roman*)}]
\item \label{item:fast_conv_smooth1}
$
\DS{f(x(t))-\inf_{\cH} f = o\pa{\frac{1}{t^2}}}  \mbox{ as } t \to +\infty.
$
\item \label{item:fast_conv_smooth2}
$
\DS{\int_{t_0}^{+\infty} t^2 \|\nabla f(x(t))\|^2 dt < +\infty} .
$
\item \label{item:fast_conv_smooth3}
$
\DS{\int_{t_0}^{+\infty} t\pa{f(x(t))-\min_{\cH} f} dt < +\infty} .
$
\item \label{item:fast_conv_smooth4a}
for any $x^\star \in S$,
$
\DS{\int_{t_0}^{+\infty} t\dotp{\nabla f(x(t))}{x(t)-x^\star} dt < +\infty} .
$
\item \label{item:fast_conv_smooth4b}
$
\DS{\int_{t_0}^{+\infty} t\|\dot{x}(t)\|^2 dt < +\infty}.
$ 
\item \label{item:fast_conv_smooth4c}
$
\DS{\norm{\dot{x}(t)} = o\pa{t^{-1}}}  \mbox{ as } t \to +\infty.
$
\end{enumerate}
\end{theorem}
}
%The assumption $t_0 > \beta\frac{\alpha-2}{\alpha-3}$ is somehow artificial and not needed for the asymptotic claims. The other estimates also hold without it replacing $t_0$ by $t_1=\max(t_0,\beta\frac{\alpha-2}{\alpha-3})$. This is also the case for other statements in the rest of the section.

\begin{proof}
{Following an argument similar to that of Theorem \ref{lem:smoothexplicitestimates}, since our analysis is asymptotic, there is no restriction in assuming that $t_0 > \frac{\beta(\alpha-2)}{\alpha-3}$.
This gives $w(t) \geq \frac{1}{\alpha-2} > 0$ for all $t \geq t_0$.} Define,
\[
v(t)\eqdef (\alpha-1)(x(t)-x^\star)+t(\dot x(t)+\beta\nabla f(x(t))).
\]
Take $T >t_0$, and define for all $t_0 \leq t \leq T$
\[
\cE(t) \eqdef \delta(t)(f(x(t))-f(x^\star))+\frac{1}{2}\|v(t)\|^2-\int_t^T \tau\dotp{v(\tau)}{g(\tau)} d\tau.
\]
This is a well-defined differentiable function. Taking its derivative in time yields
\[
\dot{\cE}(t)=\dot{\delta}(t)(f(x(t))-f(x^\star))+\delta(t)\dotp{\nabla f(x(t))}{\dot{x}(t)}+\dotp{v(t)}{\dot{v}(t) + t g(t)} .
\]
From \eqref{eq:origode_a}, we have,
\begin{align*}
\dot{v}(t)
&=\alpha \dot{x}(t)+\beta \nabla f(x(t))+t\pa{\ddot{x}(t)+\beta \nabla^2 f(x(t))\dot{x}(t)} \\
&= \alpha \dot{x}(t)+\beta \nabla f(x(t))+t\pa{-\frac{\alpha}{t}\dot{x}(t)-\nabla f(x(t))-g(t)} \\
&= -t\pa{1-\frac{\beta}{t}}\nabla f(x(t))-t g(t) .
\end{align*}
{Let us inject this expression into the {scalar} product $\dotp{v(t)}{\dot{v}(t) + t g(t)}$. After developing and rearranging, and taking into account  the definition of $w$ and $\delta$ (see \eqref{def:w_delta}), we  obtain}
\begin{equation}\label{eq:dotExp}
\dot{\cE}(t) = \dot{\delta}(t) (f(x(t))-f(x^\star)) - (\alpha-1)t w(t) \dotp{\nabla f(x(t))}{x(t)-x^\star} - \beta\delta(t)\|\nabla f(x(t))\|^2 .
\end{equation}
Convexity of $f$ then yields
\begin{equation}\label{eq:dotEbnd}
\dot{\cE}(t)+\beta\delta(t)\|\nabla f(x(t))\|^2+\pa{(\alpha-1)t w(t)-\dot{\delta}(t)} (f(x(t))-f(x^\star))\le 0 .
\end{equation}
{By assumption on the parameters, we have for any $t \geq t_0 > \frac{\beta(\alpha-2)}{\alpha-3}$, 
\begin{equation}\label{eq:dotEbnd_bis}
(\alpha-1)t w(t)-\dot{\delta}(t)=t\left ((\alpha-3) -\frac{\beta}{t}(\alpha-2) \right) 
\geq ct
\end{equation}
with $c= (\alpha-3) -\frac{\beta}{t_0}(\alpha-2) > 0$. Thefore,} \eqref{eq:dotEbnd} implies that $\cE$ is non-increasing. In turn, {by definition of $\cE$}
\begin{equation}\label{eq:Ebnd}
\delta(t)\left(f(x(t))-f(x^\star)\right)+\frac{1}{2} \|v(t)\|^2 \le  C - \int_{t_0}^t \tau \dotp{v(\tau)}{g(\tau)}d\tau
\end{equation}
with $C=\delta(t_0)(f(x_0)-f(x^\star))+\frac{1}{2}\|v(t_0)\|^2$.
{ The above argument is valid for all $t_0 \leq t \leq T$ and arbitrary $T$, therefore for all $ t \geq t_0$.
Neglecting the nonnegative term  $\delta(t)\left(f(x(t))-f(x^\star)\right)$,  we infer from \eqref{eq:Ebnd} }
\[
\frac{1}{2}\|v(t)\|^2\le C+\int_{t_0}^t \|v(\tau)\| \pa{\tau\|g(\tau)\|} d\tau .
\]
Lemma~\ref{lem:BrezisA5} {(Gronwall lemma)} then gives 
\begin{equation}\label{eq:boundvexplicit}
\|v(t)\| \le \pa{2C}^{1/2}+\int_{t_0}^t \tau \pa{\|e(\tau)\|+\beta \|\dot{e}(\tau)\|} d\tau ,
\end{equation}
and thus 
\begin{equation}\label{eq:supboundv}
\sup\limits_{t \ge t_0} \|v(t)\|<+\infty. 
\end{equation}
{By reinjecting this inequality into \eqref{eq:Ebnd} we obtain}
\begin{equation}\label{eq:boundobjexplicit}
\frac{t^2}{\alpha-2}\pa{f(x(t))-f(x^\star)} \leq \delta(t)\pa{f(x(t))-f(x^\star)}
\leq C + \sup\limits_{t\ge t_0} \|v(t)\| \int_{t_0}^t  \pa{\tau\|e(\tau)\|+\beta \tau\|\dot{e}(\tau)\|}d\tau < +\infty
\end{equation}
hence proving 
$$
\DS{f(x(t))-\inf_{\cH} f = \cO \pa{\frac{1}{t^2}}}  \mbox{ as } t \to +\infty.
$$
{We will see a little later how to refine this estimate and go from a capital $\cO$ to a small $o$ to prove the statement  {\ref{item:fast_conv_smooth1}}.
Then}, integrating~\eqref{eq:dotEbnd}, and using the fact that $\cE(t)$ is bounded from below by \eqref{eq:supboundv} and the assumptions on the errors, we get
\[
\beta \int_{t_0}^{+\infty} t^2 w(t)\|\nabla f(x(t))\|^2 dt \leq C ,
\]
and
\[
c \int_{t_0}^{+\infty} t (f(x(t))-f(x^\star)) dt \leq \int_{t_0}^{+\infty} \pa{(\alpha-1)tw(t)+\dot{\delta}(t)} (f(x(t))-f(x^\star))dt \leq C ,
\]
for some constant $C > 0$. This shows the integral estimates {\ref{item:fast_conv_smooth2}} and {\ref{item:fast_conv_smooth3}}.

\noindent Let us turn to statement
~\ref{item:fast_conv_smooth4a}. We embark from \eqref{eq:dotExp} to write, for some $\rho \in ]0,1[$ to be chosen shortly,
\begin{align*}
&\dot{\cE}(t) 
= \dot{\delta}(t) (f(x(t))-f(x^\star)) - (1-\rho)(\alpha-1)t w(t) \dotp{\nabla f(x(t))}{x(t)-x^\star} \\
&\qquad\qquad - \rho(\alpha-1)t w(t) \dotp{\nabla f(x(t))}{x(t)-x^\star} - \beta\delta(t)\|\nabla f(x(t))\|^2 \\
&\leq -\pa{(1-\rho)(\alpha-1)t w(t)-\dot{\delta}(t)} (f(x(t))-f(x^\star)) - \rho(\alpha-1)t w(t) \dotp{\nabla f(x(t))}{x(t)-x^\star} - \beta\delta(t)\|\nabla f(x(t))\|^2.
\end{align*}
To conclude, it remains to check that $\pa{(1-\rho)(\alpha-1)t w(t)-\dot{\delta}(t)}$ is non-negative. 
{Since $t_0 > \frac{\beta(\alpha-2)}{\alpha-3}$, we deduce by a continuity argument the existence of some $\eps >0$ such that 
$t_0 > \frac{\beta(\alpha-2-\eps)}{\alpha-3-\eps}$.}
Then, take $\rho=\eps/(\alpha-1) \in ]0,1[$. In view of the assumption on the parameters, we have
\begin{multline*}
(1-\rho)(\alpha-1)t w(t)-\dot{\delta}(t) 
= (\alpha-1-\eps)t w(t) - \dot{\delta}(t) = t((\alpha-3-\eps) w(t) - t\dot{w}(t)) \\
= t\pa{(\alpha-3-\eps) - \frac{(\alpha-2-\eps)\beta}{t}} \geq t_0\pa{(\alpha-3-\eps) - \frac{(\alpha-2-\eps)\beta}{t_0}} \geq 0.
\end{multline*}
For claim~\ref{item:fast_conv_smooth4b}, we multiply~\eqref{eq:origode_a} by $t^2\dot{x}(t)$ to get
\begin{equation*}
t^2 \dotp{\ddot{x}(t)}{\dot{x}(t)}+\alpha t\norm{\dot{x}(t)}^2 + t^2 \beta \dotp{\dot{x}(t)}{\nabla^2 f(x(t))\dot{x}(t)} + t^2 \dotp{\nabla f(x(t))}{\dot{x}(t)} + t^2\dotp{g(t)}{\dot{x}(t)} = 0 .
\end{equation*}
With the chain rule, Cauchy-Schwarz inequality and convexity of $f$, we obtain
\begin{equation}\label{eq:derivodea}
\frac{1}{2}t^2 \frac{d}{dt} \norm{\dot{x}(t)}^2+\alpha t \norm{\dot{x}(t)}^2 
%+ t^2 \beta \dotp{\dot{x}(t)}{\nabla^2 f(x(t))\dot{x}(t)} 
+ t^2 \frac{d}{dt} (f(x(t))-\bar{f}) \le \|tg(t)\|\|t\dot{x}(t)\| .
\end{equation}
Integrating by parts on $[t_0,t]$ we get,
\begin{equation}\label{eq:mainbound}
\frac{t^2}{2}\norm{\dot{x}(t)}^2+(\alpha-1)\int_{t_0}^t s\|\dot{x}(s)\|^2 ds 
%+ \beta\int_{t_0}^t s^2 \beta \dotp{\dot{x}(s)}{\nabla^2 f(x(s))\dot{x}(s)}ds
\le C_0+2\int_{t_0}^t s(f(x(s))-\bar{f}) ds+\int_{t_0}^t \|s g(s)\|\|s \dot{x}(s)\|ds
\end{equation}
for some non-negative constant $C_0$, where we have used claim~{\ref{item:fast_conv_smooth1}} of Theorem~\ref{fast_conv_smooth}. Now by claim~{\ref{item:fast_conv_smooth3}} of Theorem~\ref{fast_conv_smooth} and ignoring the non-negative terms since $\alpha > 1$, we obtain 
\begin{equation*}
\frac{1}{2}\norm{t\dot{x}(t)}^2 \le C_1 + \int_{t_0}^t \|s g(s)\|\|s \dot{x}(s)\|ds ,
\end{equation*}
for another non-negative constant $C_1$. {Applying Lemma~\ref{lem:BrezisA5} again  then gives}
\begin{equation}\label{eq:boundtdotx}
\sup_{t \geq t_0} t\|\dot{x}(t)\| < +\infty .
\end{equation}
Using this in \eqref{eq:mainbound}, {we also get that}
\begin{equation}\label{eq:intbounds}
\begin{aligned}
\int_{t_0}^{+\infty} t\|\dot{x}(t)\|^2 dt < +\infty .
%\qandq \int_{t_0}^{+\infty}  t^2 \dotp{\dot{x}(t)}{\nabla^2 f(x(t))\dot{x}(t)} dt < +\infty .
\end{aligned}
\end{equation}
%Moreover, \eqref{eq:supboundv}, \eqref{eq:boundtdotx} and claim~{\ref{item:fast_conv_smooth2}} of Theorem~\ref{fast_conv_smooth} entail 
%\[
%\sup\limits_{t\ge t_0} \|x(t)\| < +\infty .
%\]
We finally turn to statement~\ref{item:fast_conv_smooth4c}. We embark from \eqref{eq:derivodea}, use \eqref{eq:boundtdotx}, and integrate on $[s,t]$ to see that
\begin{multline*}
t^2\pa{\frac{1}{2}\norm{\dot{x}(t)}^2+(f(x(t))-\bar{f})} - s^2\pa{\frac{1}{2}\norm{\dot{x}(s)}^2+(f(x(s))-\bar{f})} \\
+ (\alpha-1)\int_{s}^t \tau\|\dot{x}(\tau)\|^2 d\tau - 2\int_{s}^t \tau(f(x(\tau))-\bar{f}) d\tau - C\int_{s}^t \|\tau g(\tau)\|d\tau \leq 0,
\end{multline*}
where $C = \sup_{t \geq t_0} t\|\dot{x}(t)\|$. This means that the function 
\begin{equation*}
\cG(t) = t^2\pa{\frac{1}{2}\norm{\dot{x}(t)}^2+(f(x(t))-\bar{f})} 
+ (\alpha-1)\int_{t_0}^t \tau\|\dot{x}(\tau)\|^2 d\tau 
- 2\int_{t_0}^t \tau(f(x(\tau))-\bar{f}) d\tau - C\int_{t_0}^t \|\tau g(\tau)\|d\tau
\end{equation*}
is non-increasing on $[t_0,+\infty[$. Since it is bounded from below by assumption on the errors and claim~{\ref{item:fast_conv_smooth3}}, $\lim_{t \to +\infty} \cG(t)$ exists. This together with assertions~{\ref{item:fast_conv_smooth3} }and \ref{item:fast_conv_smooth4b} shows that the limit
\[
0 \leq L \eqdef \lim_{t \to +\infty} t^2\pa{\frac{1}{2}\norm{\dot{x}(t)}^2+(f(x(t))-\bar{f})}
\] 
exists. Suppose that $L > 0$. Then, there exists $s \geq t_0$ such that
\begin{equation*}
\int_{t_0}^{+\infty}\pa{\frac{t}{2}\norm{\dot{x}(t)}^2+t(f(x(t))-\bar{f})} dt \geq \int_{s}^{+\infty}t^2\pa{\frac{1}{2}\norm{\dot{x}(t)}^2+(f(x(t))-\bar{f})} t^{-1} dt 
\geq \int_{s}^{+\infty}\frac{L}{2t} dt = +\infty ,
\end{equation*}
leading to a contradiction with claims {\ref{item:fast_conv_smooth3}} and  \ref{item:fast_conv_smooth4b}. {This proves ~\ref{item:fast_conv_smooth4c} and completes the proof of \ref{item:fast_conv_smooth1} with small $o$ instead of capital $\mathcal O$}.\qed
\end{proof}
\begin{remark}
The first three claims (resp. fourth claim) of Theorem~\ref{fast_conv_smooth} are a non-trivial generalization of \cite[Theorem~3]{attouch2019first} (resp. \cite[Theorem~2.1]{attouchiterates2021}) to the perturbed case. The presence of perturbations necessitates a careful analysis of several bounds and new estimates to handle the presence of errors and eventually preserve the convergence rates.
\end{remark}
\begin{remark}
The choice of the viscous damping parameter $\alpha$ is important for optimality of the convergence rates obtained. For the subcritical case $\alpha \leq 3$, $\beta=0$ and $e \equiv 0$, it has  been shown by \cite{AAD} and \cite{ACR-subcrit} that the convergence rate of the objective values is $\displaystyle{\cO\pa{t^{-\frac{2\alpha}{3}}}}$, and these rates are optimal, that is, they can be attained, or approached arbitrarily closely. For $\alpha \geq 3$, the optimal rate $\displaystyle{\cO\pa{t^{-2}}}$ is achieved for the function $f(x) = \|x\|^{r}$ with $r \to +\infty$ \cite{ACPR}, and for $\alpha  < 3 $, the optimal rate $\displaystyle{\cO\pa{t^{-\frac{2\alpha}{3}}}}$ is achieved by taking $f(x) = \|x\|$ \cite{AAD}. Theorem~\ref{fast_conv_smooth} is consistent with these optimality results. The condition $\alpha > 3$ is important to get the asymptotic rate $o(1/t^2)$.
\end{remark}

%%%%%%%%%%%%%%%%%%%%%%%%%%%%%%%%%%%%%%%%%%%%%%%%%%%%%%%%%%%%%%%%%%%%%%%%%%%%%%%%%%%%%%%%%%%%%%%%%%%%%%%%%%%%%%%
\subsubsection{Convergence of the trajectories}
We complete our analysis by showing weak convergence of the trajectories. 
%The proof follows similar lines as the one of~\cite[Theorem 4.4]{attouch2018fast} using Opial's lemma.
\begin{theorem}\label{T:weak_convergence-smooth}
{Assume that $e\in \cC^1([t_0, +\infty[; \cH)$ with $\DS{\int_{t_0}^{+\infty} t \|e(t)\| dt<+\infty}$ and $\DS{\int_{t_0}^{+\infty} t \|\dot e(t)\| dt<+\infty}$. Let $x(t)$ be a solution trajectory to~\eqref{eq:origode_a} for $\alpha > 3$ and $\beta > 0$. Then $x(t)$ converges weakly to a minimizer of $f$.}
\end{theorem}
\begin{proof}
%From~\eqref{eq:dotEbnd} we know that,
%\[
%E(t)+\int_{t_0}^t \brac{\frac{\alpha-1}{t}\delta(\tau)+\dot{\delta}(\tau)} (f(x(\tau))-f(x^\star)) d\tau\le E(t_0)
%\]
%and so,
%\[
%\int_{t}^\infty \tau\dotp{v(\tau),e(\tau)+\beta \dot{e}(\tau)} d\tau+\int_{t_0}^t \brac{\frac{\alpha-1}{t}\delta(\tau)+\dot{\delta}(\tau)} (f(x(\tau))-f(x^\star)) d\tau \le E(t_0)
%\]
%and thus from,
%\[
%\int_{t_0}^t \brac{\frac{\alpha-1}{t}\delta(\tau)+\dot{\delta}(\tau)} (f(x(\tau))-f(x^\star)) d\tau \le E(t_0)+\sup\limits_{t\ge t_0}\|v(t)\|\inf_{t_0}^\infty \tau(\|e(\tau)\|+\beta \|\dot{e}(\tau)\|d\tau <\infty
%\]
%conclude that,
%\[
%\int_{t_0}^t \brac{\frac{\alpha-1}{t}\delta(\tau)+\dot{\delta}(\tau)} (f(x(\tau))-f(x^\star)) d\tau \le \infty
%\]
%Since we are interested in an aymptotic result, we can suppose that $t > t_1 \eqdef \max\pa{t_0,\beta\frac{\alpha-2-\eps}{\alpha-3-\eps}}$.

Keeping in mind that the goal is to apply Opial's Lemma (see Lemma~\ref{Opial}), we will now show that $\lim_{t \to +\infty}\|x(t)-x^\star\|$ exists. 
{Following an argument similar to that of Theorem \ref{lem:smoothexplicitestimates} and \ref{fast_conv_smooth}, since our analysis is asymptotic, there is no restriction in assuming that $t_0 > \frac{\beta(\alpha-2)}{\alpha-3}$. Hence the existence of $\eps >0$ such that
$t_0 \geq \frac{\beta(\alpha-2-\eps)}{\alpha-3-\eps}$ for some $\eps \in ]0,\alpha-3[$.}
Recall the Lyapunov function $\cE$ from the proof of Theorem~\ref{fast_conv_smooth}, and define its generalized version
\begin{equation}\label{eq:Eeps}
\cE_\eps(t) \eqdef 
\pa{\delta(t)+\eps\beta t}(f(x(t))-f(x^\star))+\frac{1}{2}\|v_\eps(t)\|^2\\
+\frac{\eps(\alpha-1-\eps)}{2}\norm{x(t)-x^\star}^2-\int_t^T \tau\dotp{v_\eps(\tau)}{g(\tau)} d\tau ,
\end{equation}
where 
\[
v_\eps(t) = (\alpha-1-\eps)(x(t)-x^\star)+t(\dot x(t)+\beta\nabla f(x(t))) .
\]
One can check, arguing as for $\cE$, that for $t_0 \leq t\leq T$
\begin{equation*}
\dot{\cE}_\eps(t) = \pa{\dot{\delta}(t)+\eps\beta}(f(x(t))-f(x^\star)) - (\alpha-1-\eps)t w(t) \dotp{\nabla f(x(t))}{x(t)-x^\star} 
- \beta\delta(t)\|\nabla f(x(t))\|^2
- \eps t \norm{\dot{x}(t)}^2.
\end{equation*}
Convexity of $f$ then entails
\begin{equation*}
\dot{\cE}_\eps(t) + \pa{(\alpha-1-\eps)t w(t) - \dot{\delta}(t)}(f(x(t))-f(x^\star)) +\beta\delta(t)\|\nabla f(x(t))\|^2
+ \eps t \norm{\dot{x}(t)}^2 \leq 0 .
\end{equation*}
The assumption on the parameters gives
\begin{equation*}
(\alpha-1-\eps)t w(t) - \dot{\delta}(t) = t((\alpha-3-\eps) w(t) - t\dot{w}(t)) \geq t_0\pa{(\alpha-3-\eps) - \frac{(\alpha-2-\eps)\beta}{t_0}} \geq 0.
\end{equation*}
Thus, ignoring the non-negative terms in this inequality entails that $\cE_\eps(\cdot)$ is a decreasing function on $[t_0,T[$. 
{According to the boundedness of  $\cE_\eps(\cdot)$,  an argument similar to that developed in Theorem 5 gives that 
\begin{equation}\label{eq:supboundv_epsilon}
\sup\limits_{t \ge t_0} \|v_\eps(t)\|<+\infty, 
\end{equation}
with a bound which is independent of $\eps$ and $T$.
Comparing with \eqref{eq:supboundv} gives 
\begin{equation}\label{eq:supbound_x}
\sup\limits_{t \ge t_0} \|x(t)\|<+\infty, 
\end{equation}
where we use that the previous argument is valid for arbitrary $t\leq T$, hence for all $t\geq t_0$.} 
As a consequence, the energy functions $\cE(\cdot)$ and $\cE_{\eps}(\cdot)$ with $T=+\infty$ are well-defined on $[t_0,+\infty[$, and are then Lyapunov functions for the dynamical system \eqref{eq:origode_a}. Both $\cE(t)$ and $\cE_{\eps}(t)$ thus have limits as $t \to +\infty$, and so does their difference
\begin{multline*}
\cE_{\eps}(t) - \cE(t) = \eps\beta t(f(x(t))-f(x^\star))
-\frac{\eps(\alpha-1)}{2}\norm{x(t)-x^\star}^2 - \eps t\dotp{\dot{x}(t)}{x(t)-x^\star} \\ 
- \eps t\dotp{\nabla f(x(t))}{x(t)-x^\star} + \eps \int_t^{+\infty} \tau\dotp{x(\tau)-x^\star}{g(\tau)} d\tau .
\end{multline*}
{By Theorem~\ref{fast_conv_smooth}\ref{item:fast_conv_smooth1}, the first term converges to $0$ as $t \to +\infty$. By the integrability assumptions on the errors and boundedness of $x(t)$ (see ~\eqref{eq:supbound_x}), the last term also converges to $0$ as $t \to +\infty$. We have then shown that the limit as $t$ goes to infinity of 
\[
p(t) \eqdef \frac{\alpha - 1}{2}\norm{x(t) -x^\star}^{2} + t \dotp{\dot{x}(t)}{x(t) -x^\star} + t \dotp{\nabla f(x(t))}{x(t) -x^\star} 
\]
exists. Set
\[
q(t)\eqdef \frac{\alpha - 1}{2}\norm{x(t) -x^\star}^{2} + (\alpha - 1)\int_{t_0}^t \dotp{\nabla f(x(s))}{x(s) - x^\star} ds. 
\]
We obviously have
\[
p(t)= q(t) + \frac{t}{\alpha - 1}\dot{q}(t)  - (\alpha - 1)\int_{t_0}^t\dotp{\nabla f(x(s))}{x(s) -x^\star} ds.
\]
By Theorem~\ref{fast_conv_smooth}\ref{item:fast_conv_smooth4a}, and since $\dotp{\nabla f(x(s))}{x(s) -x^\star}$ is non-negative, we have that
\begin{equation}\label{eq:limiprodgradf}
\lim_{t\to +\infty}  \int_{t_0}^t \dotp{\nabla f(x(s))}{x(s) -x^\star} ds 
\end{equation}
exists. Overall, we have shown that 
\[
\lim_{t \to +\infty}\pa{q(t) + \frac{t}{\alpha-1}\dot{q}(t)} 
\]
exists. Since $\alpha > 1$, it follows from \cite[Lemma~7.2]{APR1} that $\lim_{t \to +\infty} q(t)$ exists. Using again \eqref{eq:limiprodgradf}, we deduce that $\lim_{t \to +\infty}\norm{x(t)-x^\star}$ exists for any $x^\star \in S$.} From claim {\ref{item:fast_conv_smooth1}} of Theorem~\ref{fast_conv_smooth} (see also Lemma~\ref{lem:smoothexplicitestimates}{\ref{lem:smoothexplicitestimates4}}), it follows that
for any sequence $\pa{x(t_n)}_{n \in \N}$ which converges weakly to, say, $\bar{x}$, we have
\[
f(\bar{x}) \leq \liminf_{n \to +\infty} f(x(t_n)) = \lim_{t \to +\infty} f(x(t)) = \bar{f} ,
\]
\ie, $\bar{x} \in S$. Consequently, all the conditions of Lemma~\ref{Opial} are satisfied, hence the weak convergence of the trajectories. \qed
\end{proof}

\begin{remark}
In \cite[Theorem~2.2]{attouchiterates2021}, the authors proved weak convergence of the trajectory for the perturbation-free system \eqref{eq:origode}. Theorem~\ref{T:weak_convergence-smooth} shows that weak convergence is preserved under perturbations provided that they verify reasonable integrability results. Again, the proof  necessitates new estimates and bounds to cope with the presence of errors.
\end{remark}

\begin{remark}
The condition $\alpha > 3$ is known to play an important role to show that each trajectory converges weakly to a minimizer. The case $\alpha = 3$, which corresponds to Nesterov's historical algorithm when $\beta=0$ and $e \equiv 0$, is critical. In fact, even for those inertial systems with $\alpha = 3$, convergence of the trajectories remains an open problem (except in one dimension where it holds as shown in \cite{ACR-subcrit}).
\end{remark}

\subsection{Implicit Hessian Damping}\label{s:approxconv}

{We now turn to the second-order ODE \eqref{eq:odetwo_a} where $f \in \cC^1(\cH)$, {$e \in \cC([t_0,+\infty[)$} and $\beta(t)=\gamma+\frac{\beta}{t}$, $\gamma,\beta \geq 0$. Let us denote for brevity $\bar{f} \eqdef \inf_{\cH} f$. Given $x^\star \in S$, we consider the function}
\begin{multline}\label{eq:lyap}
\cE(t) = a(t)\pa{f\pa{x(t)+\beta(t)\dot{x}(t)}-\bar{f}}+\frac{1}{2}\|b(t)(x(t)-x^\star)+c(t)\dot{x}(t)\|^2
+\frac{d(t)}{2}\|x(t)-x^\star\|^2 \\ -\int_t^{+\infty} c(\tau)\dotp{b(\tau)(x(\tau)-x^\star)+c(\tau) \dot{x}(\tau)}{e(\tau)} d\tau
-\int_t^{+\infty} a(\tau)\beta(\tau) \dotp{\nabla f\pa{x(\tau)+\beta(\tau)\dot{x}(\tau)}}{e(\tau)} d\tau 
\end{multline}
parametrized by some functions $a(t)$, $b(t)$, $c(t)$ and $d(t)$ to be specified later.
%where $a(t)$ and $d(t)$ are non-negative parameters.

%%%%%%%%%%%%%%%%%%%%%%%%%%%%%%%%%%%%%%%%%%%%%%%%%%%%%%%%%%%%%%%%%%%%%%%%%%%%%%%%%%%%%%%%%%%%%%%%%%%%%%%%%%%%%%%
\subsubsection{Lyapunov function}
We first show that for proper choices of $(a(t),b(t),c(t),d(t))$ as a function of the problem parameters $(\alpha,\gamma,\beta)$, $\cE$ can serve as a Lyapunov function for \eqref{eq:odetwo_a}. We will denote for short $\alpha(t)=\frac{\alpha}{t}$.
\begin{lemma}\label{lem:decreaselyap}
Assume that {$f \in \cC^1(\cH)$, $e \in \cC([t_0,+\infty[)$, and}
\begin{equation}\label{eq:conditionsscalarfunc}
\begin{cases}
\dot{a}(t)-b(t)c(t) &\le 0,\\
-a(t)\beta(t) &\leq 0,\\
-a(t)\alpha(t)\beta(t)+a(t)\dot{\beta}(t)+a(t)-c(t)^2+b(t)c(t)\beta(t) &= 0,\\
\dot{b}(t)b(t)+\frac{\dot{d}(t)}{2} &\le 0,\\
\dot{b}(t)c(t)+b(t)(b(t)+\dot{c}(t)-c(t)\alpha(t))+d(t) &=0,\\
c(t)(b(t)+\dot{c}(t)-c(t)\alpha(t)) &\le 0 .
\end{cases}
\end{equation}
Then
\begin{eqnarray}\label{eq:edotfour}
\dot{\cE}(t) &\le& (\dot{a}(t)-b(t)c(t))(f(x(t)+\beta(t)\dot{x}(t))-\bar{f})-a(t)\beta(t)\|\nabla f(x(t)+\beta(t)\dot{x}(t))\|^2  \nonumber \\
 &&+\pa{\dot{b}(t)b(t)+\frac{\dot{d}(t)}{2}}\|x(t)-x^\star\|^2+c(t)(b(t)+\dot{c}(t)-c(t)\alpha(t))\norm{\dot{x}(t)}^2 \le 0 .
\end{eqnarray}
\end{lemma}
\begin{proof}
{Recall that since $f \in \cC^1(\cH)$ and $e \in \cC([t_0,+\infty[)$, \eqref{eq:odetwo_a} has a unique  classical global solution $x$; see paragraph after \eqref{eq:fos5_implicit}}. {We now proceed as in the proof of  Theorems~\ref{lem:smoothexplicitestimates} and \ref{fast_conv_smooth}, and first consider the function $\cE_T$ where the integrals involving the error terms are calculated on $[t,T]$, $T<+\infty$. This shows that
$\cE$ is well-posed under our assumptions. We can then compute the  time derivative of $\cE$  and use the chain rule to get}
\begin{eqnarray}
\dot{\cE}(t) &=& \dot{a}(t)\pa{f\pa{x(t)+\beta(t)\dot{x}(t)}-\bar{f}}+
a(t)\dotp{\nabla f \pa{x(t)+\beta(t)\dot{x}(t)}}{\dot{x}(t)+\dot{\beta}(t)\dot{x}(t)+\beta(t)\ddot{x}(t)} \nonumber \\
&&+ \dotp{(b(t)+\dot{c}(t))\dot{x}(t)+c(t)\ddot{x}(t)+\dot{b}(t)(x(t)-x^\star)}{b(t) (x(t)-x^\star)+c(t)\dot{x}(t)} \nonumber \\
&&+ \frac{\dot{d}(t)}{2}\|x(t)-x^\star\|^2+d(t)\dotp{\dot{x}(t)}{x(t)-x^\star} \nonumber \\
&&+ c(t)\dotp{b(t)(x(t)-x^\star)+c(t) \dot{x}(t)}{e(t)}+a(t)\beta(t)\dotp{\nabla f\pa{x(t)+\beta(t)\dot{x}(t)}}{e(t)}\label{eq:edot}
\end{eqnarray}
Using \eqref{eq:odetwo_a} in the second term of \eqref{eq:edot}, we get
\begin{multline}\label{eq:term1}
a(t)\dotp{\nabla f \pa{x(t)+\beta(t)\dot{x}(t)}}{\dot{x}(t)+\dot{\beta}(t)\dot{x}(t)+\beta(t)\ddot{x}(t)} \\ 
= a(t)\dotp{\nabla f\pa{x(t)+\beta(t)\dot{x}(t)}}{\pa{1+\dot{\beta}(t)-\alpha(t) \beta(t)} \dot{x}(t)-\beta(t)\nabla f\pa{x(t)+\beta(t)\dot{x}(t)}-\beta(t)e(t)} \\
= -a(t) \beta(t) \norm{\nabla f\pa{x(t)+\beta(t)\dot{x}(t)}}^2+\pa{1+\dot{\beta}(t)-\alpha(t) \beta(t)} a(t) \dotp{\nabla f\pa{x(t)+\beta(t)\dot{x}(t)}}{\dot{x}(t)}
\\-\beta(t)a(t)\dotp{\nabla f\pa{x(t)+\beta(t)\dot{x}(t)}}{e(t)} .
\end{multline}
We expand the third term in \eqref{eq:edot} as
\begin{eqnarray}
&&\dotp{(b(t)+\dot{c}(t))\dot{x}(t)+c(t)\ddot{x}(t)+\dot{b}(t)(x(t)-x^\star)}{b(t) (x(t)-x^\star)+c(t)\dot{x}(t)} \nonumber \\ 
&& = 
\dotp{(b(t)+\dot{c}(t)-c(t)\alpha(t))\dot{x}(t)-c(t)\nabla f\pa{x(t)+\beta(t)\dot{x}(t)}}{b(t) (x(t)-x^\star)+c(t)\dot{x}(t)}  \nonumber \\
&&\quad +\dotp{-c(t)e(t)+\dot{b}(t)(x(t)-x^\star)}{b(t) (x(t)-x^\star)+c(t)\dot{x}(t)} \nonumber \\
&&= -c(t)\dotp{\nabla f(x(t)+\beta(t)\dot{x}(t))}{b(t)(x(t)-x^\star)+c(t)\dot{x}(t)} \nonumber \\
&&\quad +\dotp{\dot{b}(t)(x(t)-x^\star)}{b(t)(x(t)-x^\star)+c(t)\dot{x}(t)} \nonumber\\ 
&&\quad +c(t)(b(t)+\dot{c}(t)-c(t)\alpha(t))\norm{\dot{x}(t)}^2-c(t)\dotp{e(t)}{b(t)(x(t)-x^\star)+c(t)\dot{x}(t)} . \label{eq:term2}
\end{eqnarray}
Plugging~\eqref{eq:term1} and~\eqref{eq:term2} into~\eqref{eq:edot}, we get,
\begin{eqnarray}
\dot{\cE}(t) &=& \dot{a}(t)\pa{ f\pa{x(t)+\beta(t)\dot{x}(t)}-\bar{f}}-a(t) \beta(t) \norm{\nabla f\pa{x(t)+\beta(t)\dot{x}(t)}}^2 \nonumber \\
 &+&c(t)(b(t)+\dot{c}(t)-c(t)\alpha(t))\norm{\dot{x}(t)}^2 + \pa{\dot{b}(t)b(t)+\frac{\dot{d}(t)}{2}}\|x(t)-x^\star\|^2 \nonumber\\
&+&\pa{b(t)^2+b(t)\dot{c}(t)+\dot{b}(t)c(t)-b(t)c(t)\alpha(t)+d(t)}\dotp{\dot{x}(t)}{x(t)-x^\star} \nonumber \\
&+& \pa{-a(t)\alpha(t)\beta(t)+a(t)\dot{\beta}(t)+a(t)-c(t)^2} \dotp{\nabla f\pa{x(t)+\beta(t)\dot{x}(t)}}{\dot{x}(t)}\nonumber
\\
&-& b(t)c(t)\dotp{\nabla f\pa{x(t)+\beta(t)\dot{x}(t)}}{x(t)-x^\star} .\label{eq:edottwo}
\end{eqnarray}
Since,
\begin{equation*}
\dotp{\nabla f(x(t)+\beta(t)\dot{x}(t))}{x(t)-x^\star} = \dotp{\nabla f(x(t)+\beta(t)\dot{x}(t))}{x(t)+\beta(t)\dot{x}(t)-x^\star} \\
-\dotp{\nabla f(x(t)+\beta(t)\dot{x}(t))}{\beta(t)\dot{x}(t)}
\end{equation*}
and using the convex (sub)differential inequality on $f$, 
% \[
% \dotp{\nabla f(x(t)+\beta(t)\dot{x}(t))}{x(t)+\beta(t)\dot{x}(t)-x^\star} \ge f(x(t)+\beta(t)\dot{x}(t))-\bar{f}
% \]
we can write
\begin{multline*}
-b(t)c(t)\dotp{\nabla f(x(t)+\beta(t)\dot{x}(t)}{x(t)-x^\star} \\ 
\le -b(t)c(t)( f(x(t)+\beta(t)\dot{x}(t))-\bar{f})+b(t)c(t)\beta(t)\dotp{\nabla f(x(t)+\beta(t)\dot{x}(t))}{\dot{x}(t)}
\end{multline*}
and we arrive at
\begin{eqnarray}
&&\dot{\cE}(t) \le (\dot{a}(t)-b(t)c(t))(f(x(t)+\beta(t)\dot{x}(t))-\bar{f})-a(t)\beta(t)\|\nabla f(x(t)+\beta(t)\dot{x}(t))\|^2   \nonumber \\
&& \; +(-a(t)\alpha(t)\beta(t)+a(t)\dot{\beta}(t)+a(t)-c(t)^2+b(t)c(t)\beta(t))\dotp{\nabla f(x(t)+\beta(t)\dot{x}(t))}{\dot{x}(t)} \nonumber\\
&&\; +\pa{\dot{b}(t)b(t)+\frac{\dot{d}(t)}{2}}\|x(t)-x^\star\|^2 \nonumber\\
&& \;+(b(t)^2+b(t)\dot{c}(t)+\dot{b}(t)c(t)-b(t)c(t)\alpha(t)+d(t))\dotp{\dot{x}(t)}{x(t)-x^\star}
+c(t)(b(t)+\dot{c}(t)-c(t)\alpha(t))\norm{\dot{x}(t)}^2 .\nonumber
%\label{eq:edotthree}
\end{eqnarray}
Thus,  conditions~\eqref{eq:conditionsscalarfunc}  guarantee that $\dot{\cE}(t)\le 0$, in particular they imply \eqref{eq:edotfour}. \qed
\end{proof}

Following the discussion of \cite[Remark~11]{alecsa2019extension}, in the rest of the section, we take
\begin{equation}\label{eq:odetwo_aparamchoice}
\begin{gathered}
\beta(t)=\gamma+\frac{\beta}{t}, \quad \gamma,\beta \geq 0, \\
b(t) \equiv b \in ]0,\alpha-1], \alpha > 1, \qquad c(t)=t \qandq d(t) \equiv b(\alpha-1-b). 
\end{gathered}
\end{equation}
Such a choice is reminescent of that in \eqref{eq:Eeps}. The choices of $d(t)$ and $b(t)$ comply with the fourth, fifth and sixth conditions of \eqref{eq:conditionsscalarfunc}. To satisfy the third condition, one has to take
\begin{equation}\label{eq:odetwo_choice_a}
a(t) = t^2\pa{1+\frac{(\alpha-b)\gamma t - \beta(\alpha+1-b)}{t^2-\alpha\gamma t - \beta(\alpha+1)}} .
\end{equation}
Clearly, for $t$ large enough, one has $a(t) \geq t^2$ and $\beta(t) \geq \gamma/2$. Thus, the second condition is in force. One can also verify that the first inequality is satisfied for $t$ large enough provided that $b > 2$ (and thus $\alpha > 3$) when $\gamma > 0$, and $b=2$ (with $\alpha=3$) when $\gamma=0$. 

%%%%%%%%%%%%%%%%%%%%%%%%%%%%%%%%%%%%%%%%%%%%%%%%%%%%%%%%%%%%%%%%%%%%%%%%%%%%%%%%%%%%%%%%%%%%%%%%%%%%%%%%%%%%%%%
\subsubsection{Fast convergence rates}
We start with the following boundedness properties.
\begin{lemma}\label{lem:boundconv}
Let
\[
E(t) = a(t)\pa{f\pa{x(t)+\beta(t)\dot{x}(t)}-\bar{f}}+\frac{1}{2}\|b(x(t)-x^\star)+t\dot{x}(t)\|^2
+\frac{b(\alpha-1-b)}{2}\|x(t)-x^\star\|^2 .
\]
Choose the parameters according to \eqref{eq:odetwo_aparamchoice}-\eqref{eq:odetwo_choice_a} with $\alpha > 3$, $\gamma > 0$. Define, for $t \geq t_0$,
\begin{equation}\label{eq:asm}
m(t) \eqdef 
%\begin{cases}
\max\pa{t,L|a(t)\beta(t)|,L|a(t)|\beta(t)^2}. %\\
%\max\pa{t,La(t)\beta(t),La(t)\beta(t)^2}\max\pa{1,1/t_0} & b=\alpha-1 . 
%\end{cases}
\end{equation}
%assuming these two cases regarding $d(t)$ are exhaustive. 
Assume that $\nabla f$ is $L$-Lipschitz continuous, {$e \in \cC([t_0,+\infty[)$ and} $m(\cdot)e(\cdot) \in L^1(t_0,+\infty;\cH)$. Then, we have
  \begin{center}
  $\sup_{t \geq t_0} E(t) < +\infty$, $\sup_{t \geq t_0} t\norm{\dot{x}(t)} < +\infty$ and $\sup_{t \geq t_0} \|x(t)-x^\star\| < +\infty$.
  \end{center}
\end{lemma}

\begin{proof}
Consider the function $\cE(t)$ in \eqref{eq:lyap} with the choices \eqref{eq:odetwo_aparamchoice}-\eqref{eq:odetwo_choice_a} for $c(t)$, $d(t)$, $b(t)$ and $a(t)$, with $b \in ]2,\alpha-1[$. For such a choice, there exists $t_1 \geq t_0$ such that for all $t\ge t_1$, $a(t) > 0$, $\beta(t) > 0$ (and in turn, $m(t) > 0$), and all conditions of \eqref{eq:conditionsscalarfunc} are satisfied. Thus, $\cE(t)$ is monotonically decreasing on $[t_1,+\infty[$ according to Lemma~\ref{lem:decreaselyap}. Since the solution $x(t)$ is continuous, it is bounded on $[t_0,t_1]$ and so without loss of generality we can assume that $t_1=t_0$ and proceed to show,
\begin{equation}\label{eq:bounde}
\begin{aligned}
&E(t)
\leq E(t_0)+\int_{t_0}^t \dotp{\tau\pa{b(x(\tau)-x^\star)+\tau \dot{x}(\tau)}+a(\tau) \beta(\tau)\nabla f\pa{x(\tau)+\beta(\tau)\dot{x}(\tau)}}{e(\tau)} d\tau \\
&= E(t_0)+\int_{t_0}^t \dotp{\tau\pa{b(x(\tau)-x^\star)+\tau \dot{x}(\tau)}+a(\tau) \beta(\tau)\pa{\nabla f\pa{x(\tau)+\beta(\tau)\dot{x}(\tau)}-\nabla f(x^\star)}}{e(\tau)} d\tau \\
&\leq E(t_0)+\int_{t_0}^t \bpa{\tau\norm{b(x(\tau)-x^\star)+\tau \dot{x}(\tau)}+a(\tau) \beta(\tau)\norm{\nabla f\pa{x(\tau)+\beta(\tau)\dot{x}(\tau)}-\nabla f(x^\star)}}\norm{e(\tau)} d\tau \\
&\leq E(t_0)+\int_{t_0}^t \bpa{\tau\norm{b(x(\tau)-x^\star)+\tau \dot{x}(\tau)}+a(\tau) \beta(\tau)L\norm{x(\tau)-x^\star+\beta(\tau)\dot{x}(\tau)}}\norm{e(\tau)} d\tau \\
&\leq E(t_0)+\int_{t_0}^t \bpa{\tau\norm{b(x(\tau)-x^\star)+\tau \dot{x}(\tau)}+La(\tau) \beta(\tau)\norm{x(\tau)-x^\star}+L a(\tau)\beta(\tau)^2\norm{\dot{x}(\tau)}}\norm{e(\tau)} d\tau .
\end{aligned}
\end{equation}
Denote $d = b(\alpha-1-b)$. We have $d > 0$. Moreover,  $a(t) > 0$ for $t \geq t_0$. One can then drop the first term in $E(t)$, and \eqref{eq:bounde} becomes, for any $t \geq t_0$,
\begin{align*}
&\frac{1}{2}\|b(x(t)-x^\star)+t\dot{x}(t)\|^2+\frac{d}{2}\|x(t)-x^\star\|^2\\
&\leq E(t_0)+\int_{t_0}^t \bpa{\tau\norm{b(x(\tau)-x^\star)+\tau \dot{x}(\tau)}+\sqrt{d}\frac{La(\tau) \beta(\tau)}{\sqrt{d}}\norm{x(\tau)-x^\star}+\tau\frac{L a(\tau)\beta(\tau)^2}{t_0}\norm{\dot{x}(\tau)}}\norm{e(\tau)} d\tau \\
&\leq E(t_0)+\int_{t_0}^t \bpa{\norm{b(x(\tau)-x^\star)+\tau \dot{x}(\tau)}+\sqrt{d}\norm{x(\tau)-x^\star}+\tau\norm{\dot{x}(\tau)}}\max\pa{1,d^{-1/2},t_0^{-1}}m(\tau)\norm{e(\tau)} d\tau \\
&\leq E(t_0)+\int_{t_0}^t \bpa{2\norm{b(x(\tau)-x^\star)+\tau \dot{x}(\tau)}+(b+\sqrt{d})\norm{x(\tau)-x^\star}}\max\pa{1,d^{-1/2},t_0^{-1}}m(\tau)\norm{e(\tau)} d\tau \\
&\leq E(t_0)+\int_{t_0}^t \bpa{\norm{b(x(\tau)-x^\star)+\tau \dot{x}(\tau)}+\sqrt{d}\norm{x(\tau)-x^\star}} C m(\tau)\norm{e(\tau)} d\tau ,
\end{align*}
for some constant $C \geq \max\pa{1,d^{-1/2},t_0^{-1}}\max\pa{2,1+\sqrt{\frac{b}{\alpha-1-b}}}$. Now, Jensen's inequality yields
\begin{align*}
&\frac{1}{4}\pa{\|b(x(t)-x^\star)+t\dot{x}(t)\| + \sqrt{d}\|x(t)-x^\star\|}^2 \leq \frac{1}{2}\|b(x(t)-x^\star)+t\dot{x}(t)\|^2+\frac{d}{2}\|x(t)-x^\star\|^2\\
&\leq E(t_0)+\int_{t_0}^t \bpa{\norm{b(x(\tau)-x^\star)+\tau \dot{x}(\tau)}+\sqrt{d}\norm{x(\tau)-x^\star}} C |m(\tau)|\norm{e(\tau)} d\tau .
\end{align*}
Using {the Gronwall} Lemma~\ref{lem:BrezisA5}, we conclude that, for all $t \geq t_0$
\begin{equation}\label{eq:integrability}
\|b(x(t)-x^\star)+t\dot{x}(t)\| + \sqrt{d}\|x(t)-x^\star\| \leq 2\sqrt{|E(t_0)|} + 2 C \int_{t_0}^{+\infty} |m(\tau)| \|e(\tau)\|d\tau < +\infty ,
\end{equation}
whence we get boundedness of $\|x(t)-x^\star\|$ and $\|b(x(t)-x^\star)+t\dot{x}(t)\|$. The triangle inequality then shows that $t\norm{\dot{x}(t)}$ is also bounded. Using this into \eqref{eq:bounde} together with Cauchy-Schwarz inequality and our integrability assumption, we deduce boundedness of $E(t)$.\qed
\end{proof}

\begin{remark}\label{rem:sumbound}
Recall our discussion on the parameters in \eqref{eq:odetwo_aparamchoice}-\eqref{eq:odetwo_choice_a}. Notice that we have $t^2  \leq a(t) \leq t^2 + \kappa_1$ and $\gamma/2 \leq \beta(t) \leq \beta_0$ for $t$ large enough, where $\beta_0 = \gamma + |\beta|/t_0$ and $\kappa$ is a non-negative constant. In turn, for $t$ large enough, we have
\[
\max\pa{1,\gamma/2}L\gamma/2t^2 \leq m(t) \leq \max\pa{1,\beta_0}L\beta_0(t+\kappa_1)^2 .
\]
%Observe that since $t_0 > 0$, $t^2 e(t) \in L^1(t_0,+\infty;\cH)$ implies $t^{\delta} e(t) \in L^1(t_0,+\infty;\cH)$, for any $\delta \in [0,2]$. 
Clearly the condition $m(\cdot)e(\cdot) \in L^1(t_0,+\infty;\cH)$ is equivalent to $t^2e(t) \in L^1(t_0,+\infty;\cH)$. 
\end{remark}
\noindent From Lemma~\ref{lem:boundconv}, we obtain the following convergence rates and integral estimates.
\begin{theorem}\label{th:int}
Under the assumptions of Lemma~\ref{lem:boundconv}, the following holds:
\begin{enumerate}[label={\rm (\roman*)}]

\item \label{th:intclaim1}
{$\DS{f\left(x(t)+\beta(t)\dot{x}(t)\right) -  \min_{\cH} f = \cO\pa{\frac{1}{t^2}}}$ as $t \to +\infty$};
\item \label{th:intclaim2}
$\DS{\norm{\dot{x}(t)} = \cO\pa{\frac{1}{t}}}$ as $t \to +\infty$;
\item \label{th:intclaim3}
$\DS{
\int_{t_0}^{+\infty} t \pa{f(x(t)) -  \min_{\cH} f }dt < +\infty
}
$;
\item \label{th:intclaim4}
$\DS{
\int_{t_0}^{+\infty} t^2 \norm{\nabla f\left( x(t)+\beta(t)\dot{x}(t)\right)}^2dt < +\infty
}
$;
\item \label{th:intclaim5}
$
\DS{\int_{t_0}^{+\infty} t\norm{\dot{x}(t)}^2 dt < +\infty
}
$.
\end{enumerate}
\end{theorem}
\begin{proof} {
Claim~{\ref{th:intclaim2}} follows from Lemma~\ref{lem:boundconv}. Discarding the non-negative terms in $E(t)$, Lemma~\ref{lem:boundconv} together with the fact that $a(t) \geq t^2$ for $t$ large enough, also gives
\[
f\left(x(t)+\beta(t)\dot{x}(t)\right) -  \min_{\cH} f = \cO\pa{\frac{1}{t^2}}.
\]
To show the remaining integral estimates, consider the function $\cE(\cdot)$ in \eqref{eq:lyap} with the choices \eqref{eq:odetwo_aparamchoice}-\eqref{eq:odetwo_choice_a} of $c(t)$, $d(t)$, $b(t)$ and $a(t)$, where $b \in ]2,\alpha-1[$. We first argue similarly to \cite{alecsa2019extension} to show that for $t$ large enough, we have $\dot{a}(t)-bt \leq -\frac{(\alpha-3)t}{2}$, since $\alpha > 3$ and $b > 2$. In addition, for (a possibly different) $t$ large enough, it is straightforward to see that $a(t)\pa{\gamma+\beta/t} \geq t^2\gamma/2$. With these bounds, \eqref{eq:edotfour} reads, for  $t$ large enough,
\begin{equation}\label{eq:edotfourspecial}
\dot{\cE}(t) \leq -\frac{(\alpha-3)t}{2}(f(x(t)+\beta(t)\dot{x}(t))-\bar{f})
-t^2\gamma/2\norm{\nabla f(x(t)+\beta(t)\dot{x}(t))}^2 \\ -t(\alpha-1-b)\norm{\dot{x}(t)}^2 .
\end{equation}
Integrating \eqref{eq:edotfourspecial}, and using that $\cE$ is bounded thanks to Lemma~\ref{lem:boundconv}, we get statements {\ref{th:intclaim4}-\ref{th:intclaim5}} and 
\begin{equation}\label{eq:claim3}
\int_{t_0}^{+\infty} t \pa{f(x(t)+\beta(t)\dot{x}(t)) - \min_{\cH} f} dt < +\infty .
\end{equation}
 Let $\beta_0 = \gamma+\abs{\beta}/t_0$. By the gradient descent lemma, 
\begin{align}
f(x(t)) - f(x(t)+\beta(t)\dot{x}(t)) 
&\leq -\beta(t)\dotp{\nabla f(x(t)+\beta(t)\dot{x}(t))}{\dot{x}(t)} + \frac{L}{2} \beta(t)^2\norm{\dot{x}(t)}^2 \label{lem:descent}\\ 
&\leq \beta_0 \norm{\nabla f(x(t)+\beta(t)\dot{x}(t))}\norm{\dot{x}(t)} + \frac{L}{2} \beta_0^2\norm{\dot{x}(t)}^2.\nonumber
\end{align}
By Cauchy-Schwarz inequality, we have
\begin{eqnarray*}
&&\int_{t_0}^{+\infty} t \pa{f(x(t)) -  \min_{\cH} f}dt 
\leq \int_{t_0}^{+\infty} t \pa{f(x(t)+\beta(t)\dot{x}(t)) -  \min_{\cH} f}dt \\
&&+\beta_0 \pa{\int_{t_0}^{+\infty} t \norm{\nabla f(x(t)+\beta(t)\dot{x}(t))}^2 dt}^{1/2} \pa{\int_{t_0}^{+\infty} t\norm{\dot{x}(t)}^2 dt}^{1/2} 
+ \frac{L\beta_0^2}{2}\int_{t_0}^{+\infty} t\norm{\dot{x}(t)}^2 dt.
\end{eqnarray*}
In view of \eqref{eq:claim3} and claims {\ref{th:intclaim4}-\ref{th:intclaim5}}, statement {\ref{th:intclaim3}} follows. }\qed
\end{proof}

%%%%%%%%%%%%%%%%%%%%%%%%%%%%%%%%%%%%%%%%%%%%%%%%%%%%%%%%%%%%%%%%%%%%%%%%%%%%%%%%%%%%%%%%%%%%%%%%%%%%%%%%%%%%%%%
\subsubsection{Convergence of the trajectories}
We now turn to showing weak convergence of the trajectories to a minimizer. 
\begin{theorem}\label{T:weak_convergence-smooth-implicit}
Suppose that the assumptions of Lemma~\ref{lem:boundconv} hold. Then $x(t)$ converges weakly to a minimizer of $f$.
\end{theorem}
\begin{proof}
As in the explicit case, we  invoke {Opial}'s Lemma~\ref{Opial}. Recall that the trajectory $x(\cdot)$ is bounded by Lemma~\ref{lem:boundconv}. Therefore, for any sequence $\pa{x(t_n)}_{n \in \N}$ which converges weakly to, say, $\bar{x}$, as $t_n \to +\infty$, {Theorem~\ref{th:int}\ref{th:intclaim1}-\ref{th:intclaim2} entails that
\[
f(\bar{x}) \leq \liminf_{n \to +\infty} f(x(t_n)+ \beta(t_n)\dot{x}(t_n)) = \lim_{t \to +\infty} f\left(x(t)+\beta(t)\dot{x}(t)\right) = \bar{f} ,
\]
}
\ie, each weak cluster point of $x(t_n)$ belongs to $S$. To get weak convergence of the trajectory, it remains to show that $\lim_{t \to +\infty}\|x(t)-x^\star\|$ exists.

Let $h: t \in [t_0,+\infty[ ~ \mapsto \frac{1}{2}\|x(t)-x^\star\|^2$. {Under the assumptions on $f$ and $e$, $x$ is the unique classical global solution to \eqref{eq:odetwo_a}, \ie, $x \in \cC^2([t_0,+\infty[)$. Thus so is $h$ and}
\[
\dot{h}(t)=\dotp{\dot{x}(t)}{x(t)-x^\star} \qandq \ddot{h}(t)=\dotp{\ddot{x}(t)}{x(t)-x^\star}+\norm{\dot{x}(t)}^2 .
\]
From \eqref{eq:odetwo_a}, we obtain
\begin{align*}
&\ddot{h}(t)+\frac{\alpha}{t} \dot{h}(t)
=\dotp{\ddot{x}(t)+\frac{\alpha}{t}\dot{x}(t)}{x(t)-x^\star}+\norm{\dot{x}(t)}^2 \\
&= -\dotp{\nabla f(x(t)+\beta(t)\dot{x}(t))+e(t)}{x(t)-x^\star} + \norm{\dot{x}(t)}^2 \\
&= -\dotp{\nabla f(x(t)+\beta(t)\dot{x}(t))}{x(t)+\beta(t)\dot{x}(t)-x^\star} -\dotp{e(t)}{x(t)-x^\star} +\norm{\dot{x}(t)}^2 \\
&\quad+\beta(t)\dotp{\nabla f(x(t)+\beta(t)\dot{x}(t))}{\dot{x}(t)} .
\end{align*}
Convexity of $f$ implies,
\[
-\dotp{\nabla f(x(t)+\beta(t)\dot{x}(t))}{x(t)+\beta(t)\dot{x}(t)-x^\star} \leq \bar{f}-f(x(t)+\beta(t)\dot{x}(t)) \leq 0 ,
\]
and thus,
\begin{align*}
\ddot{h}(t)+\frac{\alpha}{t} \dot{h}(t)
&\leq \norm{x(t)-x^\star}\norm{e(t)} + \norm{\dot{x}(t)}^2 + \beta_0\norm{\nabla f(x(t)+\beta(t)\dot{x}(t))}\norm{\dot{x}(t)} \\
&\leq C\norm{e(t)} + \norm{\dot{x}(t)}^2 + \beta_0\norm{\nabla f(x(t)+\beta(t)\dot{x}(t))}\norm{\dot{x}(t)}
\end{align*}
where $C = \sup_{t \geq t_0}\norm{x(t)-x^\star} < +\infty$ thanks to Lemma~\ref{lem:boundconv}, and we denoted $\beta_0 = 1+|\beta|/t_0$. Multiplying both sides by $t$, we arrive at
\begin{align*}
t\ddot{h}(t)+\alpha \dot{h}(t)
&\leq C t\norm{e(t)} + t\norm{\dot{x}(t)}^2 + \frac{\beta_0}{\sqrt{t_0}} (t\norm{\nabla f(x(t)+\beta(t)\dot{x}(t))})(\sqrt{t}\norm{\dot{x}(t)}) .
\end{align*}
The right-hand side of this inequality belongs to $L^1(t_0,+\infty;\R)$ by assumption on the error, and using the Cauchy-Schwarz inequality and Theorem~\ref{th:int}\ref{th:intclaim4}-\ref{th:intclaim5} for the last two terms. {Since $h \in \cC^2([t_0,+\infty[$),} it then follows from Lemma~\ref{lem:convw} that $\lim_{t \to +\infty}\|x(t)-x^\star\|$ exists. We have now shown that all conditions of Lemma~\ref{Opial} are satisfied, {hence the weak convergence of the trajectories}. \qed
\end{proof}

\begin{remark}
For the unperturbed case, similar rates to ours in Theorem~\ref{th:int} and weak convergence of the trajectory were proved in \cite{alecsa2019extension}. Again, handling errors necessitates new estimates and bounds, for instance those established in Lemma~\ref{lem:boundconv}. 
\end{remark}

%%%%%%%%%%%%%%%%%%%%%%%%%%%%%%%%%%%%%%%%%%%%%%%%%%%%%%%%%%%%%%%%%%%%%%%%%%%%%%%%%%%%%%%%%%%%%%%%%%%%%%%%%%%%%%%
%%%%%%%%%%%%%%%%%%%%%%%%%%%%%%%%%%%%%%%%%%%%%%%%%%%%%%%%%%%%%%%%%%%%%%%%%%%%%%%%%%%%%%%%%%%%%%%%%%%%%%%%%%%%%%%
\subsection{Discussion}\label{s:discussion}
We now discuss the main differences between the two systems in terms of their stability to perturbations and the corresponding assumptions. Recall from Remark~\ref{rem:sumbound}, the integrability assumption $m(\cdot)e(\cdot) \in L^1(t_0,+\infty;\cH)$ required to ensure stability for system \eqref{eq:odetwo_a} is equivalent to ensuring that the second-order moment of the error $e(\cdot)$ is finite. One may wonder whether this is more stringent than the integrability assumptions for the explicit Hessian system \eqref{eq:origode_a} involving the control of the first-order moments of the error and its derivative (see Section~\ref{s:secondorderconv}). The answer is clearly affirmative in the scalar case with a simple integration by parts argument. {Indeed, supposing without loss of generality that $e(\cdot)$ is a non-increasing and non-negative function, one has
\[
\int_{t_0}^{+\infty} t |\dot{e}(t)| dt = -\int_{t_0}^{+\infty} t \dot{e}(t) dt \leq t_0e(t_0) +  \int_{t_0}^{+\infty}e(t) dt \leq t_0 e(t_0) +  t_0^{-2}\int_{t_0}^{+\infty}t^2|e(t)| dt .
\]
}
Another intuitive way to understand this is to look at what happens if the system is discretized with finite differences. In this case, the integrability assumptions on the errors for system \eqref{eq:origode_a} boil down to controlling only the  first-order moment of the (discretized) error.
{
Indeed, temporal discretization with fixed step size of $t\|\dot{x}(t)\|$  gives $k \|x_{k+1} -x_k\|$ whose summability is clearly implied by the summability of $k \|x_{k} \|$.
} 
We conclude this discussion by noting that Lipschitz continuity of the gradient is not needed for the estimates and convergence analysis of \eqref{eq:origode_a} while it is used extensively to analyze \eqref{eq:odetwo_a}. This is a distinctive avantage of \eqref{eq:origode_a} compared to \eqref{eq:odetwo_a}. This will be even more notable when extending to the non-smooth case; see Section~\ref{s:nonsmooth}.
%%%%%%%%%%%%%%%%%%%%%%%%%%%%%%%%%%%%%%%%%%%%%%%%%%%%%%%%%%%%%%%%%%%%%%%%%%%%%%%%%%%%%%%%%%%%%%%%%%%%%%%%%%%%%%%
\section{Smooth Strongly Convex Case}\label{s:sconvex}
%%%%%%%%%%%%%%%%%%%%%%%%%%%%%%%%%%%%%%%%%%%%%%%%%%%%%%%%%%%%%%%%%%%%%%%%%%%%%%%%%%%%%%%%%%%%%%%%%%%%%%%%%%%%%%%
We will successively examine the Explicit Hessian Damping, then the Implicit Hessian Damping.
%%%%%%%%%%%%%%%%%%%%%%%%%%%%%%%%%%%%%%%%%%%%%%%%%%%%%%%%%%%%%%%%%%%%%%%%%%%%%%%%%%%%%%%%%%%%%%%%%%%%%%%%%%%%%%%
\subsection{Explicit Hessian Damping}\label{s:secondordersconv}
In this section we consider the explicit Hessian system under the assumption of strong convexity of $f$. Following Polyak's heavy ball system \cite{BP}, consider the second-order perturbed system 
\begin{equation}\label{dyn-sc}
\ddot{x}(t) + 2\sqrt{\mu} \dot{x}(t) + \beta \nabla^2 f (x(t))\dot{x}(t) + \beta \dot{e}(t)+ \nabla f (x(t)) + e(t)=0 ,
\end{equation}
which has a fixed positive damping coefficient that is adjusted to the modulus $\mu$ of strong convexity of $f$. To study 
\eqref{dyn-sc}, we define the function $\cE : [t_0, +\infty[ 	\to \R_{+}$
\begin{align}
t \mapsto \cE (t) \eqdef f(x(t))- \min_{\cH} f  + \frac{1}{2} \| v(t) \|^2,
\label{dyn-sc-c}
\end{align}
where
\begin{equation}\label{dyn-sc-d} 
v(t)= \sqrt{\mu} (x(t) -x^\star) + \dot{x}(t)+ \beta \nabla f (x(t)) .
\end{equation}

\begin{theorem}\label{strong-conv-thm}
Suppose that $f: \cH \to \R$ is $\mu$-strongly convex for some $\mu >0$, let $x^\star$ be the unique minimizer of $f$. Let $x(\cdot): [t_0, + \infty[ \to \cH$ be a solution trajectory of \eqref{dyn-sc}. Suppose that 
\begin{enumerate}[label=\alph*)]
\item $\DS{ 0 \leq \beta \leq \frac{1}{2\sqrt{\mu}}}$. 
\item $\DS{\int_{t_0}^{+\infty} \|e(t)\| dt< +\infty}$ and $\DS{\int_{t_0}^{+\infty} \|\dot e(t)\| dt< +\infty}$.
\end{enumerate}
{Then the following properties are satisfied:}
\begin{enumerate}[label={\rm (\roman*)}]
\item Minimizing properties:  for all $t\geq t_0$
\[
\cE (t)   \leq \cE (t_0) e^{-\frac{\sqrt{\mu}}{2}(t-t_0)}  + M e^{-\frac{\sqrt{\mu}}{2}t}\int_{t_0}^t  e^{\frac{\sqrt{\mu}}{2}\tau} \| e(\tau)+\beta \dot{e}(\tau) \| d\tau,
\]
where $M \eqdef \sqrt{2 \cE(t_0)} + \DS{\int_{t_0}^{+\infty}  \| e(\tau)+\beta \dot{e}(\tau) \|  d \tau}$.
As a consequence,
\begin{eqnarray*}
&& \lim_{t\to +\infty} \cE (t) =0; \; \lim_{t\to +\infty}  f(x(t))=  \min_{\cH} f \\
&& \lim_{t\to +\infty}\|  x(t) -x^\star\|=
 \lim_{t\to +\infty}\| \nabla f (x(t))\|=
 \lim_{t\to +\infty}\| \dot{x}(t)\|=0.
\end{eqnarray*}
\item Convergence rates: suppose moreover that for some $p>0$, 
$
\DS{\|e(t)+\beta \dot{e}(t)\| = \cO\pa{\frac{1}{t^p}}},
$ as $t \to +\infty$.
Then
$
\cE (t)= \cO \pa{   \frac{1}{t^p} },
$
\ie \; $\cE (t)$ inherits the decay rate of the error terms. As a consequence, as $t \to +\infty$
\begin{eqnarray*}
&& f(x(t))-  \min_{\cH} f  = \cO \pa{   \frac{1}{t^p} };\\
&& \|  x(t) -x^\star\|^2= \cO \pa{   \frac{1}{t^p} }; \; 
\| \dot{x}(t)\|^2= \cO \pa{   \frac{1}{t^p} }.
\end{eqnarray*}
In addition, when $\beta >0$
\[
e^{- \sqrt{\mu}t} \int_{t_0}^t  e^{ \sqrt{\mu}s}\|  \nabla f (x(s))\|^2 ds = \cO \pa{   \frac{1}{t^p} }.
\]
\end{enumerate}
\end{theorem} 
\begin{proof}
Recall $\bar{f} \eqdef \min_{\cH} f = f(x^\star)$. Define $g(t) \eqdef  e(t)+\beta \dot{e}(t)$, so that the constitutive equation is written in the compact form 
\begin{equation}\label{dyn-sc-b}
\ddot{x}(t) + 2\sqrt{\mu} \dot{x}(t) + \beta \nabla^2 f (x(t))\dot{x}(t) + \nabla f (x(t)) + g(t)=0.
\end{equation}
Derivation of $\cE (\cdot) $ gives 
\begin{eqnarray*}
\dot{\cE}(t) 
&=& \dotp{\nabla f (x(t))}{\dot{x}(t)} + \dotp{v(t)}{\dot{v}(t)}\\
&=& \dotp{\nabla f (x(t))}{\dot{x}(t)} + \dotp{v(t)}{\sqrt{\mu}\dot{x}(t) +  \ddot{x}(t) + \beta \nabla^2 f (x(t))\dot{x}(t)}.
\end{eqnarray*}
Using the definition of $v(t)$ and \eqref{dyn-sc-b}, we get
\begin{multline*}
\dot{\cE} (t)=\dotp{\nabla f (x(t))}{\dot{x}(t)}  + \dotp{\sqrt{\mu} (x(t) -x^\star) + \dot{x}(t)+ \beta \nabla f (x(t))}{-\sqrt{\mu}\dot{x}(t)- \nabla f (x(t))} - \dotp{v(t)}{g(t)} .
\end{multline*}
After developing and simplifying, we obtain
\begin{eqnarray*}
&\dot{\cE}(t) + \sqrt{\mu}\dotp{\nabla f (x(t))}{x(t)-x^\star} + \mu\dotp{x(t)-x^\star}{\dot{x}(t)} + \sqrt{\mu} \| \dot{x}(t) \|^2  \\
&+ \beta\sqrt{\mu}\dotp{\nabla f (x(t))}{\dot{x}(t)} + \beta \|\nabla f (x(t))\|^2  = -\dotp{ v(t)}{g(t)}.
\end{eqnarray*}
According to strong convexity of $f$, we have
\[
\dotp{\nabla f(x(t))}{x(t) - x^\star}  \geq f(x(t))- \bar{f} + \frac{\mu}{2} \|x(t) -x^\star\|^2 .
\]
Thus, by combining the last two relations, and by the Cauchy-Schwarz inequality, we obtain 
\[
\dot{\cE} (t) + \sqrt{\mu}A(t) \leq \| v(t)\| \| g(t) \| ,
\]
where 
\begin{equation*}
A(t) \eqdef f(x(t))- \bar{f} + \frac{\mu}{2} \|x(t) - x^\star\|^2  + \sqrt{\mu} \dotp{ x(t) -x^\star }{\dot{x}(t)} +  \| \dot{x}(t) \|^2  \\
+ \beta \dotp{  \nabla f (x(t))}{ \dot{x}(t)       } + \frac{\beta}{\sqrt{\mu}} \| \nabla f (x(t)) \|^2  .
\end{equation*}
Let us make appear $\cE(t)$ in $A(t)$,
\begin{multline*}
A(t)=\cE  (t)   - \frac{1}{2} \|  \dot{x}(t)+ \beta \nabla f (x(t)) \|^2   -  \sqrt{\mu} \dotp{ x(t) -x^\star }{\dot{x}(t) + \beta \nabla f (x(t))} + \sqrt{\mu} \dotp{ x(t) -x^\star }{\dot{x}(t)}   \\
+ \| \dot{x}(t) \|^2 + \beta \dotp{\nabla f (x(t))}{\dot{x}(t)} + \frac{\beta}{\sqrt{\mu}} \| \nabla f (x(t)) \|^2 . 
\end{multline*}
After developing and simplifying, we obtain
\begin{equation*}
\dot{\cE} (t) + \sqrt{\mu}\pa{\cE (t)      
+ \frac{1}{2} \|  \dot{x}(t) \|^2  + \pa{\frac{\beta}{\sqrt{\mu}}-\frac{\beta^2}{2}} \| \nabla f (x(t)) \|^2      -  \beta\sqrt{\mu} \dotp{ x(t) -x^\star }{ \nabla f (x(t))}} 
\leq  \| v(t)\| \| g(t) \| .
\end{equation*}
Since $ 0 \leq \beta \leq \frac{1}{\sqrt{\mu}}$, it holds that
$\frac{\beta}{\sqrt{\mu}}-\frac{\beta^2}{2} \geq \frac{\beta}{2\sqrt{\mu}}$. Hence
\begin{equation*}
\dot{\cE} (t) + \sqrt{\mu}\pa{\cE (t)      
+ \frac{1}{2} \|  \dot{x}(t) \|^2  + \frac{\beta}{2\sqrt{\mu}}\| \nabla f (x(t)) \|^2      -  \beta\sqrt{\mu} \dotp{ x(t) -x^\star }{ \nabla f (x(t))}} 
\leq  \| v(t)\| \| g(t) \| .
\end{equation*}
Let us use again the strong convexity of $f$ to write
\[
\cE (t) = \frac{1}{2}\cE (t) + \frac{1}{2}\cE (t) \geq \frac{1}{2}\cE (t)+ \frac{1}{2} \pa{f(x(t))- \bar{f}} \geq 
\frac{1}{2}\cE (t) +  \frac{\mu}{4} \| x(t) -x^\star \|^2 .
\]
By combining the two inequalities above, we obtain
\[
\dot{\cE} (t) + \frac{\sqrt{\mu}}{2}\cE (t) + \frac{\sqrt{\mu}}{2}\|  \dot{x} (t)\|^2  +  \sqrt{\mu} B(t)  \leq \| v(t)\| \| g(t) \|,
\]
where $B(t) = \frac{\mu}{4} \| x(t) -x^\star \|^2 + \frac{\beta}{2\sqrt{\mu}} \|\nabla f(x(t))\|^2      - \beta\sqrt{\mu} \|x(t) -x^\star\|  \|\nabla f(x(t))  \| $.
Set $X=\| x -x^\star\|$, $Y=  \| \nabla f (x) \|$. Elementary algebraic computation gives that, under the condition $ 0 \leq \beta \leq \frac{1}{2\sqrt{\mu}}$
\[
\frac{\mu}{4} X^2 + \frac{\beta}{2\sqrt{\mu}} Y^2 -  \beta\sqrt{\mu}XY \geq 0.
\]
Hence for $ 0 \leq \beta \leq \frac{1}{2\sqrt{\mu}}$
\begin{equation}\label{basic_diff_ineq_1}
\dot{\cE} (t) + \frac{\sqrt{\mu}}{2}\cE (t) + \frac{\sqrt{\mu}}{2}\|  \dot{x} (t)\|^2  \leq \| v(t)\| \| g(t) \|.
\end{equation}
\begin{enumerate}[label={\rm (\roman*)},leftmargin=3ex]
\item From \eqref{basic_diff_ineq_1}, we  first deduce that
\[
\dot{\cE} (t)  \leq \| v(t)\| \| g(t) \|,
\]
which by integration gives
\[
\cE (t)  \leq \cE(t_0) + \int_{t_0}^t \| v(\tau)\| \| g(\tau) \|d\tau.
\]
By definition of $\cE(t)$, we have $\cE(t) \geq \demi \|v(t)\|^2 $,  which gives
\[
\demi \|v(t)\|^2  \leq \cE(t_0) + \int_{t_0}^t \| v(\tau)\| \| g(\tau) \|d\tau.
\]
According to Lemma~\ref{lem:BrezisA5}, we obtain
\[
\|v(t)\|  \leq  \sqrt{2 \cE(t_0)} + \int_{t_0}^t  \| g(\tau) \|d\tau.
\]
Set $M \eqdef \sqrt{2 \cE(t_0)} + \int_{t_0}^{+\infty}  \| g(\tau) \|d\tau$. By assumption, $\DS{\int_{t_0}^{+\infty}  \| g(\tau) \|d\tau <+\infty}$, and thus $\sup_{t \geq t_0} \|v(t)\|  \leq  M < +\infty$.
Returning to \eqref{basic_diff_ineq_1} we deduce that
\begin{equation}\label{basic_diff_ineq_2}
\dot{\cE} (t) + \frac{\sqrt{\mu}}{2}\cE (t) + \frac{\sqrt{\mu}}{2}\|  \dot{x} (t)\|^2\leq M  \| g(t) \|.
\end{equation}
Therefore
\begin{equation}\label{basic_diff_ineq_3}
\dot{\cE} (t) + \frac{\sqrt{\mu}}{2}\cE (t) \leq M  \| g(t) \|.
\end{equation}
By integrating the differential inequality above, we obtain
\begin{equation}\label{basic_diff_ineq_4}
\cE (t)   \leq \cE (t_0) e^{-\frac{\sqrt{\mu}}{2}(t-t_0)}  + M e^{-\frac{\sqrt{\mu}}{2}t}\int_{t_0}^t  e^{\frac{\sqrt{\mu}}{2}\tau} \| g(\tau) \| d\tau.
\end{equation}
We now use Lemma~\ref{basic-int}, which is the continuous version of Kronecker's Theorem for series, with $f(t)= \|g(t)\| $ and $\varphi(t)= e^{\frac{\sqrt{\mu}}{2}t}$. By assumption we have $ \int_{t_0}^{+\infty}  \| g(\tau) \|d\tau <+\infty$.
We deduce that
\[
\lim_{t\to +\infty}  \frac{1}{e^{\frac{\sqrt{\mu}}{2}t}}\int_{t_0}^t  e^{\frac{\sqrt{\mu}}{2}\tau} \| g(\tau) \| d\tau =0.
\]
Therefore, from \eqref{basic_diff_ineq_4} we obtain
\[
\lim_{t\to +\infty} \cE (t) =0.
\]
By definition of $\cE (t)$ this implies
\begin{eqnarray}
&&\lim_{t\to +\infty}  f(x(t))-  \min_{\cH} f =0  \label{conv1} \\
&&\lim_{t\to +\infty} \| \sqrt{\mu} (x(t) -x^\star) + \dot{x}(t)+ \beta \nabla f (x(t)) \|=0. \label{conv2}
\end{eqnarray}
Acoording to \eqref{conv1} and the strong convexity of $f$ we deduce that
\[
\lim_{t\to +\infty}\|x(t) -x^\star\|=0
\]
By continuity of $\nabla f $, and since $\nabla f (x^\star)=0$, we deduce that
\[
\lim_{t\to +\infty}\| \nabla f (x(t))\|=0.
\]
Combining the above results with \eqref{conv2}, we deduce that
\[
 \lim_{t\to +\infty}\| \dot{x}(t)\|=0.
\]
 
\item Let us make precise the argument developed above, and assume that, as $t \to +\infty$
\[
\| g(t) \| = \cO \pa{   \frac{1}{t^p} },
\]
where $p>0$. Then, from  \eqref{basic_diff_ineq_4} we get 
\begin{eqnarray*}
\cE (t)   &\leq& \cE (t_0) e^{-\frac{\sqrt{\mu}}{2}(t-t_0)}  + M e^{-\frac{\sqrt{\mu}}{2}t}\Big(  \int_{t_0}^{\frac{t}{2}} e^{\frac{\sqrt{\mu}}{2}\tau} \| g(\tau) \| d\tau 
+  \int_{\frac{t}{2}}^t  e^{\frac{\sqrt{\mu}}{2}\tau} \| g(\tau) \| d\tau \Big)\\
&\leq& \cE (t_0) e^{-\frac{\sqrt{\mu}}{2}(t-t_0)}  + M e^{-\frac{\sqrt{\mu}}{2}t}\Big(  C_1  e^{\frac{\sqrt{\mu}t}{4}} 
+  \int_{\frac{t}{2}}^t  e^{\frac{\sqrt{\mu}}{2}\tau}  \frac{C_2}{\tau^p}  d\tau \Big)\\
&\leq& \cE (t_0) e^{-\frac{\sqrt{\mu}}{2}(t-t_0)}  + M e^{-\frac{\sqrt{\mu}}{2}t}\Big(  C_1  e^{\frac{\sqrt{\mu}t}{4}} 
+   \frac{C_2}{t^p}  e^{\frac{\sqrt{\mu}}{2}t} \Big)\\
&\leq & \cE (t_0) e^{-\frac{\sqrt{\mu}}{2}(t-t_0)}  + M\Big( C_1  e^{-\frac{\sqrt{\mu}t}{4}} 
+   \frac{C_2}{t^p} \Big) \\
&=&  \cO\pa{\frac{1}{t^p} }.
\end{eqnarray*}
By definition of $\cE (t)$ and strong convexity of $f$, we infer
\[
\frac{\mu}{2}\norm{x(t)-x^\star}^2 \leq f(x(t)) -  \min f(\cH) = \cO\pa{\frac{1}{t^p} } \qandq
\|\sqrt{\mu} (x(t) -x^\star) + \dot{x}(t)+ \beta \nabla f (x(t)) \|^2 = \cO\pa{\frac{1}{t^p} } .
\]
Developing the left-hand side of the last expression, we obtain
\begin{multline*}
\mu \| x(t) -x^\star\|^2 + \|  \dot{x}(t)\|^2 + \beta^2  \|\nabla f (x(t))\|^2 + 2\beta \sqrt{\mu}\dotp{ x(t) -x^\star}{\nabla f (x(t))} \\
 + \dotp{\dot{x}(t)}{ 2 \beta\nabla f (x(t)) +  2\sqrt{\mu}  ( x(t) -x^\star)}  \leq \frac{C}{t^p}.
\end{multline*}
By convexity of $f$, we have $\dotp{x(t) -x^\star}{\nabla f (x(t))} \geq f(x(t)) - \bar{f}$. Moreover,
\begin{equation*}
\dotp{\dot{x}(t)}{ 2 \beta\nabla f (x(t)) +  2\sqrt{\mu}  ( x(t) -x^\star)}
= \frac{d}{dt} \pa{2 \beta (f(x(t)) - \bar{f}    )+ \sqrt{\mu} \| x(t) -x^\star\|^2  }.  
\end{equation*}
Combining the above results, we obtain
\begin{equation*}
\sqrt{\mu} \pa{2 \beta (f(x(t)) - \bar{f}) +\sqrt{\mu}  \| x(t) -x^\star\|^2} + \beta^2  \|  \nabla f (x(t))\|^2     
 +   \frac{d}{dt} \pa{2 \beta (f(x(t)) - \bar{f})+ \sqrt{\mu} \| x(t) -x^\star\|^2  } \leq \frac{C}{t^p}.
\end{equation*}
Set $Z(t)\eqdef 2 \beta (f(x(t)) - \bar{f}) +\sqrt{\mu}  \| x(t) -x^\star\|^2$. We have
\[
\frac{d}{dt} Z(t) + \sqrt{\mu} Z(t) + \beta^2  \|  \nabla f (x(t))\|^2 
 \leq \frac{C}{t^p}.
\]
By integrating this differential inequality, elementary computation gives 
\[
e^{- \sqrt{\mu}t} \int_{t_0}^t  e^{ \sqrt{\mu}s}\|  \nabla f (x(s))\|^2 ds \leq \frac{C}{t^p}.
\]
This completes the proof. \qed
\end{enumerate}
\end{proof}

%%%%%%%%%%%%%%%%%%%%%%%%%%%%%%%%%%%%%%%%%%%%%%%%%%%%%%%%%%%%%%%%%%%%%%%%%%%%%%%%%%%%%%%%%%%%%%%%%%%%%%%%%%%%%%%
%%%%%%%%%%%%%%%%%%%%%%%%%%%%%%%%%%%%%%%%%%%%%%%%%%%%%%%%%%%%%%%%%%%%%%%%%%%%%%%%%%%%%%%%%%%%%%%%%%%%%%%%%%%%%%%
\subsection{Implicit Hessian Damping}\label{s:approxsconv}
We now turn to the implicit Hessian system, and take in the Polyak heavy ball system  a fixed positive damping coefficient which is adjusted to the modulus of strong convexity of $f$. This gives the system
\begin{equation}\label{dyn-sc-implicit}
\ddot{x}(t) + 2\sqrt{\mu} \dot{x}(t) +  \nabla f \pa{x(t)+ \beta\dot{x}(t)} + e(t)=0.
\end{equation}
To analyze \eqref{dyn-sc-implicit}, we define the function $\cE : [t_0, +\infty[ 	\to \R_+ $
\begin{align}
t 	\mapsto \cE (t) \eqdef f\pa{x(t)+ \beta\dot{x}(t))} -  \min_{\cH} f  + \frac{1}{2} \| \sqrt{\mu} (x(t) -x^\star) + \dot{x}(t)\|^2.
\end{align}

\begin{theorem}\label{strong-conv-thm-implicit}
Suppose that $f: \cH \to \R$ is $\mu$-strongly convex for some $\mu >0$, and let $x^\star$ be the unique minimizer of $f$. Let $x(\cdot): [t_0, + \infty[ \to \cH$ be a solution trajectory of \eqref{dyn-sc-implicit}.
Suppose that 
\begin{enumerate}[label=\alph*)]
\item $\DS{ 0 \leq \beta \leq \frac{1}{2\sqrt{\mu}}}$. 
\item$\DS{\int_{t_0}^{+\infty} \|e(t)\| dt< +\infty}$.
\end{enumerate}
{Then the following properties are satisfied:} 
\begin{enumerate}[label={\rm (\roman*)}]
\item Minimizing properties: there exists a positive constant $M$ such that for all $t\geq t_0$
\[
\cE (t)   \leq \cE (t_0) e^{-\frac{\sqrt{\mu}}{2}(t-t_0)}  + M e^{-\frac{\sqrt{\mu}}{2}t}\int_{t_0}^t  e^{\frac{\sqrt{\mu}}{2}\tau} \| e(\tau) \| d\tau.
\]
More precisely,
\[
M \eqdef  \sqrt{\frac{\cE(t_0)}{c}} +  \frac{1}{2c}\int_{t_0}^{+\infty} \| e(\tau) \|d\tau  \; \mbox{ with } \; c = \frac{\min\{\mu,1\}}{4 \max\{\beta^2L^2,1\}}
\]
and $L$ is the Lipschitz constant of $\nabla f$. Consequently,
\begin{eqnarray*}
&& \lim_{t\to +\infty} \cE (t) =0; \; \lim_{t\to +\infty}  f(x(t)) =  \min_{\cH} f \\
&& \lim_{t\to +\infty}\|  x(t) - x^\star\|=
 \lim_{t\to +\infty}\| \nabla f (x(t))\|=
 \lim_{t\to +\infty}\| \dot{x}(t)\|=0.
\end{eqnarray*}

\item Convergence rates: suppose moreover that for some $p>0$, 
$
\DS{\|e(t) \| = \cO\pa{   \frac{1}{t^p}}},
$ as $t \to +\infty$.
Then
$
\cE (t)= \cO \pa{   \frac{1}{t^p} },
$
\ie $\cE (t)$ inherits the decay rate of the error terms. In turn, as $t \to +\infty$
\begin{eqnarray*}
&& f\pa{x(t)} - \min_{\cH} f = \cO\pa{\frac{1}{t^p} };\\
&& \|x(t) -x^\star\|^2= \cO\pa{\frac{1}{t^p} }; \; 
\| \dot{x}(t)\|^2= \cO\pa{\frac{1}{t^p} };  \; \|\nabla f(x(t))\|^2 = \cO \pa{\frac{1}{t^p}}.
\end{eqnarray*}
\end{enumerate}
\end{theorem} 
\begin{proof}
Let us define
\begin{equation}\label{dyn-sc-d-implicit} 
v(t)= \sqrt{\mu} (x(t) -x^\star) + \dot{x}(t).
\end{equation}
and thus, $\cE$ equivalently reads
\begin{equation}\label{dyn-sc-c-implicit}
\cE (t) = f\pa{x(t)+ \beta\dot{x}(t)} -  \min_{\cH} f + \frac{1}{2} \| v(t) \|^2 .
\end{equation}
Taking the derivative in time of $\cE (\cdot) $ gives 
\begin{align*}
\dot{\cE} (t)&=\dotp{\nabla f\pa{x(t)+ \beta\dot{x}(t)}}{\dot{x}(t) + \beta  \ddot{x}(t)}  + \dotp{ v(t)}{\dot{v}(t)}\\
&= \dotp{\nabla f\pa{x(t)+ \beta\dot{x}(t)}}{ \dot{x}(t)  + \beta  \ddot{x}(t)}  + \dotp{ \sqrt{\mu} (x(t) -x^\star) + \dot{x}(t)}{\sqrt{\mu}  \dot{x}(t)  +  \ddot{x}(t)}.
\end{align*}
Using  the constitutive equation \eqref{dyn-sc-b}, we get
\begin{multline*}
\dot{\cE} (t) = \dotp{  \nabla f\pa{x(t)+ \beta\dot{x}(t)}}{ (1-2\beta \sqrt{\mu})\dot{x}(t) - \beta   \nabla f\pa{x(t)+ \beta\dot{x}(t)}-\beta e(t)   }  \\
+ \dotp{ \sqrt{\mu} (x(t) -x^\star) + \dot{x}(t)}{-\sqrt{\mu}  \dot{x}(t)  -  \nabla f\pa{x(t)+ \beta\dot{x}(t)} -e(t)  }.
\end{multline*}
After developing and simplifying, we obtain
\begin{multline*}
\dot{\cE} (t) +2\beta \sqrt{\mu}\dotp{  \nabla f\pa{x(t)+ \beta\dot{x}(t)}}{ \dot{x}(t)}  + \sqrt{\mu}\dotp{  \nabla f\pa{x(t)+ \beta\dot{x}(t)}}{x(t) -x^\star }  + \beta \|  \nabla f\pa{x(t)+ \beta\dot{x}(t)}\|^2  \\
+ \sqrt{\mu} \| \dot{x}(t) \|^2 + \mu \dotp{ x(t) -x^\star }{ \dot{x}(t)} = -\dotp{ \sqrt{\mu} (x(t) -x^\star) + \dot{x}(t)
+\beta\nabla f\pa{x(t)+ \beta\dot{x}(t)}}{e(t)}.
\end{multline*} 
In view of strong convexity of $f$, we have
\begin{multline*}
\dotp{\nabla f\pa{x(t)+ \beta\dot{x}(t)}}{ x(t) -x^\star}  = \dotp{  \nabla f\pa{x(t)+ \beta\dot{x}(t)}}{x(t)+\beta\dot{x}(t) - x^\star} - \dotp{\nabla f\pa{x(t) + \beta\dot{x}(t)}}{ \beta\dot{x}(t)} \\
\geq f\pa{x(t)+ \beta\dot{x}(t)}- \bar{f} + \frac{\mu}{2} \| x(t) -x^\star + \beta\dot{x}(t) \|^2 - \dotp{\nabla f\pa{x(t)+ \beta\dot{x}(t)}}{\beta\dot{x}(t)} .
\end{multline*} 
Thus, by combining the last two relations, we obtain
\begin{multline}
\dot{\cE} (t) +\beta \sqrt{\mu}\dotp{  \nabla f\pa{x(t)+ \beta\dot{x}(t)}}{\dot{x}(t)}  + \sqrt{\mu}\pa{f\pa{x(t)+ \beta\dot{x}(t)}- \bar{f} + \frac{\mu}{2} \| x(t) -x^\star  +\beta\dot{x}(t) \|^2}  \\
+ \beta \|\nabla f\pa{x(t)+ \beta\dot{x}(t)}\|^2  + \sqrt{\mu} \| \dot{x}(t) \|^2 + \mu \dotp{ x(t) -x^\star }{\dot{x}(t)}  \leq \| w(t) \| \|e(t)\|, \label{Lyap_22}
\end{multline} 
where we have used Cauchy-Schwarz inequality, and we set
\[
w(t) \eqdef \sqrt{\mu} (x(t) -x^\star) + \dot{x}(t)
+\beta\nabla f\pa{x(t)+ \beta\dot{x}(t)}.
\] 
Let us make $\cE (t)$ appear on the left-hand side of \eqref{Lyap_22}. We get
\[
\dot{\cE} (t) + \sqrt{\mu}\cE (t) + B(t) \leq \| w(t)\| \| e(t) \|
\]
where 
\begin{equation*}
B(t) \eqdef \beta \| \nabla f\pa{x(t)+ \beta\dot{x}(t)}\|^2
+ \frac{\sqrt{\mu}}{2} (\beta^2 \mu +1)\| \dot{x}(t) \|^2+ \beta \sqrt{\mu}\dotp{\nabla f\pa{x(t)+ \beta\dot{x}(t)}}{\dot{x}(t)} \\
+ \beta \mu \sqrt{\mu} \dotp{ x(t) -x^\star }{ \dot{x}(t)}  .
\end{equation*}
Let us use again the strong convexity of $f$ to write
\[
\cE (t) = \frac{1}{2}\cE (t) + \frac{1}{2}\cE (t)  \geq \frac{1}{2}\cE (t)+ \frac{1}{2} \pa{f(x(t)+ \beta\dot{x}(t))- \bar{f}} \geq 
\frac{1}{2}\cE (t) +  \frac{\mu}{4} \| x(t) -x^\star  + \beta\dot{x}(t)\|^2 .
\]
By combining the inequalities above, we obtain
\[
\dot{\cE} (t) + \frac{\sqrt{\mu}}{2}\cE (t)+   C(t)  \leq \| w(t)\| \| e(t) \|,
\]
where
\begin{multline*}
C(t) \eqdef 
\beta \|  \nabla f  (y(t) \|^2
+ \beta \sqrt{\mu}\dotp{\nabla f(y(t))}{\dot{x}(t)} +  \frac{\sqrt{\mu}}{2} (\beta^2 \mu +1)\| \dot{x}(t) \|^2 + \beta \mu \sqrt{\mu} \dotp{x(t) -x^\star }{\dot{x}(t)} \\
+\frac{\mu\sqrt{\mu}}{4} \| x(t) -x^\star + \beta\dot{x}(t)\|^2 ,
\end{multline*}
and we set $y(t) \eqdef x(t)+ \beta\dot{x}(t) $.
Let us show that, for an adequate choice of the parameters, $C(t)$ is non-negative. Let us reformulate $C(t)$ as follows:
Young's inequality gives the following minorization for the two first terms of $C(t)$
\[
\beta \|  \nabla f  (y(t)) \|^2
+ \beta \sqrt{\mu}\langle  \nabla f  (y(t)),  \dot{x}(t)       \rangle \geq -\frac{1}{4}\beta \mu \| \dot{x}(t) \|^2 .
\]
By using this inequality in $C(t)$, and after simplification, we arrive at 
\begin{align*}
&C(t)\geq   
\pa{\frac{\sqrt{\mu}}{2} (\beta^2 \mu +1)-\frac{1}{4}\beta \mu}\| \dot{x}(t) \|^2 + \beta \mu \sqrt{\mu} \dotp{ x(t) -x^\star }{\dot{x}(t)} 
+\frac{\mu\sqrt{\mu}}{4} \| x(t) -x^\star + \beta\dot{x}(t)\|^2 \\
&=  \frac{\mu\sqrt{\mu}}{4} \| x(t) -x^\star + \beta\dot{x}(t)\|^2 + \pa{\frac{\sqrt{\mu}}{2} (\beta^2 \mu +1)-\frac{1}{4}\beta \mu 
- \beta^2 \mu \sqrt{\mu}}\| \dot{x}(t) \|^2 + \beta \mu \sqrt{\mu} \dotp{ x(t) -x^\star + \beta\dot{x}(t)}{ \dot{x}(t)     } \\
&=  \frac{\mu\sqrt{\mu}}{4} \| x(t) -x^\star + \beta\dot{x}(t)\|^2 + \sqrt{\mu}\pa{-\frac{\beta^2 \mu }{2} -\frac{1}{4}\beta \sqrt{\mu}
+ \frac{1}{2}}\| \dot{x}(t) \|^2 + \beta \mu \sqrt{\mu} \dotp{ x(t) -x^\star + \beta\dot{x}(t)}{ \dot{x}(t)}.
\end{align*}
Elementary algebra gives that 
$
-\frac{\beta^2 \mu }{2} -\frac{1}{4}\beta \sqrt{\mu} + \frac{1}{2} \geq 0
$
if and only if 
$
\beta \sqrt{\mu} \leq \frac{\sqrt{17}-1}{4}.
$
According to the classical rule for the sign of a quadratic function of a real variable, we get that $C(t) \geq 0$ under the condition
\[
(\beta \mu \sqrt{\mu})^2 \leq \mu^2\pa{-\frac{\beta^2 \mu }{2} -\frac{1}{4}\beta \sqrt{\mu}
+ \frac{1}{2}}.
\]
Setting $Z= \beta  \sqrt{\mu}$, the latter inequality is equivalent to ensuring 
$$
\frac{3}{2} Z^2 + \frac{1}{4}Z - \demi \leq 0.
$$
which is satisfied for $0 \leq Z \leq \demi$, implying
$
\beta \leq \frac{1}{2 \sqrt{\mu}}.
$
Since $\demi < \frac{\sqrt{17}-1}{4}$, we get as a final condition 
\[
\beta \leq \frac{1}{2 \sqrt{\mu}}.
\]
Thus under this condition we get
\begin{equation}\label{basic-equ-implit-1}
\dot{\cE} (t) + \frac{\sqrt{\mu}}{2}\cE (t) \leq \| w(t)\| \| e(t) \|.
\end{equation}
From \eqref{basic-equ-implit-1}, we first deduce that
\[
\dot{\cE} (t)  \leq \| w(t)\| \| e(t) \|,
\]
which, after integration, gives
\[
\cE (t)  \leq \cE(t_0) + \int_{t_0}^t \| w(\tau)\| \| e(\tau) \| d\tau.
\]
By definition of $w$ we have
\begin{align*}
\| w(t)\| 
&\leq \| v(t)\| + \beta \| \nabla f\pa{x(t)+ \beta\dot{x}(t)}- \nabla f (x^\star) \| \\
&\leq \| v(t)\| + \beta L \|  x(t)- x^\star + \beta\dot{x}(t) \| ,
\end{align*}
where $L$ is the Lipschitz constant of $\nabla f$. On the other hand, strong convexity of $f$ entails
\[
\cE (t) \geq  \frac{\mu}{2} \|  x(t) -x^\star + \beta\dot{x}(t)  \|^2  + \demi  \| v(t)\|^2.
\]
Hence, there exists a positive constant $c$ such that\footnote{One can take $c = \frac{\min\{\mu,1\}}{4 \max\{\beta^2L^2,1\}}$.}
\[
\cE (t) \geq  c \|w(t)\|^2.
\]
This in turn gives
\[
c \|w(t)\|^2  \leq \cE(t_0) + \int_{t_0}^t \| w(\tau)\| \| e(\tau) \|d\tau.
\]
According to Lemma~\ref{lem:BrezisA5}, and $ \int_{t_0}^{+\infty}  \| e(\tau) \|d\tau <+\infty$, we deduce that 
\[
\sup_{t \geq t_0} \|w(t)\| \leq M \eqdef  \sqrt{\frac{\cE(t_0)}{c}} +  \frac{1}{2c}\int_{t_0}^{+\infty} \| e(\tau) \|d\tau < +\infty.
\]
Returning to \eqref{basic-equ-implit-1} we deduce that
\begin{equation}\label{basic_diff_ineq_3_implicit}
\dot{\cE} (t) + \frac{\sqrt{\mu}}{2}\cE (t)        \leq M  \| e(t) \|.
\end{equation}
By integrating the differential inequality above, we obtain
\begin{equation}\label{basic_diff_ineq_4_implicit}
\cE (t)   \leq \cE (t_0) e^{-\frac{\sqrt{\mu}}{2}(t-t_0)}  + M e^{-\frac{\sqrt{\mu}}{2}t}\int_{t_0}^t  e^{\frac{\sqrt{\mu}}{2}\tau} \| e(\tau) \| d\tau .
\end{equation}

\begin{enumerate}[label={\rm (\roman*)},leftmargin=3ex]
\item We first deduce from \eqref{basic_diff_ineq_4_implicit} that $\cE (t) $ tends to zero as $t\to +\infty$. This implies that
\begin{eqnarray}
&&\lim_{t\to +\infty}  f( x(t)+ \beta\dot{x}(t))) =  \min_{\cH} f , \label{conv10}\\
&& \lim_{t\to +\infty}  \| \sqrt{\mu} (x(t) -x^\star) + \dot{x}(t)\| =0 . \label{conv11}
\end{eqnarray}
From \eqref{conv10} and strong convexity of $f$ we deduce that
\begin{equation}
\lim_{t\to +\infty}  \| (x(t)-x^\star)+ \beta\dot{x}(t)\| =0 \label{conv12} .
\end{equation}
From \eqref{conv11} and \eqref{conv12}, and $\beta \neq \frac{1}{\sqrt{\mu}}$ (a consequence of the assumption $\beta \leq \frac{1}{2\sqrt{\mu}  }$), elementary algebra gives
\[
\lim_{t\to +\infty}  \| x(t)-x^\star\| = \lim_{t\to +\infty}  \| \dot{x}(t)\| =0.
\]
In turn, continuity of $f$ and $\nabla f$ imply 
\[
\lim_{t\to +\infty}  \|\nabla f(x(t))\| = 0 \qandq \lim_{t\to +\infty}  f(x(t)) = \min_{\cH} f .
\]

\item Let us now assume that, as $t \to +\infty$, we have
$
\| e(t) \| = \cO \pa{   \frac{1}{t^p} },
$
where $p>0$. Based on \eqref{basic_diff_ineq_4_implicit}, {a similar argument as in the explicit case} (see the proof of Theorem~\ref{strong-conv-thm}) gives
$
\cE (t) =  \cO\pa{\frac{1}{t^p}}.
$
By definition of $\cE (t)$, we infer that
\begin{equation}\label{eq:energy_6_02}
f(x(t)+ \beta\dot{x}(t)) - \min_{\cH} f = \cO\pa{\frac{1}{t^p}}
\end{equation}
and 
\begin{equation}\label{eq:energy_6_02_b}
\| \sqrt{\mu} (x(t) -x^\star) + \dot{x}(t) \|^2 = \cO\pa{\frac{1}{t^p}}.
\end{equation}
From \eqref{eq:energy_6_02} and strong convexity of $f$ we deduce that
\begin{eqnarray}\label{conv12_b}
\| (x(t)-x^\star)+ \beta\dot{x}(t)\|^2 = \cO\pa{\frac{1}{t^p}} .
\end{eqnarray}
Combining \eqref{eq:energy_6_02_b} and \eqref{conv12_b}, and recalling that $\beta \sqrt{\mu} \neq 1$ we immediately obtain 
\begin{equation}\label{eq:convratevel}
\|x(t)-x^\star\|^2 \leq  \frac{C}{t^p}  \qandq \| \dot{x}(t)\|^2 = \cO\pa{\frac{1}{t^p}}.
\end{equation}
According to the Lipschitz continuity of $\nabla f$, and 
$\nabla f (x^\star) =0  $ we deduce that 
\[
\|\nabla f(x(t))\|^2 \leq L^2 \| x(t)-x^\star\|^2 = \cO\pa{\frac{1}{t^p}} .
\]
Now, combining the descent lemma with \eqref{eq:energy_6_02}, \eqref{conv12_b} and \eqref{eq:convratevel}  shows that
\begin{align*}
f(x(t))- \min_{\cH} f
&\leq f(x(t)+\beta\dot{x}(t)) - \min_{\cH} f - \beta\dotp{\nabla f(x(t)+\beta\dot{x}(t))}{\dot{x}(t)} + \frac{L\beta^2}{2}\norm{\dot{x}(t)}^2 \\
&\leq f(x(t)+\beta\dot{x}(t))-\inf_{\cH} f + L\beta\norm{x(t)-x^\star+\beta\dot{x}(t)}\norm{\dot{x}(t)} + \frac{L\beta^2}{2}\norm{\dot{x}(t)}^2 \\
&= \cO\pa{\frac{1}{t^p}} ,
\end{align*}
which completes the proof. \qed
\end{enumerate}
\end{proof}

\begin{remark}
The results of Theorem~\ref{strong-conv-thm-implicit} appear new. Even for the unperturbed case of system \eqref{dyn-sc-implicit}, where $e \equiv 0$, we are not aware of any guarantees for these dynamics in the literature.
\end{remark}

%%%%%%%%%%%%%%%%%%%%%%%%%%%%%%%%%%%%%%%%%%%%%%%%%%%%%%%%%%%%%%%%%%%%%%%%%%%%%%%%%%%%%%%%%%%%%%%%%%%%%%%%%%%%%%%
\section{The Non-smooth Case}\label{s:nonsmooth}
%%%%%%%%%%%%%%%%%%%%%%%%%%%%%%%%%%%%%%%%%%%%%%%%%%%%%%%%%%%%%%%%%%%%%%%%%%%%%%%%%%%%%%%%%%%%%%%%%%%%%%%%%%%%%%%

%%%%%%%%%%%%%%%%%%%%%%%%%%%%%%%%%%%%%%%%%%%%%%%%%%%%%%%%%%%%%%%%%%%%%%%%%%%%%%%%%%%%%%%%%%%%%%%%%%%%%%%%%%%%%%%
\subsection{Explicit Hessian Damping}

In the sequel, we will show that most properties obtained in the smooth case still hold for the global strong solution of \eqref{eq:fos2} (and in particular, all properties that do not require $x(t)$ to be twice differentiable).

%%%%%%%%%%%%%%%%%%%%%%%%%%%%%%%%%%%%%%%%%%%%%%%%%%%%%%%%%%%%%%%%%%%%%%%%%%%%%%%%%%%%%%%%%%%%%%%%%%%%%%%%%%%%%%%
\subsubsection{Minimizing properties}
{From now on, we assume that, for all $T >t_0$}, \; $e(\cdot)\in \cW^{1,1} (t_0,T; \cH)$. Let $(x,y):[t_0,+\infty[\to\cH\times\cH$ be the global strong solution to \eqref{eq:fos2} with Cauchy data $(x(t_0),y(t_0))=(x_0,y_0)\in\dom(f)\times\cH$.
For $t\geq t_0$ define
\begin{equation} \label{u0}
u(t)=\int_{t_0}^t \pa{-\beta e(s)+\pa{\frac{1}{\beta}-\frac{\alpha}{s}}x(s)-\frac{1}{\beta}y(s)}ds.
\end{equation}
Thus $u$ is continuously differentiable, with derivative satisfying
\begin{align}
\dot u(t)
& = -\beta e(t) +\pa{\frac{1}{\beta}-\frac{\alpha}{t}}x(t)-\frac{1}{\beta}y(t), \quad \forall t\geq t_0, \label{guniter} \\
& = \dot x(t)+\beta \xi(t),\mbox{\hspace{6em}for almost all }t>t_0, \label{gubis}
\end{align}
where $\xi(t) \in \partial f(x(t))$, and the last equality follows from Theorem~\ref{Thm-existence}\ref{existence-6}-\ref{existence-6:b}. Therefore, $u$ can be also written equivalently as
\[
u(t)=x(t)-x_0+ \beta \DS{ \int_{t_0}^t \xi(s)  ds}.
\]
With parts {\ref{existence-1} }and {\ref{existence-2}} of Theorem~\ref{Thm-existence}, equality \eqref{guniter} shows that $\dot u$ is absolutely continuous on 
any compact subinterval of $[t_0,+\infty[$, hence differentiable almost everywhere on $[t_0,+\infty[$. Therefore, 
\[
\ddot u(t)= -\beta \dot{e}(t)+\frac{\alpha}{t^2}x(t)+\pa{\frac{1}{\beta}-\frac{\alpha}{t}}\dot x(t) -\frac{1}{\beta}\dot y(t).
\]
The equality above, combined with 
$\dot y(t)=\frac{\alpha\beta}{t^2} x(t) +\dot x(t)+\beta(\xi(t) +e(t))$ (which is obtained by taking the difference of the two equations in \eqref{eq:fos2}), yields 
\begin{equation} \label{dguniter0}
\ddot u(t)=-\frac{\alpha}{t}\dot x(t)-\xi(t) -(e(t)+ \beta \dot{e}(t)),
\end{equation}
for almost all $t>t_0$. Using \eqref{gubis}, we obtain
\begin{align} 
\ddot u(t)
&=\pa{\frac{1}{\beta}-\frac{\alpha}{t}}\dot x(t)-\frac{1}{\beta}\dot u(t)  -(e(t)+ \beta \dot{e}(t))
\label{dguniter} 
%\\
%\ddot u(t)+\frac{\alpha}{t}\dot u(t)
%& = &
%-\pa{1-\frac{\alpha\beta}{t}}(\xi(t) + e(t)),
%\label{guter}
\end{align}
for almost all $t>t_0$. We will need the following energy function of the system, defined for all 
$T\geq t\geq t_0$ (recall \eqref{guniter} for the definition of $\dot u(t)$):
\begin{equation}\label{gE:W}
W_T(t) \eqdef \frac{1}{2}\|\dot u(t)\|^2 + f (x(t)) - \int_t^T \langle \dot{u}(\tau),e(\tau) +\beta \dot{e}(\tau)\rangle d\tau,
\end{equation}
and when the following expression is well-defined (we will prove it later)
\begin{equation}\label{gE:W_b}
W(t) \eqdef \frac{1}{2}\|\dot u(t)\|^2 + f (x(t)) - \int_t^{+\infty} \langle \dot{u}(\tau),e(\tau) +\beta \dot{e}(\tau)\rangle d\tau.
\end{equation}

\begin{theorem} \label{Thm-g-weak-conv2}
{Let $\alpha>0$. Suppose that  $\inf_{\cH} f > -\infty$. Suppose that $e(\cdot)\in \cW^{1,1} (t_0,T; \cH)$ for all $T > t_0$, with
$\DS{\int_{t_0}^{+\infty} \| e(t)\| <+\infty}$ and $\DS{\int_{t_0}^{+\infty} \| \dot{e}(t)\| <+\infty}$. Then
for any global strong solution of \eqref{eq:fos2},  $(x,y):[t_0,+\infty[\to\cH\times\cH$
\begin{enumerate}[label={\rm (\roman*)}]
\item \label{gweak-conv2-1}
$W$ is well-defined and non-increasing on $[t_1,+\infty[$ for some $t_1 \geq t_0$. 
\item \label{gweak-conv2-5} 
  $\DS{\int_{t_0}^{+\infty}\frac{1}{t}\|\dot x(t)\|^2dt<+\infty}$, 
  $\DS{\int_{t_0}^{+\infty}\frac{1}{t}\|\xi(t)\|^2dt<+\infty}$.
  \item \label{gweak-conv2-2} 
$\lim_{t\to+\infty}W(t)=\lim_{t\to+\infty}f(x(t))=\inf_{\cH} f \in \R\cup\{-\infty\}$, \ $\lim_{t\to+\infty}\|\dot x(t) +\beta \xi(t)\|=0$.
\item \label{gweak-conv2-3}
As $t\to+\infty$, every sequential weak cluster point of $x(t)$ belongs to $S$.
%\item  \label{gweak-conv2-4}
%If $\|x(t)\|\not\to+\infty$ as $t\to+\infty$, then $\argmin_\cH f\neq\emptyset$. 
\item\label{gweak-conv2-6} 
If, moreover, the solution set $S\neq\emptyset$ and $\DS{\int_{t_0}^{+\infty} \log t ~ \| e(t)\|  <+\infty}$ and $\DS{\int_{t_0}^{+\infty} \log t ~ \| \dot{e}(t)\| < +\infty}$, then 
\begin{enumerate}[label={\rm (\alph*)}]
  \item \label{gweak-conv2-6a}
  $f(x(t))-\inf_{\cH} f=\DS{\cO\pa{\frac{1}{\log t}}}$ and 
  $\|\dot u(t)\|=\DS{\cO\pa{\frac{1}{\sqrt{\log t}}}}$ as $t \to +\infty$.
  \item \label{gweak-conv2-6b}
  $\DS{\int_{t_0}^{+\infty}\frac{1}{t}(f(x(t))-\inf_{\cH} f)dt < +\infty}$.
  \end{enumerate}
\end{enumerate}
}
\end{theorem}
\begin{proof}
Since we are interested in asymptotic analysis, we can assume $t \geq t_1=\max\pa{t_0,2\alpha\beta}$.
\paragraph{Claim~\ref{gweak-conv2-1}}
According to Theorem \ref{Thm-existence}, $W_T$ is absolutely continuous. Taking the derivative and using the chain rule we get  
\[
\dot W_T(t)=\langle\dot u(t),\ddot u(t)\rangle+\langle\xi(t),\dot x(t)\rangle +\langle \dot{u}(t), e(t)+ \beta \dot{e}(t)\rangle,
\]
for almost every $T> t>t_0$. Now use \eqref{gubis} and \eqref{dguniter} to obtain
\begin{align*}
\dot W_T(t)& = 
\dotp{\dot u(t)}{\pa{\frac{1}{\beta}-\frac{\alpha}{t}}\dot x(t)
  -\frac{1}{\beta}\dot u(t) -(e(t)+ \beta \dot{e}(t))} +\dotp{\xi(t)}{\dot x(t)} +\dotp{ \dot{u}(t)}{e(t)+ \beta \dot{e}(t)} \nonumber \\ 
& =
\dotp{\dot u(t)}{\pa{\frac{1}{\beta}-\frac{\alpha}{t}}\dot x(t)
  -\frac{1}{\beta}\dot u(t)} +\dotp{\xi(t)}{\dot x(t)}  \nonumber \\
& =
-\frac{1}{\beta}\|\dot u(t)\|^2
  +   \dotp{\dot x(t)}{\pa{\frac{1}{\beta}-\frac{\alpha}{t}}\dot u(t) + \xi(t) } \nonumber \\ 
& = 
-\frac{1}{\beta}\|\dot u(t)\|^2
  +   \dotp{\dot x(t)}{\pa{\frac{1}{\beta}-\frac{\alpha}{t}}\dot u(t) + \frac{1}{\beta}(\dot u(t)-\dot x(t))} \nonumber \\ 
& =
-\frac{1}{\beta}\|\dot u(t)\|^2
  +   \dotp{\dot x(t)}{\pa{\frac{2}{\beta}-\frac{\alpha}{t}}\dot u(t) - \frac{1}{\beta}\dot x(t)} \nonumber \\
& = 
-\frac{1}{\beta}\|\dot u(t)\|^2 -\frac{1}{\beta}\|\dot x(t)\|^2
  +  \pa{\frac{2}{\beta}-\frac{\alpha}{t}} \dotp{\dot x(t)}{\dot u(t)} \nonumber \\
& \leq
  -\frac{\alpha}{2t}\|\dot x(t)\|^2-\frac{\alpha}{2t}\|\dot u(t)\|^2,
\end{align*}
for almost every 
$t \geq t_1$. So $W_T$ is non-increasing on $[t_1,+\infty[$, because it is absolutely continuous and its derivative is non-positive therein.  
Therefore  $W_T (t) \leq W_T (t_1)$ for all $t\in [t_1, T]$.
Equivalently
\[
\frac{1}{2}\|\dot u(t)\|^2 + f (x(t)) - \int_t^T \langle \dot{u}(\tau),e(\tau) +\beta \dot{e}(\tau)\rangle d\tau
\leq \frac{1}{2}\|\dot u(t_1)\|^2 + f (x(t_1)) - \int_{t_1}^T \langle \dot{u}(\tau),e(\tau) +\beta \dot{e}(\tau)\rangle d\tau.
\]
After simplification, and setting $C= \frac{1}{2}\|\dot u(t_1)\|^2 + f (x(t_1))- \inf f(\cH)$, we obtain
\[
\frac{1}{2}\|\dot u(t)\|^2 
\leq C - \int_{t_1}^t \langle \dot{u}(\tau),e(\tau) + \beta \dot{e}(\tau)\rangle d\tau.
\]
By Cauchy-Schwarz inequality we get
\[
\frac{1}{2}\|\dot u(t)\|^2 
\leq C + \int_{t_0}^t \|\dot{u}(\tau)\| \|e(\tau) +\beta \dot{e}(\tau)\| d\tau.
\]
According to {Gronwall's} Lemma~\ref{lem:BrezisA5}
\begin{equation}\label{gE:u_b}
\|\dot u(t)\| 
\leq \sqrt{2C} + \int_{t_0}^t  \|e(\tau) +\beta \dot{e}(\tau)\| d\tau \leq M \eqdef  \sqrt{2C} + \int_{t_0}^{+\infty}  \|e(\tau) +\beta \dot{e}(\tau)\| d\tau.
\end{equation}
So, $\|\dot u(t)\|$ is bounded on $[t_0, +\infty[$, which allows us to define
\begin{equation}\label{gE:W_bb}
W(t)=\frac{1}{2}\|\dot u(t)\|^2 + f (x(t)) - \int_t^{+\infty} \dotp{\dot{u}(\tau)}{e(\tau) +\beta \dot{e}(\tau)} d\tau.
\end{equation}
Noticing that $W$ and $W_T$ have the same derivative we conclude that
\begin{equation}\label{gE:W_bbb}
\dot{W}(t) +\frac{\alpha}{2t}\|\dot x(t)\|^2 +\frac{\alpha}{2t}\|\dot u(t)\|^2 \leq 0,
\end{equation}
and thus $W$ is non-increasing on $[t_1,+\infty[$.
\paragraph{Claim~\ref{gweak-conv2-5}}
Integrating \eqref{gE:W_bbb}, and using that $f$, and hence $W$, is bounded from below, we obtain,
\begin{equation}\label{gE:u_bb}
 \int_{t_0}^{+\infty}\frac{1}{t}\|\dot x(t)\|^2dt<+\infty, \qandq
 \int_{t_0}^{+\infty}\frac{1}{t}\|\dot u(t)\|^2dt<+\infty .
\end{equation}
Using Jensen's inequality, we get the integrability claim on $\xi(t)$. 
\paragraph{Claim~\ref{gweak-conv2-2}}
Given $z\in\cH$, let us define $h:[t_0,+\infty [ \to\R_+$ by
$
h(t)=\frac{1}{2}\|u(t)-z\|^2.
$
The function $h$ is continuously differentiable with
\[
\dot h(t) = \dotp{ u(t) - z }{\dot{u}(t)  },
\]
and  $\dot h$ is absolutely continuous on compact subintervals of $[t_0,+\infty[$ (since $\dot u$ is) and satisfies 
\[
\ddot h(t) = \dotp{ u(t) - z }{\ddot{u}(t)  } + \| \dot{u}(t) \|^2
\]
for almost every $t>t_0$. Using \eqref{gubis} and \eqref{dguniter0}
we get 
\[
\ddot u(t) + \frac{\alpha}{t} \dot u(t) =  -\pa{1-\frac{\alpha\beta}{t}} \xi(t) -(e(t) +\beta \dot{e}(t))  .
\]
Therefore, for almost every $t>t_0$
\begin{eqnarray*} 
&&\ddot h(t) + \frac{\alpha}{t} \dot h(t)  
 = 
  \|\dot u(t)\|^2
  - \dotp{ u(t)-z}{\pa{1-\frac{\alpha\beta}{t}} \xi(t)} -\dotp{ u(t)-z}{e(t) +\beta \dot{e}(t)}
\\
&& =
  \|\dot u(t)\|^2
  - \pa{1-\frac{\alpha\beta}{t}} \dotp{x(t)-z -x_0 + \beta \int_{t_0}^t\xi(s)ds}{\xi(t)} -\dotp{ u(t)-z}{e(t) +\beta \dot{e}(t)}
\\
&& \leq 
  \|\dot u(t)\|^2
  -\pa{1-\frac{\alpha\beta}{t}} \dotp{x(t)-z}{\xi(t)}
  -\pa{1-\frac{\alpha\beta}{t}}
    \dotp{-x_0+ \beta \int_{t_0}^t\xi(s)ds}{\xi(t)}
    + \|  e(t) +\beta \dot{e}(t) \| \|  u(t)-z \| .
\end{eqnarray*}
To interpret $ \left\langle-x_0+ \beta \int_{t_0}^t\xi(s)ds\,,\,\xi(t)\right\rangle$ as a temporal derivative, let us introduce
\begin{center}
$
I(t)=\frac{1}{2\beta}\left\|-x_0+ \beta\int_{t_0}^t\xi(s)ds\right\|^2.
$
\end{center}
Then $I(\cdot)$ is locally absolutely continuous and 
$
\dot I(t)=\dotp{-x_0+\DS{\int_{t_0}^t\xi(s)ds}}{\xi(t)}
$ 
almost everywhere, because $\xi\in L^2(t_0,T;\cH)\subseteq L^1(t_0,T;\cH)$ for all $T >t_0$; see part {\ref{existence-6}-\ref{existence-6:c}} of Theorem \ref{Thm-existence}. So, 
\[
\ddot h(t) + \frac{\alpha}{t} \dot h(t)\leq  
  \|\dot u(t)\|^2
  -\pa{1-\frac{\alpha\beta}{t}}\dotp{x(t)-z}{\xi(t)}
  -\pa{1-\frac{\alpha\beta}{t}}\dot I(t) +\|  e(t) +\beta \dot{e}(t) \| \|  u(t)-z \|,
\]
for almost every $t>t_0$. On the other hand, by convexity of $f$ and  $\xi(t) \in \partial f(x(t))$ 
\[
\dotp{x(t)-z}{\xi(t)} \geq f(x(t) - f(z).
\]
Therefore
\[
\ddot h(t) + \frac{\alpha}{t} \dot h(t) + \pa{1-\frac{\alpha\beta}{t}}(f(x(t)) - f(z)) +\pa{1-\frac{\alpha\beta}{t}}\dot I(t)\leq  \|\dot u(t)\|^2  + \|  e(t) +\beta \dot{e}(t) \| \|  u(t)-z \|.
\]
Using the definition \eqref{gE:W_bb} of $W$,  we get
\begin{multline*}
\ddot h(t) + \frac{\alpha}{t} \dot h(t) +\pa{1-\frac{\alpha\beta}{t}}(W(t) - f(z)) +\pa{1-\frac{\alpha\beta}{t}}\dot I(t) \leq  \pa{\frac{3}{2}-\frac{\alpha\beta}{2t}}\|\dot u(t)\|^2 \\
+ \|  e(t) +\beta \dot{e}(t) \| \|  u(t)-z \| -\pa{1-\frac{\alpha\beta}{t}}\int_t^{+\infty} \dotp{ \dot{u}(\tau)}{e(\tau) +\beta \dot{e}(\tau)} d\tau .    
\end{multline*}
According to \eqref{gE:W_bbb}, we have $\|\dot u(t)\|^2 \leq -\frac{2t}{\alpha} \dot{W}(t)$. Therefore,  
\begin{multline*}
\ddot h(t) + \frac{\alpha}{t} \dot h(t) +\pa{1-\frac{\alpha\beta}{t}}(W(t) - f(z)) +\pa{1-\frac{\alpha\beta}{t}}\dot I(t) \leq  -\pa{\frac{3t}{\alpha}-\beta } \dot W(t) \\
+ \|  e(t) +\beta \dot{e}(t) \| \|  u(t)-z \| -\pa{1-\frac{\alpha\beta}{t}}\int_t^{+\infty} \dotp{ \dot{u}(\tau)}{e(\tau) +\beta \dot{e}(\tau)} d\tau .    
\end{multline*}
Dividing by $t$ and rearranging the terms, we have with $g(t) \eqdef   e(t) +\beta \dot{e}(t) $
\begin{multline*}
\frac{1}{t}\ddot h(t)+\pa{\frac{1}{t}-\frac{\alpha\beta}{t^2}}\pa{W(t)-f(z)}\leq -\pa{\frac{3}{\alpha}-\frac{\beta}{t}}\dot W(t)-\left[\frac{\alpha}{t^2}\dot h(t)+ \pa{\frac{1}{t}-\frac{\alpha\beta}{t^2}}  \dot I(t)\right]\\
+\frac{1}{t} \|  g(t)  \| \|  u(t)-z \| -\pa{\frac{1}{t}-\frac{\alpha\beta}{t^2}}\int_t^{+\infty} \dotp{\dot{u}(\tau)}{g(\tau)} d\tau.
\end{multline*}
After integration, and using Lemma \ref{L:int_bounded}, we get
\begin{equation} \label{E:h_dot_pert}
\frac{1}{t}\dot h(t)+\int_{t_1}^t\pa{\frac{1}{s}-\frac{\alpha\beta}{s^2}}\big(W(s)-f(z)\big)\,ds \le -\int_{t_1}^t\pa{\frac{3}{\alpha}-\frac{\beta}{s}}\dot W(s)\,ds+C + K_1(t) + K_2 (t),
\end{equation} 
where
\[
K_1(t)= \int_{t_1}^t   \frac{1}{s}  \|  g(s)  \| \|  u(s)-z \| ds\ \qandq K_2 (t)= \int_{t_1}^t \pa{\frac{1}{s}-\frac{\alpha\theta}{s^2}}  \int_s^{\infty} \| \dot u (\tau)\| \| g(\tau) \| d\tau ds.
\]
Let us  majorize  $K_1 (t)$ and $K_2 (t)$. The relation 
\[
\norm{u(s)-z} \leq \norm{u(t_1)-z} + \int_{t_1}^{s}\norm{\dot{u}(\tau)} d\tau, 
\]
and $\dot{u}(\cdot)$ bounded (see \eqref{gE:u_b}) give
\[
K_1(t)\leq \int_{t_{1}}^{t}\frac{1}{s} \| g(s)\| \|u (s) -z \|  ds \leq \pa{  \frac{\Vert u(t_1)-z\Vert}{t_{1}}+
\sup_{t\ge t_1}\Vert \dot{u}(\tau)\Vert} \int_{t_{1}}^{+\infty}\Vert g(s)\Vert ds\leq C < +\infty.
\]
For $K_2 (t)$, we use again $\dot{u}(\cdot)$ bounded (see \eqref{gE:u_b})   and integration by parts to obtain
$$
K_2 (t)\leq C \int_{t_{1}}^{t} \pa{ \frac{1}{s}\int_{s}^{\infty}\norm{g(\tau)} d\tau}  ds \leq C\pa{ \log t \int_{t}^{\infty}\norm{g(\tau)} d\tau + \int_{t_{1}}^{t} \norm{g(\tau)}\log \tau  \ d\tau+1 } .
$$
Let us examine the integral terms that enter \eqref{E:h_dot_pert}.
Since $W(\cdot)$ is non-increasing
\begin{eqnarray} \label{E:int_W_pert}
\int_{t_1}^t\pa{\frac{1}{s}-\frac{\alpha\beta}{s^2}}\big(W(s)-f(z)\big)\,ds 
& \ge &  \pa{W(t)-f(z)}\int_{t_1}^t\pa{\frac{1}{s}-\frac{\alpha\beta}{s^2}}\,ds \nonumber \\
& = & \pa{W(t)-f(z)}\pa{\log t - \log t_1+\frac{\alpha\beta}{t}-\frac{\alpha\beta}{t_1}}.
\end{eqnarray}
In turn, integration by parts gives
\begin{eqnarray} \label{E:int_dot_W_pert}
&&-\int_{t_1}^t\pa{\frac{3}{\alpha}-\frac{\beta}{s}}\dot W(s)\,ds 
= \pa{\frac{3}{\alpha}-\frac{\beta}{t_1}}\big(W(t_1)-f(z)\big)
-\pa{\frac{3}{\alpha}-\frac{\beta}{t}}\pa{W(t)-f(z)} + \beta\int_{t_1}^t\frac{W(s)-f(z)}{s^2}\,ds \nonumber \\
&& \hspace{2.5cm}\le  \pa{\frac{3}{\alpha}-\frac{\beta}{t_1}}\big(W(t_1)-f(z)\big)
-\pa{\frac{3}{\alpha}-\frac{\beta}{t}}\pa{W(t)-f(z)}
+\beta\big(W(t_1)-f(z)\big)\pa{\frac{1}{t_1}-\frac{1}{t}},\nonumber\\
&& \hspace{2.5cm} \leq \frac{3}{\alpha}\big|W(t_1)-f(z)\big|-\pa{\frac{3}{\alpha}-\frac{\beta}{t}}\pa{W(t)-f(z)}
\end{eqnarray} 
since $t\mapsto W(t)-f(z)$ is non-increasing and $t\ge t_1\ge\alpha\beta$.
Combining \eqref{E:h_dot_pert} with \eqref{E:int_W_pert} and \eqref{E:int_dot_W_pert},  we obtain
\[
\frac{1}{t}\dot h(t)
+\pa{W(t)-f(z)}\pa{\log t+D+\frac{E}{t}} 
\leq C \pa{ \log t \int_{t}^{\infty}\norm{g(\tau)} d\tau + \int_{t_{1}}^{t} \norm{g(\tau)}\log \tau  \ d\tau+1 } 
\]
for appropriate constants $C,D,E\in\R$.
Now, take $t_2\ge t_1$ such that $\log s+D+\frac{E}{s}\ge 0$ for all $s\ge t_2$. Integrate from $t_2$ to $t$ and use again that $W$ is non-increasing to obtain
\begin{multline*}
\frac{h(t)}{t}-\frac{h(t_2)}{t_2}+\int_{t_2}^t\frac{h(s)}{s^2}ds +\pa{W(t)-f(z)}
\int_{t_2}^t\pa{\log s+D+\frac{E}{s}}ds\\
\leq C' \int_{t_2}^t  \pa{ \log s \int_{s}^{\infty}\norm{g(\tau)} d\tau + \int_{t_{1}}^{s} \norm{g(\tau)}\log \tau  \ d\tau + 1 }ds.
\end{multline*}
Since $h$ is non-negative, this implies
\begin{multline}\label{eq:bndWnonsmooth}
\pa{W(t)-f(z)} \pa{t\log t+(D-1)t+E\log t+  F} \\
\leq C'\pa{t + t \log t \int_{t}^{\infty}\norm{g(\tau)} d\tau  + \int_{t_{2}}^{t} \norm{ g(\tau)} \tau \log \tau  \ d\tau  +   t \int_{t_{1}}^{t}\norm{g(\tau)} \log \tau d\tau }  + G,
\end{multline}
for some appropriate constants $D,E,F,G\in\R$. Divide by $t \log t$, let $t\to+\infty$, and use Lemma \ref{basic-int}, to obtain $\lim_{t\to+\infty}W (t) \leq f(z)$. 
The integrability of $g$ and $\dot{u}(\cdot)$ bounded (see \eqref{gE:u_b}) yield $\DS{\lim_{t\to+\infty} \int_t^{+\infty} \dotp{\dot{u}(\tau)}{g(\tau)}d\tau =0}$. As a consequence, 
\[
\lim_{t\to+\infty}\pa{  f(x(t))+\frac{1}{2}\|\dot x(t)+\beta \xi(t)\|^2 } \leq f(z)
\]
for each $z \in \cH$. Thus 
\[
\inf_{\cH} f \leq \liminf_{t \to +\infty} f(x(t)) \leq \limsup_{t \to +\infty} f(x(t)) \leq \lim_{t \to +\infty} \pa{f(x(t))+\frac{1}{2}\|\dot x(t)+\beta \xi(t)\|^2} \leq \inf_{\cH} f,
\]
whence we get $\lim_{t\to+\infty}f (x(t))= \inf_{\cH} f $, and thus $\lim_{t\to+\infty} \|\dot x(t)+\beta \xi(t)\|  =0$. 
%
%\smallskip
%
\paragraph{Claim~\ref{gweak-conv2-3}} This follows from claim~{\ref{gweak-conv2-2}} and lower semicontinuity of $f$.
\paragraph{Claim~\ref{gweak-conv2-6}-\ref{gweak-conv2-6a}} Let $x^\star \in S$. We start from \eqref{eq:bndWnonsmooth} with $z=x^\star$ and divide by $t$. To conclude, we note that
\begin{gather*}
\log t \int_{t}^{\infty}\norm{g(\tau)} d\tau \leq \int_{t}^{\infty}\log \tau \norm{g(\tau)} d\tau < +\infty, \\
\int_{t_{2}}^{t} \norm{ g(\tau)} \frac{\tau}{t} \log \tau  \ d\tau \leq \int_{t_{2}}^{t} \norm{ g(\tau)} \log \tau  \ d\tau < +\infty \qandq \\
\log t \int_t^{+\infty} \dotp{\dot{u}(\tau)}{g(\tau)}d\tau \leq C \int_t^{+\infty} \log \tau\norm{g(\tau)}d\tau,
\end{gather*}
where $C = \sup_{t \geq t_0} \norm{\dot{u}(t)} < +\infty$ (see \eqref{gE:u_b}).
\paragraph{Claim~\ref{gweak-conv2-6}-\ref{gweak-conv2-6b}} Putting together \eqref{E:h_dot_pert} and \eqref{E:int_dot_W_pert} with $z=x^\star \in S$, and using non-negativity of $h$, we infer that for some positive constant $C$
\[
\pa{1-\frac{\alpha\beta}{t_1}}\int_{t_1}^t\frac{1}{s}\pa{W(s)-f(z)} ds \leq C + K_2 (t).
\]
Arguing similarly as for proving part {\ref{gweak-conv2-6}-\ref{gweak-conv2-6a}}, we can show that $K_2(\cdot)$ is bounded. Thus
\[
\int_{t_1}^t\frac{1}{s}\pa{f(x(s))-f(z) + \norm{\dot{u}(s)}^2} ds \leq \int_{t_1}^t\frac{1}{s}\pa{W(s)-f(z)} ds + \frac{1}{t_
1}\int_{t_1}^t\norm{\dot{u}(s)}\norm{g(s)} ds < +\infty,
\]
which completes the proof. \qed
\end{proof}

%%%%%%%%%%%%%%%%%%%%%%%%%%%%%%%%%%%%%%%%%%%%%%%%%%%%%%%%%%%%%%%%%%%%%%%%%%%%%%%%%%%%%%%%%%%%%%%%%%%%%%%%%%%%%%%
\subsubsection{Fast convergence rates}

When $\alpha \geq 3$, under a reinforced integrability assumption on the perturbation term, we will show fast convergence results. The following theorem is the non-smooth counterpart of Theorem \ref{fast_conv_smooth}. 

\begin{theorem}\label{fastconv-thm}
Suppose that  $\alpha \geq 3$.
Let $f \in \Gamma_0(\cH)$ such that $S \neq \emptyset$. Suppose that $e(\cdot) \in \cW^{1,1} (t_0,T; \cH)$ for all $T >t_0$, with $\DS{\int_{t_0}^{+\infty} t \| e(t) + \beta \dot{e}(t)\| dt<+\infty}$. Then,
for any global strong solution $(x,y)$ of \eqref{eq:fos2}
\begin{enumerate}[label={\rm (\roman*)}]
\item \label{gfast-conv2-1} $f(x(t))- \min_{\cH} f = \cO\pa{ t^{-2}}.$
\item\label{gfast-conv2-2} 
  $\DS{\int_{t_0}^{+\infty} t(f(x(t))-\min_{\cH} f)dt<+\infty}$,
  $\DS{\int_{t_0}^{+\infty} t^2\|\xi(t)\|^2dt<+\infty}$, 
  $\DS{\int_{t_0}^{+\infty} t\|\dot{x}(t)\|^2dt<+\infty}$.
\item\label{gfast-conv2-3} 
  $\|\dot x(t)+\beta\xi(t)\|=\cO(t^{-1})$.
\end{enumerate} 
\end{theorem}
\begin{proof}
Let $(x,y):[t_0,+\infty[\to\cH\times\cH$ be a global strong solution of \eqref{eq:fos2}. Take $\alpha\geq3$ and $x^\star \in S$. Recall $\bar{f} \eqdef \min_{\cH} f$ and $g(t)=  e(t)+\beta \dot{e}(t)$. Our analysis relies on the non-smooth version of the Lyapunov function in \eqref{eq:Eeps}, which is defined for $\lambda\in[2,\alpha-1]$, as  $\cE_{\lambda,T}:[t_0,T] \to \R$ by  
\begin{equation}
\cE_{\lambda,T}(t)=
  t(t-\beta(\lambda+2-\alpha))(f(x(t))-\bar{f})+
  \frac{1}{2}\|v_\lambda(t)\|^2  
  +
  \lambda(\alpha-\lambda-1)\frac{1}{2}\|x(t)-x^\star\|^2 
  - \int_t^{T} \tau \dotp{v_\lambda(\tau)}{g(\tau)} d\tau, \label{grfast2}
\end{equation}
where $v_\lambda(t) \eqdef \lambda(x(\tau)-x^\star)+\tau\dot u(t)$, and $u$ is defined on $[t_0,+\infty[$ by \eqref{u0} and $\dot u$ is given by \eqref{guniter}.\\
 $\cE_{\lambda,T}(\cdot)$ is the sum of four terms, 
each of which is  absolutely continuous on $[t_0,T]$ for all $T>t_0$. Hence $\cE_{\lambda,T}$ is differentiable almost everywhere. We first differentiate each term of $\cE_{\lambda,T}$:
\begin{equation*}
\frac{d}{dt}\brac{t(t-\beta(\lambda+2-\alpha))(f(x(t))-\bar{f})}
= (2t-\beta(\lambda+2-\alpha))(f(x(t))-\bar{f})
+  t(t-\beta(\lambda+2-\alpha))\dotp{\xi(t)}{\dot x(t)}.
\end{equation*}
Using \eqref{dguniter0}, we have  
\begin{align*}
&\frac{d}{dt}\frac{1}{2}\|v_\lambda(t)\|^2 
= \dotp{\lambda(x(t)-x^\star)+t\dot u(t)}{
  \lambda\dot x(t)+\dot u(t)+t\ddot u(t)} \\
&=
  \dotp{\lambda(x(t)-x^\star)+t\dot u(t)}{
  (\lambda+1-\alpha)\dot x(t)-(t-\beta)\xi(t)-t g(t)} \\
&= 
  \lambda(\lambda+1-\alpha)\dotp{ x(t)-x^\star}{\dot x(t)} 
  -t(\alpha-\lambda-1)\|\dot x(t)\|^2 
  -\beta t(t-\beta)\|\xi(t)\|^2 \\
&-\lambda(t-\beta)\dotp{ x(t)-x^\star}{\xi(t)}
  -t(t-\beta(\lambda+2-\alpha))\dotp{\xi(t)}{\dot x(t)}
 -t \dotp{v_\lambda(t)}{g(t)}, \\
&\frac{d}{dt}\lambda(\alpha-\lambda-1)\frac{1}{2}\|x(t)-x^\star\|^2
=\lambda(\alpha-\lambda-1)\dotp{x(t)-x^\star}{\dot x(t)}, \qandq \\
&\frac{d}{dt}\pa{ - \int_t^{T} \tau \dotp{v_\lambda(\tau)}{g(\tau)} d\tau}
=  t \dotp{v_\lambda(t)}{g(t)} .
\end{align*}
By collecting these results, the perturbation terms cancel each other out. We get 
\begin{eqnarray}
\frac{d}{dt}\cE_{\lambda,T}(t)
&=&  (2t-\beta(\lambda+2-\alpha))(f(x(t))-\bar{f})
  -\lambda(t-\beta)\dotp{x(t)-x^\star}{\xi(t)} \nonumber \\
 && -t(\alpha-\lambda-1)\|\dot x(t)\|^2 
  -\beta t(t-\beta)\|\xi(t)\|^2,  \label{gdrfast2}
\end{eqnarray}
for almost all $t>t_0$. Since $\xi(t)\in\partial  f(x(t))$ for all $t>t_0$, we have 
$
\dotp{\xi(t)}{x(t)-x^\star}\geq f(x(t)-f(x^\star),
$
and we deduce from \eqref{gdrfast2}, that
\begin{equation} \label{grfastoche}
\frac{d}{dt}\cE_{\lambda,T}(t)\leq
  -((\lambda-2)t-\beta(\alpha-2))(f(x(t))-\bar{f})
  -t(\alpha-\lambda-1)\|\dot x(t)\|^2-\beta t(t-\beta)\|\xi(t)\|^2,
\end{equation}
for almost all $t\geq t_1=\max\pa{t_0,\beta}$. It follows that $\cE_{\lambda,T}$ is non-increasing on $[t_1,T]$. In particular, $ \cE_{\lambda,T}(t) \leq \cE_{\lambda,T}(t_1)$ for $t_1 \leq t \leq T$. This gives the existence of a constant $C$ 
%(depending only on the Cauchy data, and the parameters)
 such that
\begin{align}\label{Gronw2}
\frac{1}{2}\|v_\lambda(t)\|^2 \leq C + \int_{t_0}^t \|v_\lambda(t)\| \| \tau g(\tau) \| d\tau .
\end{align}
Applying Lemma~\ref{lem:BrezisA5} to \eqref{Gronw2}, and using the integrability of $t\mapsto tg(t)$, it follows that
\begin{equation} \label{energy-033} 
\sup_{t \geq t_0} \|v_\lambda(t)\|  \leq  \sqrt{2C} + \int_{t_0}^\infty \| \tau g(\tau) \| d\tau< +\infty.
\end{equation} 
As a consequence, we can define the energy function
\begin{equation*}
 \cE_{\lambda} (t) \eqdef  \ t(t-\beta(\lambda+2-\alpha))(f(x(t))-\bar{f})+
  \frac{1}{2}\|v_\lambda(t)\|^2+
% \frac{1}{2}\|\lambda(x(t)-x^\star)+t(\dot x(t)+\beta\nabla\Phi(x(t)))\|^2+
  \lambda(\alpha-\lambda-1)\frac{1}{2}\|x(t)-x^\star\|^2 \\
  - \int _t^{\infty} \tau \dotp{v_\lambda(t)}{g(\tau)} d\tau,
\end{equation*}
which has the same derivative as $\cE_{\lambda,T}$. Hence $\cE_{\lambda} (t) \leq \cE_{\lambda} (t_0)$. Combined with \eqref{energy-033}, this gives
\[
t(t-\beta(\lambda+2-\alpha))(f(x(t))-\bar{f})
\leq C + \sup_{t \geq t_0} \|v_\lambda(t) \|  \int _{t_0}^{\infty} \|\tau   g(\tau) \| d\tau < + \infty ,
\]
whence statement {\ref{gfast-conv2-1}}. Claim~{\ref{gfast-conv2-3}} is obtained by letting $\lambda=0$ in \eqref{energy-033}. Integration of \eqref{grfastoche} gives the integral estimates of {\ref{gfast-conv2-2}}, which completes the proof. \qed
\end{proof}

%%%%%%%%%%%%%%%%%%%%%%%%%%%%%%%%%%%%%%%%%%%%%%%%%%%%%%%%%%%%%%%%%%%%%%%%%%%%%%%%%%%%%%%%%%%%%%%%%%%%%%%%%%%%%%%
\subsubsection{Convergence of the trajectories and faster asymptotic rates}
Similar argument as in the smooth case (see Theorem~\ref{T:weak_convergence-smooth}), but now using the Lyapunov function \eqref{grfast2}, gives weak convergence of the trajectories of \eqref{eq:fos2}. Moreover, in the same vein as Theorem~\ref{fast_conv_smooth}, $o(\cdot)$ rates can also be obtained. We leave the details to the readers for the sake of brevity. 
\begin{theorem}\label{T:weak_convergence-nonsmooth}
Let $\alpha>3$. Let $f \in \Gamma_0(\cH)$ and assume that $S =\argmin f \neq \emptyset$. Suppose that  $e(\cdot)\in \cW^{1,1} (t_0,T; \cH)$ for all $T >t_0$, with
$\DS{\int_{t_0}^{+\infty} t \| e(t) + \beta \dot{e}(t)\| dt<+\infty}$. Then,
for any global strong solution $(x,y)$ of \eqref{eq:fos2}
\begin{enumerate}[label={\rm (\roman*)}]
\item $x(t)$ converges weakly, as $t\to+\infty$ to a point in $S$;
\smallskip
\item $f(x(t))- \min_{\cH} f =  o\pa{t^{-2}}$  
 and  $\|\dot x(t)+\beta\xi(t)\|  =  o\pa{t^{-1}}$ as $t\to+\infty$. 
\end{enumerate}
\end{theorem}

\begin{remark}
In the special case where $e \equiv 0$ in \eqref{eq:fos2}, \ie, unperturbed case, Theorem~\ref{strong-conv-thm-implicit}, Theorem~\ref{fastconv-thm} and Theorem~\ref{T:weak_convergence-nonsmooth} recover the result of \cite[Section~4]{APR1}. In a nutshell, our result demonstrates that the properties of the unperturbed system are preserved under reasonable integrability conditions on the errors.
\end{remark}

%%%%%%%%%%%%%%%%%%%%%%%%%%%%%%%%%%%%%%%%%%%%%%%%%%%%%%%%%%%%%%%%%%%%%%%%%%%%%%%%%%%%%%%%%%%%%%%%%%%%%%%%%%%%%%%
\subsection{Implicit Hessian Damping}\label{sec:nonsmooth-implicit}
As we have already discussed in the smooth case (see Section~\ref{s:discussion}), the analysis of the convergence properties of the system with implicit Hessian driven damping heavily relies on Lipschitz continuity of the gradient. As shown above, such a property was not needed to analyze system \eqref{eq:fos2}. Therefore, the study of the convergence properties for the non-smooth system \eqref{eq:fos2_implicit} (even without perturbations) is an open challenging topic.
\section{Numerical Experiments}\label{s:num}
To support our theoretical claims, we consider numerical examples in $\cH=\R^2$ with two real-valued functions: 
\begin{enumerate}[label=$\bullet$]
\item The first one is given by $f(x_1,x_2)) = (x_1-1)^4+(x_2-5)^2$. This function is obviously convex (but not strongly so) and smooth, and has a unique minimizer  at $(1,5)$. For this function, we consider the continuous time dynamical system \eqref{eq:origode_a} with parameters $(\alpha,\beta)=(3.1,1)$, and \eqref{eq:odetwo_a} with parameters $(\alpha,\gamma,\beta)=(3.1,1,1)$. 

\item The second example we consider is with the convex non-smooth function $f(x_1,x_2)=(x_1-1)^4+(x_2-5)^2+0.1 (|x_1|+|x_2|)$. For this function, we use the continuous time non-smooth system \eqref{eq:fos2} with parameters $(\alpha,\beta)=(3.1,1)$. Although we have no theoretical guarantee for system \eqref{eq:fos2_implicit}, we do report the corresponding numerical results with parameters $(\alpha,\gamma,\beta)=(3.1,1,1)$. 
\end{enumerate}
\begin{figure}[htbp]
\subfigure[Explicit Hessian damping: smooth function]{\includegraphics[width=0.85\textwidth]{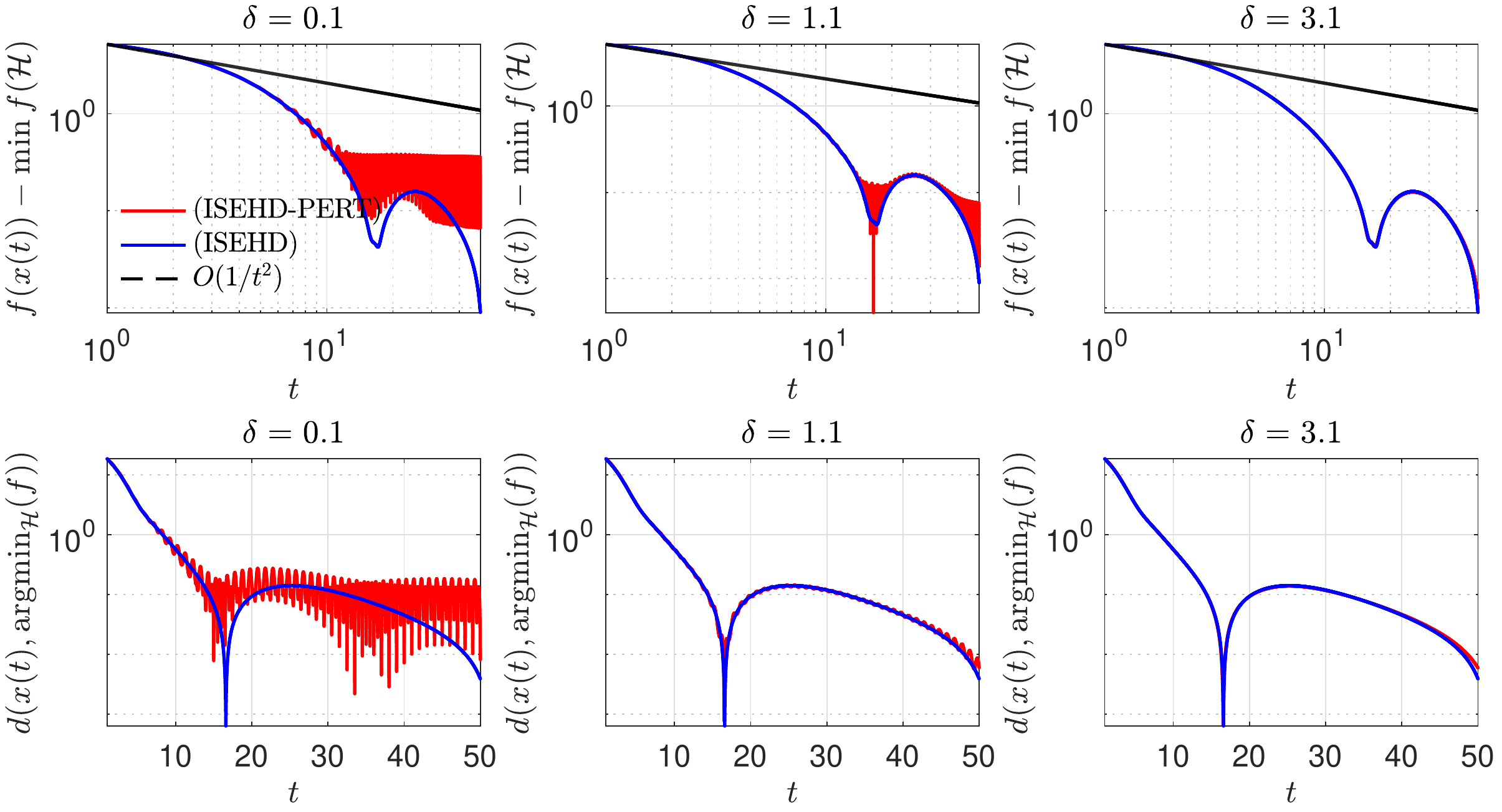}}\\
\subfigure[Implicit Hessian damping: smooth function]{\includegraphics[width=0.85\textwidth]{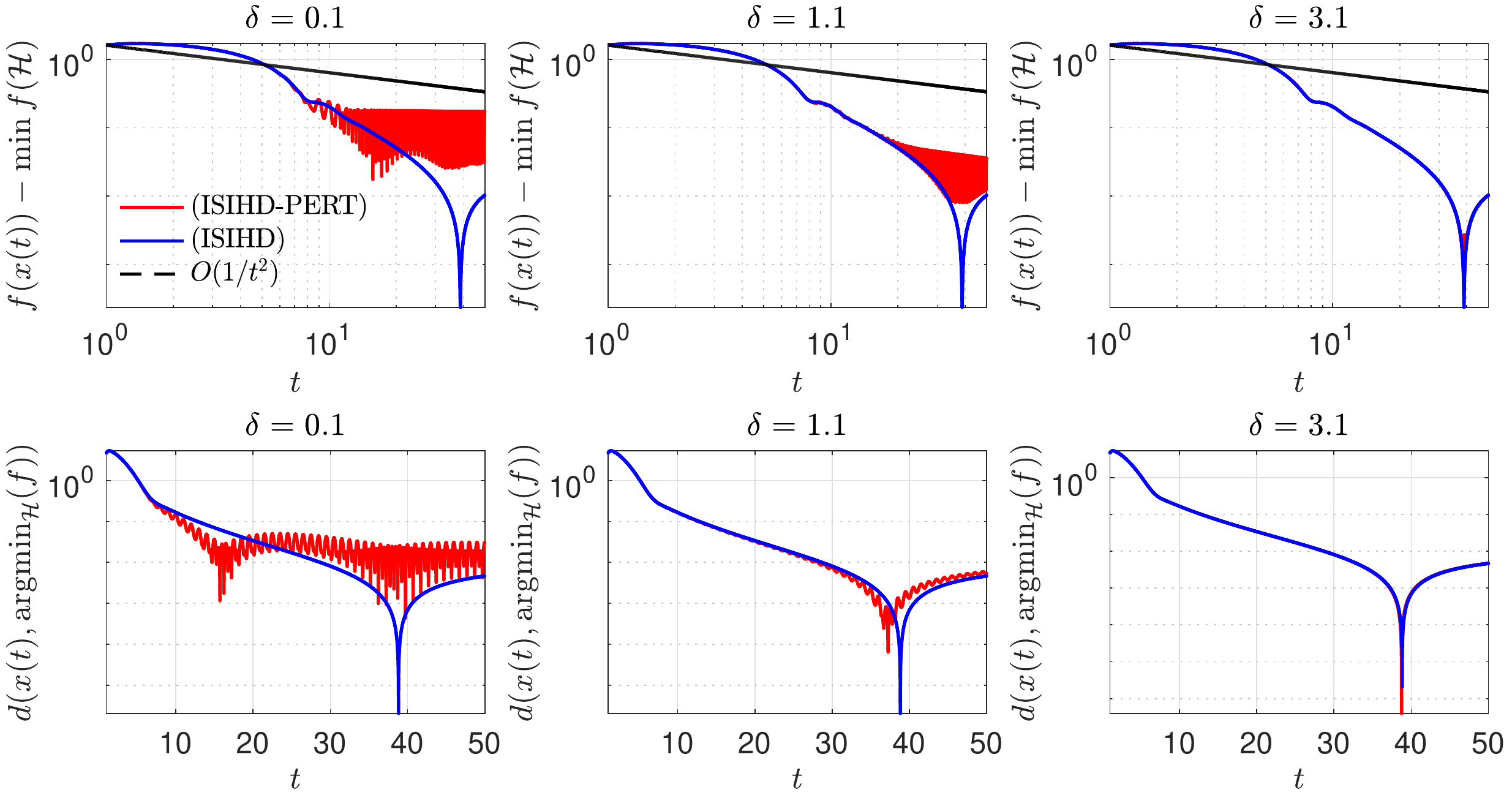}}
\caption{Example on a smooth function: Evolution of the objective error and distance to the minimizer as a function of $t$ for different error decay exponents.}
\label{fig:smoothhessianpert}    
\end{figure}
For both examples, we take as an exogenous perturbation
$$
e(t) = \displaystyle{\frac{\cos(2\pi t)}{t^{\delta}}}  \mbox{  with } \delta \in \{0.1, 1.1, 3.1\}.
$$
All systems are solved numerically with a Runge-Kutta adaptive method in MATLAB on the time interval $[1,50]$ with initial data $(x_0,\dot{x}_0)= (-10,20,5,-5)$. The results are displayed in Figure~\ref{fig:smoothhessianpert} and Figure~\ref{fig:nonsmoothhessianpert}. 

Let us first comment on the results for the smooth function. For $\delta=3.1$, all required moment assumptions on the errors are fulfilled (for the explicit Hessian, the term $\dot{e}$ is dominated by $e$ and can then be discarded). Hence the fast rates predicted by Theorem~\ref{fast_conv_smooth}\ref{item:fast_conv_smooth1} and Theorem~\ref{th:int}\ref{th:intclaim1} as well as convergence of the trajectories (see Theorem~\ref{T:weak_convergence-smooth} and Theorem~\ref{T:weak_convergence-smooth-implicit}) hold true. For the value $\delta=0.1$, since the error is not even integrable, neither the convergence of the objective value nor that of the trajectories is ensured, with large oscillations appearing. The implicit Hessian damping seems also less stable as anticipated from our discussion in Section~\ref{s:discussion}. For $\delta=1.1$, though there is no convergence guarantee for the trajectory, the objective value for \eqref{eq:origode_a} decreases but at a rate which is dominated by the error decrease. This can be explained in light of the proof of Theorem~\ref{fast_conv_smooth}\ref{item:fast_conv_smooth1}, where  a close inspection of \eqref{eq:boundvexplicit} and \eqref{eq:boundobjexplicit} shows that the bound on the objective error decomposes as
\begin{center}
$
f(x(t))-\bar{f} \leq \cO\pa{\frac{1}{t^2}} + \frac{C\pa{\DS{\int_{t_0}^t\tau \norm{e(\tau)}} d\tau}^2}{t^2} .
$ 
\end{center}
For $\delta \in ]1,2]$, the second term indeed dominates the first one and decreases at the slower rate $t^{-2(\delta-1)}$. This confirms the known rule that there is a trade-off between fast convergence of the methods and their robustness to perturbations.
Similar observations remain true for the non-smooth function with system \eqref{eq:fos2} where we now invoke Theorem~\ref{fastconv-thm} and Theorem~\ref{T:weak_convergence-nonsmooth}. As for system \eqref{eq:fos2_implicit}, it seems that it has a behaviour similar to what we observed in the smooth case for system \eqref{eq:odetwo_a}. As we argued in Section~\ref{sec:nonsmooth-implicit}, supplementing the numerical observations for system \eqref{eq:fos2_implicit} with theoretical guarantees is an open problem that we leave to a future work. 

\begin{figure}[htbp]
\subfigure[Explicit Hessian damping: non-smooth function]{\includegraphics[width=0.85\textwidth]{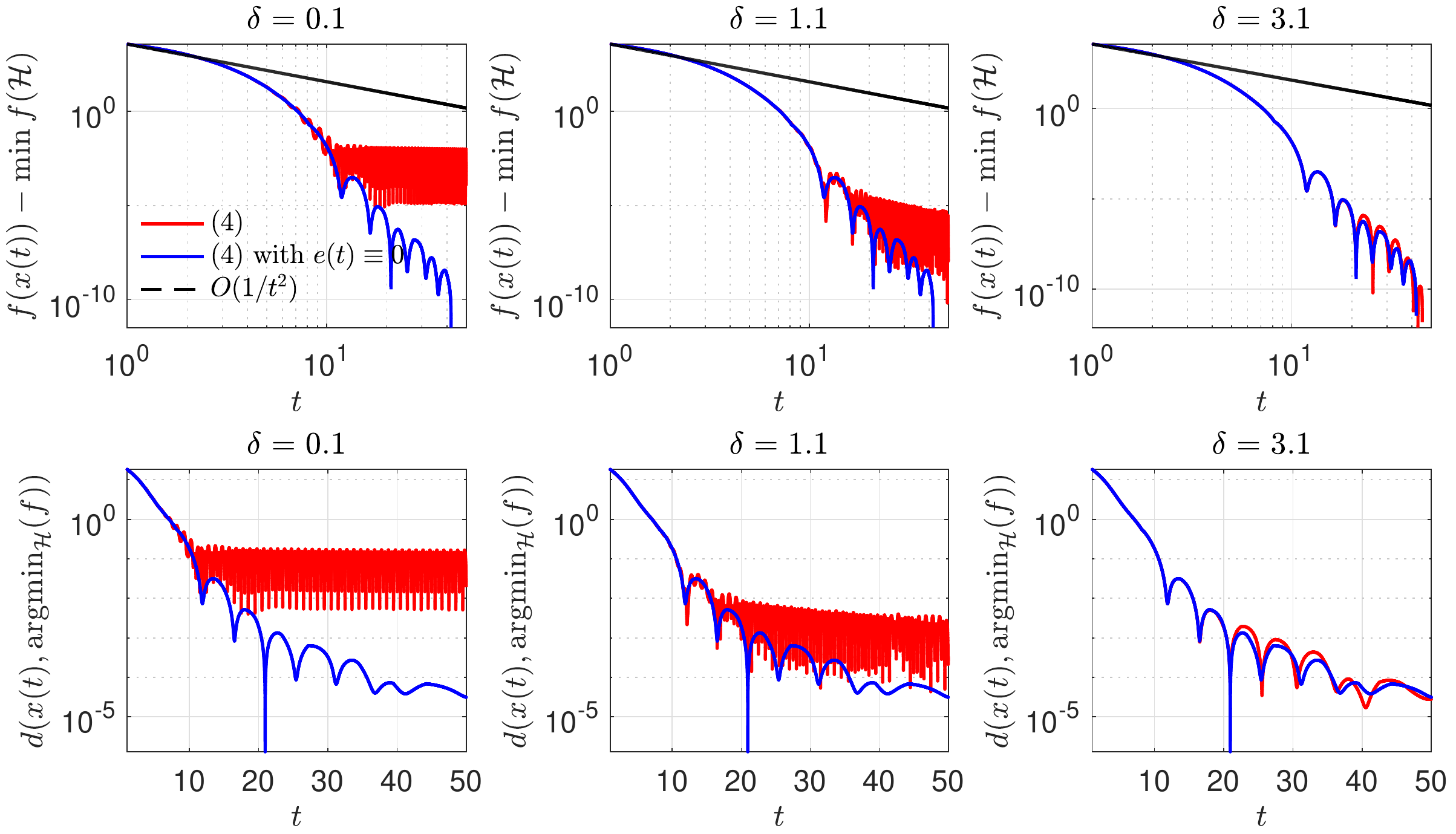}}\\
\subfigure[Implicit Hessian damping: non-smooth function]{\includegraphics[width=0.85\textwidth]{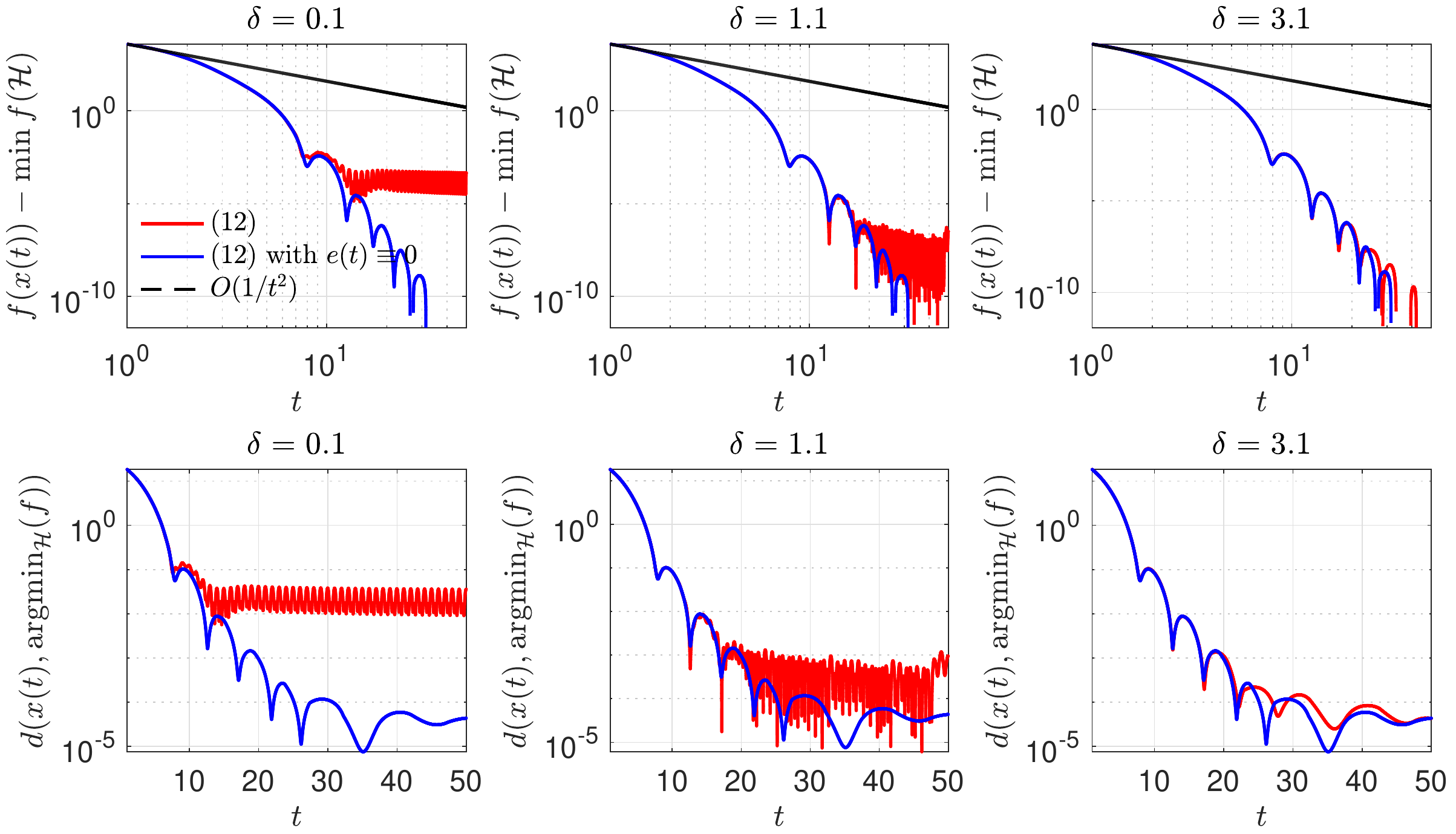}}
\caption{Example on a non-smooth function: Evolution of the objective error and distance to the minimizer as a function of $t$ for different error decay exponents.}
\label{fig:nonsmoothhessianpert}    
\end{figure}

%%%%%%%%%%%%%%%%%%%%%%%%%%%%%%%%%%%%%%%%%%%%%%%%%%%%%%%%%%%%%%%%%%%%%%%%%%%%%%%%%%%%%%%%%%%%%%%%%%%%%%%%%%%%%%%
\section{Conclusion and Perspectives}\label{s:conc}
%%%%%%%%%%%%%%%%%%%%%%%%%%%%%%%%%%%%%%%%%%%%%%%%%%%%%%%%%%%%%%%%%%%%%%%%%%%%%%%%%%%%%%%%%%%%%%%%%%%%%%%%%%%%%%%

The introduction of the correction term attached to the  damping driven by the Hessian in first-order accelerated optimization algorithms makes it possible to considerably dampen the oscillations in the trajectory.
The study of the robustness of these algorithms with respect to error perturbations is crucial for their further development in a stochastic framework.
Our systematic study of these questions for the dynamics underlying these algorithms is a fundamental first step in this direction. We paid particular attention to the explicit and the implicit forms of the Hessian driven damping, showing several advantages of the explicit form.
Our study concerns the dynamics with damping driven by the Hessian within the framework of the Nesterov acceleration gradient method. It shows that the convergence of the values still holds when the error terms satisfy an appropriate integrability condition, and fast convergence is satisfied when the (first  or second-order) moment of the error is finite.
Indeed, as a general rule, there is a balance between the rate of convergence of the methods and their robustness with respect to error disturbances.
An interesting technique studied in \cite{AA1,AA2} is the introduction of a dry friction term. This makes it possible to have errors which do not necessarily go to zero, they must not exceed a certain threshold, but on the other hand we only obtain an approximate solution.
Finding the right balance between the convergence rate and robustness is an important issue that should be the subject of further study.
Another important aspect of our study is the fact that several results are valid in the case of a non-smooth function. This opens the door to the study of similar topics with respect to structured composite optimization problems involving a non-smooth term.
These are some of the many facets of these flexible dynamics and algorithms which, in the unperturbed case, have been applied in various fields including PDE's and mechanical shocks \cite{AMR}, deep learning \cite{CBFP}, non-convex optimization \cite{ABC}, monotone inclusions \cite{AL1,AL2} to mention a few important applications.

%%%%%%%%%%%%%%%%%%%%%%%%%%%%%%%%%%%%%%%%%%%%%%%%%%%%%%%%%%%%%%

\appendix
\section{Auxiliary results}
%%%%%%%%%%%%%%%%%%%%%%%%%%%%%%%%%%%%%%%%%%%%%%%%%%%%%%%%%%%%%%%%%%%%%%%%%%%%%%%%%%%%%%%%%%%%%%%%%%%%%%%%%%%%%%%
Let us first recall the continuous form of the Opial's Lemma \cite{Op}, a key ingredient to establish convergence of the trajectories.
\begin{lemma}\label{Opial} Let $S$ be a nonempty subset of $\cH$ and let {$x:[t_0,+\infty[\to \cH$}. Assume that 
\begin{enumerate}[label={\rm (\roman*)}]
\item for every $z\in S$, $\lim_{t\to\infty}\|x(t)-z\|$ exists;
\item every weak sequential cluster point of $x(t)$, as $t\to\infty$, belongs to $S$.
\end{enumerate}
Then $x(t)$ converges weakly as $t\to\infty$ to a point in $S$.
\end{lemma}

%%%%%%
\begin{lemma}[{\cite[Lemma~7.3]{APR1}}]
\label{L:int_bounded}
Let $\tau,p>0$ and let $\psi:]\tau,+\infty[\to\R$ be $\cC^2(]\tau,+\infty[)$ and bounded from below. Then, 
\[
\inf_{t>\tau}\int_{\tau}^{t}\frac{\dot\psi(s)}{s^p}\,ds>-\infty\quad\hbox{and}\quad \inf_{t>\tau}\int_{\tau}^{t}\frac{\ddot\psi(s)}{s^p}\,ds-\frac{\dot\psi(t)}{t^p}>-\infty.
\]
\end{lemma}

% \begin{proof}
% By subtracting $\inf\psi$ we may assume that $\psi$ is non-negative. Using integration by parts, we obtain
% \[
% \int_{\tau}^{t}\frac{\dot\psi(s)}{s^p}\,ds=\frac{\psi(t)}{t^p}-\frac{\psi(\tau)}{\tau^p}+p\int_\tau^t\frac{\psi(s)}{s^{p+1}}\,ds\ge -\frac{\psi(\tau)}{\tau^p}.
% \]
% In a similar fashion, we deduce that
% \[
% \int_{\tau}^{t}\frac{\ddot\psi(s)}{s^p}\,ds-\frac{\dot\psi(t)}{t^p}=\frac{\dot\psi(\tau)}{\tau^p}+p\int_\tau^t\frac{\dot\psi(s)}{s^{p+1}}\,ds \ge \frac{\dot\psi(\tau)}{\tau^p}-p\frac{\psi(\tau)}{\tau^{p+1}},
% \]
% and we conclude.
% \end{proof}

%\begin{lemma}\label{basic-int} 
%Take $\delta >0$, and let $g \in L^1 (\delta , +\infty;\R)$ be non-negative and continuous. Consider a nondecreasing function $\psi:]\delta,+\infty[\to]0,+\infty[$ such that $\lim\limits_{t\to+\infty}\psi(t)=+\infty$. Then, 
%$$\lim_{t \to + \infty} \frac{1}{\psi(t)} \int_{\delta}^t \psi(s) g(s)ds =0.$$ 
%\end{lemma}

%%%%%%
\begin{lemma}[{\cite[Lemma~A.5]{Bre1}}]\label{lem:BrezisA5} 
Let $m : [t_0; T] \to [0,+\infty[$ be integrable. Suppose $w : [t_0,T] \to \R$ is continuous and
\[
\frac{1}{2} w(t)^2 \leq \frac{1}{2} c^2 + \int_{t_0}^t m(s)w(s) ds,
\]
for some $c\geq 0$ and for all $t \in [t_0,T]$. Then
\[
|w(t)| \leq c + \int_{t_0}^t m(s) ds, \qquad t \in [t_0,T] .
\]
\end{lemma}

%%%%%%
\begin{lemma}\label{lem:timerescale}
Let $\beta$ be a positive function on $[t_0,+\infty[$ such that $\beta \not\in L^1(t_0,+\infty;\R_+)$. Then, the differential inclusion 
\begin{equation}
\dot{z}(t) + \beta(t) \partial \Phi(z(t)) + F(t,z(t)) \ni 0 ,
\end{equation}
is equivalent to
\begin{equation}
\dot{w}(s) +  \partial \Phi (w(s)) + G(s, w(s)) \ni 0,
\end{equation}
with 
\[
G(s, w(s))= \frac{1}{\beta(\tau(s))}F(\tau(s),w(s)), \quad t = \tau(s), \qandq \beta(\tau(s))\dot{\tau}(s)=1 .
\]
\end{lemma}

\begin{proof}
Make the change of time variable
$
t = \tau(s) \qandq z(t)= z \circ \tau(s) = w(s) .
$
We then have 
%\begin{equation*}
%\frac{\dot{\tau}(s)}{\beta(\tau(s))\dot{\tau}(s)}\dot{z}(\tau(s)) +  \partial \Phi (z(\tau(s))) + \frac{1}{\beta(\tau(s))}F(\tau(s),z(\tau(s))) \ni 0 ,
%\end{equation*}
%or equivalently
\begin{equation*}
\frac{1}{\beta(\tau(s))\dot{\tau}(s)}\dot{w}(s) +  \partial \Phi (w(s)) + \frac{1}{\beta(\tau(s))}F(\tau(s),w(s)) \ni 0 ,
\end{equation*}
Choose $\tau(\cdot)$ such that 
\begin{equation}\label{basic1}
\beta(\tau(s))\dot{\tau}(s)=1 .
\end{equation}
Introduce a primitive of $\beta$,
$
p(t) = \int_{t_0}^t  \beta(r)dr
$
Therefore, \eqref{basic1} can be equivalently written
\begin{center}
$
\frac{d}{ds} p(\tau (s))=1 \iff p(\tau (s)) = \int_{t_0}^{\tau(s)}  \beta(r)dr = s + C,
$
\end{center}
for any constant $C$. Thus, $\tau$ defines a change of variable if and only if 
$
\int_{t_0}^{+\infty}  \beta(r)dr = +\infty ,
$
hence our assumption on $\beta$. \qed
\end{proof}

%%%%%%
\begin{lemma}\label{basic-int} 
Take $t_0 >0$, and let $f \in L^1 (t_0 , +\infty;\R)$ be continuous. Consider a nondecreasing function $\varphi:[t_0,+\infty[\to\R_+$ such that $\lim\limits_{t\to+\infty}\varphi(t)=+\infty$. Then, 
$
\lim_{t \to + \infty} \frac{1}{\varphi(t)} \int_{t_0}^t \varphi(s)f(s)ds =0.
$
\end{lemma}
\begin{proof} 
Given $\epsilon >0$, fix $t_\epsilon$ so that $\int_{t_\epsilon}^{\infty}  |f(s)| ds \leq \epsilon$.
Then, for $t \geq t_\epsilon$, split the integral $\DS{\int_{t_0}^t \varphi(s)f(s) ds}$ into two parts to obtain
\begin{eqnarray*}
\abs{\frac{1}{\varphi(t)} \int_{t_0}^t \varphi(s)f(s)ds} =
\abs{\frac{1}{\varphi(t)}\int_{t_0}^{t_\epsilon} \varphi(s) f(s) ds
+ \frac{1}{\varphi(t)}\int_{t_\epsilon}^t \varphi(s) f(s) ds} 
\leq \frac{1}{\varphi(t)}\int_{t_0}^{t_\epsilon} \varphi(s)|f(s)| ds
+ \int_{t_\epsilon}^t  |f(s)| ds.
\end{eqnarray*}
Let $t\to+\infty$ to deduce that
$
0\leq\limsup_{t\to+\infty}\abs{\frac{1}{\varphi(t)}\int_{t_0}^t \varphi(s)f(s)ds} \le \epsilon.
$

Since this is true for any $\epsilon>0$, the result follows.\qed
\end{proof}

%%%%%%
\begin{lemma}[{\cite[Lemma~5.9]{attouch2018fast}}]\label{lem:convw}
Let $t_0 > 0$, and let $w : [t_0,+\infty[ \to \R$ be a twice differentiable \footnote{In \cite[Lemma~5.9]{attouch2018fast}, twice differentiability was not stated, but is actually needed for the statement to make sense.} function which is bounded from below. Assume that
\[
t\ddot{w}(t) + \alpha \dot{w}(t) \leq g(t),
\]
for some $\alpha > 1$, almost every $t > t_0$, and some non-negative function $g \in L^1(t_0,+\infty;\R)$. Then, the positive part $[\dot w]_+$ of $\dot w$ belongs to $L^1(t_0,+\infty;\R)$ and $\lim_{t \to +\infty} w(t)$ exists.
\end{lemma}

%\bibliographystyle{abbrv}
%\bibliography{biblio}

%%%%%%%%%%%%%%%%%%%%%%%%%%%%%%%%%%%%%%%%%%%%%%%%%%%%%%%%%%%%%%%%%%%%%%
%%%%%%%%%%%%%%%%%%%%%%%%%%%%%%%%%%%%%%%%%%%%%%%%%%%%%%%%%%%%%%%%%%%%%%%%%%%%%%%%%%%%%%%%%%%%%%%%%%%%%%%%%%%%%%%
%%%%%%%%%%%%%%%%%%%%%%%%%%%%%%%%%%%%%%%%%%%%%%%%%%%%%%%%%%%%%%%%%%%%%%%%%%%%%%%%%%%%%%%%%%%%%%%%%%%%%%%%%%%%%%%
%%%%%%%%%%%%%%%%%%%%%%%%%%%%%%%%%%%%%%%%%%%%%%%%%%%%%%%%%%%%%%%%%%%%%%%%%%%%%%%%%%%%%%%%%%%%%%%%%%%%%%%%%%%%%%%

\end{document}